\documentclass[10pt, oneside]{article}
\usepackage{mathrsfs}
\usepackage{amsfonts}
\usepackage{amsmath,amsfonts,amssymb,amsthm}
\usepackage{caption}
\usepackage{subfigure}
\usepackage{cite}
\usepackage{color}
\usepackage{graphicx}
\usepackage{subfigure}
\usepackage{float}
\usepackage{paralist}
\usepackage{indentfirst}
\usepackage{authblk}
\usepackage[colorlinks,
linkcolor=red,
citecolor=blue
]{hyperref}

\usepackage[T1]{fontenc}

\newtheorem{theorem}{Theorem}[section]
\newtheorem{lemma}{Lemma}[section]
\newtheorem{prop}{Proposition}[section]

\newtheorem{remark}{Remark}[section]

\newtheorem{defn}{Definition}[section]

\def\bma#1\ema{{\allowdisplaybreaks\begin{split}#1\end{split}}}

\numberwithin{equation}{section}

\begin{document}
\title{{\LARGE \textbf{Global weak solutions for compressible Navier-Stokes-Vlasov-Fokker-Planck system}}}

\author[a,b]{Hai-Liang  Li  \thanks{
E-mail:		hailiang.li.math@gmail.com (H.-L. Li).}}

\author[a,b]{Ling-Yun Shou \thanks{E-mail: shoulingyun11@gmail.com(L.-Y Shou).}}
    \affil[a]{School of Mathematical Sciences,
	Capital Normal University, Beijing 100048, P.R. China}
\affil[b]{Academy for Multidisciplinary Studies, Capital Normal University, Beijing 100048, P.R. China}

\date{}
\renewcommand*{\Affilfont}{\small\it}
\maketitle
\begin{abstract}
The one-dimensional compressible Navier-Stokes-Vlasov-Fokker-Planck system with density dependent viscosity and drag force coefficients is investigated in the present paper. The existence,  uniqueness, and regularity of global weak solution to the initial value problem for general initial data are established in spatial periodic domain. Moreover, the long time behavior of the weak solution is analyzed. It is shown that as the time grows, the distribution function of the particles converges to the global Maxwellian, and both the fluid velocity and the macroscopic velocity of the particles converge to the same speed.
\end{abstract}
\section{Introduction}

Fluid-particle models have a wide range applications such as biosprays in medicine, chemical engineering, compressibility of droplets, fuel-droplets in combustion theory,  pollution settling processes, and polymers to simulate the motion of particles dispersed in dense fluids \cite{am1,berres2,caflisch1,jabin1,lin1,o1,williams2}.

In this paper, we consider the initial value problem (IVP) for the one-dimensional compressible Navier-Stokes-Vlasov-Fokker-Planck (NS-VFP) system 
\begin{equation}\label{m1}
\left\{
\begin{split}
& \rho_{t}+(\rho u)_{x}=0,\\
& (\rho u)_{t}+(\rho u^2)_{x}+(P(\rho))_{x}=(\mu(\rho)u_{x})_{x}-\int_{\mathbb{R}} \kappa(\rho)(u-v)fdv,\\
&f_{t}+vf_{x}+(\kappa(\rho)(u-v)f-\kappa(\rho)f_{v})_{v}=0, \quad (x,v)\in \mathbb{T}\times\mathbb{R},~~ t>0,~~~\mathbb{T}:=\mathbb{R}/\mathbb{Z},
\end{split}
\right.
\end{equation}
 with the initial data
 \begin{equation}
\begin{split}
&(\rho(x,0),\rho u(x,0),f(x,v,0))=(\rho_{0}(x),m_{0}(x),f_{0}(x,v)),\quad (x,v)\in \mathbb{T}\times\mathbb{R},\label{d}
\end{split}
\end{equation}
where $\rho=\rho(x,t)$ and $u=u(x,t)$ are the fluid density and velocity associated with the dense phase (fluid) respectively, and $f=f(x,v,t)$ denotes the distribution function associated with the dispersed phase (particles). The system $(\ref{m1})$ can be viewed as the compressible Navier-Stokes equations $(\ref{m1})_{1}$-$(\ref{m1})_{2}$ for the fluid and the Vlasov-Fokker-Planck equation $(\ref{m1})_{3}$ for the particles coupled each other through the drag force term $\kappa(\rho)(u-v)$. The pressure $P(\rho)$ and the viscosity coefficient $\mu(\rho)$ are given by
\begin{equation}\label{P}
\begin{split}
P(\rho)=A\rho^{\gamma},\quad\quad \mu(\rho)=\mu_{0}+\mu_{1}\rho^{\beta},
\end{split}
\end{equation}
and the drag force coefficient $\kappa(\rho)$ is chosen to be
\begin{equation}\label{kappa}
\begin{split}
\kappa(\rho)=\kappa_{0}\rho,
\end{split}
\end{equation}
where the constants $A$, $\mu_{0}$, $\mu_{1}$, $\kappa_{0}$, $\gamma$, and $\beta$ satisfy
$$
 A>0,\quad \mu_{0}>0,\quad \mu_{1}>0,\quad \kappa_{0}>0,\quad \gamma>1,\quad \beta\geq 0.
$$
Without loss of generality, we take $A=\mu_{0}=\mu_{1}=\kappa_{0}=1$ in the present paper.

There are a lot of important progress on the analysis of the global existence and dynamical behaviors of solutions for fluid-particle systems \cite{chae1,chae2,boudin3,choi1,choi3,goudon1,hamdache1,han2,hbk1,he1,jiang1,lf1,lhl1,lhl2,lhl3,lhl4,mellet1}. Among them, for incompressible NS-VFP equations, He \cite{he1} and Goudon-He-Moussa-Zhang \cite{goudon1} proved the global regularity and exponential decay rate of classical solutions in spatial periodic domain, and Chae-Kang-Lee \cite{chae1} showed the global existence of weak solutions in spatial whole space. For compressible NS-VFP system with constant drag force coefficient, the global existence of weak solutions to three-dimensional initial boundary value problem with the adiabatic constant $\gamma>\frac{3}{2}$ was obtained by Mellet-Vasseur \cite{mellet1}, the global well-posedness of strong solutions to Cauchy problem was established either for two-dimensional large initial data in \cite{hbk1} and for three-dimensional small initial data in \cite{chae2,jiang1} respectively, and the nonlinear time-asymptotical stability of planar rarefaction wave was investigated by Li-Wang-Wang \cite{lhl3}. As for compressible NS-VFP system with density-dependent drag force coefficient, Li-Mu-Wang \cite{lf1} got the global existence and time-decay estimates of strong solutions around the equilibrium state in both spatial periodic domain and spatial whole space, and Li-Sun-Zhang-Zhong \cite{lhl4} analyzed the Green function and pointwise behaviors of strong solutions to three-dimensional Cauchy problem. In addition, the compressible NS-VFP equations can be approximated by some macroscopical two-phase models as the fluid-dynamical limits, the readers can refer to \cite{carrillo1,choi4,mellet2,lhl3} and references therein.

In the present paper, we study the existence, uniqueness, regularity, and large time behavior of global weak solution to the IVP $(\ref{m1})$-(\ref{kappa}) for general initial data, as a continuation of the previous works \cite{lhl1,lhl2}. 
~\

First, we give the definition of global weak solutions to the IVP $(\ref{m1})$-$(\ref{kappa})$ as follows:

\begin{defn}\label{defn11}
$(\rho,u, f)$ is said to be a global weak solution to the IVP $(\ref{m1})$-$(\ref{kappa})$ provided for any $T>0$ that
\begin{equation}\label{r0}
\left\{
\begin{split}
&\rho\in L^{\infty}(0,T;L^{\gamma}(\mathbb{T})),\quad \sqrt{\rho} u\in L^{\infty}(0,T;L^2(\mathbb{T})),\quad \sqrt{\mu(\rho)}u_{x}\in L^2(0,T;L^2(\mathbb{T})),\\
&f\log{f} \in L^{\infty}(0,T;L^1(\mathbb{T}\times\mathbb{R})),\quad |v|^2f\in L^{\infty}(0,T;L^1(\mathbb{T}\times\mathbb{R})),
\end{split}
\right.
\end{equation}
 the equations $(\ref{m1})_{1}$-$(\ref{m1})_{3}$ are satisfied in the sense of distributions, and the following entropy inequality holds for a.e. $t\in[0,T]$:
\begin{equation}\label{entropyinequality}
\begin{split}
&\int_{\mathbb{T}}\big{(}\frac{1}{2}\rho |u|^2+\frac{\rho^{\gamma}}{\gamma-1}\big{)}(x,t)dx+\int_{\mathbb{T}\times\mathbb{R}}\big{(}\frac{1}{2}f|v|^2+f\log{f}\big{)}(x,v,t)dvdx+\int_{0}^{t}\int_{\mathbb{T}}(\mu(\rho)|u_{x}|^2)(x,\tau)dxd\tau\\
&\quad\leq \int_{\mathbb{T}}\big{(}\frac{1}{2}\frac{|m_{0}|^2}{\rho_{0}} +\frac{\rho_{0}^{\gamma}}{\gamma-1}\big{)}(x)dx+\int_{\mathbb{T}\times\mathbb{R}}\big{(}\frac{1}{2}f_{0}|v|^2+f_{0}\log{f_{0}}\big{)}(x,v)dvdx.
\end{split}
\end{equation}
\end{defn}

\vspace{2ex}
  Denote by $n$ and $nw$ the macroscopical density and momentum related to the moments of the solution $f$ to $(\ref{m1})_{3}$ as
\begin{equation}\label{nw}
\begin{split}
n(x,t)=\int_{\mathbb{R}} f(x,v,t)dv,\quad nw(x,t):=\int_{\mathbb{R}}vf(x,v,t)dv,
\end{split}
\end{equation}
with the initial values
\begin{equation}\nonumber
\begin{split}
n(x,0)=n_{0}(x):=\int_{\mathbb{R}} f_{0}(x,v)dv,\quad nw(x,0)=j_{0}(x):=\int_{\mathbb{R}}vf_{0}(x,v)dv.
\end{split}
\end{equation}
For $p\in[1,\infty]$, $k,l\geq 0$, and $\langle v\rangle:=(1+|v|^2)^{\frac{1}{2}}$, define
 \begin{equation}\nonumber
\left\{
\begin{split}
&L^{p}_{l}(\mathbb{T}\times\mathbb{R}):=\{\langle v\rangle^{l}g\in L^{p}(\mathbb{T}\times\mathbb{R})~\big{|}~ \|g\|_{L^{p}_{l}(\mathbb{T}\times\mathbb{R})}=\|\langle v\rangle^{l}g\|_{L^p(\mathbb{T}\times\mathbb{R})}<\infty\},\\
&H^{k}_{l}(\mathbb{T}\times\mathbb{R}):=\{\langle v\rangle^{l}g\in H^{k}(\mathbb{T}\times\mathbb{R})~\big{|}~\|g\|_{H^{k}_{l}(\mathbb{T}\times\mathbb{R})}= \|\langle v\rangle^{l}g\|_{H^{k}(\mathbb{T}\times\mathbb{R})}<\infty\}.
\end{split}
\right.
\end{equation}
We have the following global existence of weak solutions to the IVP $(\ref{m1})$-$(\ref{kappa})$ for general initial data which allows vacuum and can be discontinuous:
\begin{theorem}\label{theorem11}
Suppose that the initial data $(\rho_{0},m_{0},f_{0})$ satisfies
\begin{equation}
\begin{split}
&0\leq \rho_{0}\in L^{\infty}(\mathbb{T}),~~\frac{m_{0}}{\sqrt{\rho_{0}}}=0~~\text{in}~\{x\in\mathbb{T}~|~\rho_{0}(x)=0\},~~ \frac{m_{0}}{\sqrt{\rho_{0}}}\in L^2(\mathbb{T}),~~ 0\leq f_{0}\in L^1_{3}\cap L^{\infty}(\mathbb{T}\times\mathbb{R}).\label{a1}
\end{split}
\end{equation}
Then the IVP $(\ref{m1})$-$(\ref{kappa})$ admits a global weak solution $(\rho,u,f)$ in the sense of Definition \ref{defn11} satisfying for any $T>0$ that
\begin{equation}\label{r1}
\left\{
\begin{split}
&0\leq\rho(x,t)\leq \rho_{+},\quad\text{a.e.}~(x,t)\in\mathbb{T}\times[0,T],\\
& 0\leq f(x,v,t)\leq e^{\rho_{+}T}\|f_{0}\|_{L^{\infty}(\mathbb{T}\times\mathbb{R})},\quad\text{a.e.}~(x,v,t)\in\mathbb{T}\times\mathbb{R}\times[0,T],\\
&\underset{t\in[0, T]}{{\rm{ess~sup}}}~\Big{(} t^{\frac{1}{2}}\|u(t)\|_{H^1(\mathbb{T})}+\| f(t)\|_{L^1_{3}(\mathbb{T}\times\mathbb{R})}\Big{)}+\|\sqrt{\rho}f_{v}\|_{L^2(0,T;L^2(\mathbb{T}\times\mathbb{R}))}\leq C_{T},  
\end{split}
\right.
\end{equation}
where $\rho_{+}>0$ is a constant independent of the time $T>0$, and $C_{T}>0$ is a constant dependent of the time $T>0$.

Moreover, the conservation laws of mass and momentum hold for $t\in[0,T]${\rm{:}}
\begin{equation}\label{massmomentum}
\left\{
\begin{split}
&\int_{\mathbb{T}}\rho(x,t)=\int_{\mathbb{T}}\rho_{0}(x)dx,\\
& \int_{\mathbb{T}\times\mathbb{R}}f(x,v,t)dvdx=\int_{\mathbb{T}}n(x,t)dx=\int_{\mathbb{T}}n_{0}(x)dx,\\
&\int_{\mathbb{T}}(\rho u+nw)(x,t)dx=\int_{\mathbb{T}}(m_{0}+j_{0})(x)dx,
\end{split}
\right.
\end{equation}
and $(\rho,u)$ converges to $(\overline{\rho_{0}}, \overline{u}(t))$ as $t\rightarrow\infty${\rm{:}}
\begin{equation}
\begin{split}
&\lim_{t\rightarrow\infty}\Big{(}\|(\rho-\overline{\rho_{0}})(t)\|_{L^{p}(\mathbb{T})}+\|\big{(}\sqrt{\rho}(u-\overline{u})\big{)}(t)\|_{L^2(\mathbb{T})}\Big{)}=0,\quad  p\in [1,\infty),\label{b1}
\end{split}
\end{equation}
where the asymptotical states $\overline{\rho_{0}}$ and $\overline{u}(t)$ are given by
\begin{equation}\label{overrho}
\begin{split}
\overline{\rho_{0}}:=\int_{\mathbb{T}}\rho_{0}(x)dx,\quad \overline{u}(t):=\frac{\int_{\mathbb{T}}\rho u(x,t)dx}{\int_{\mathbb{T}}\rho_{0}(x)dx}.
\end{split}
\end{equation}
\end{theorem}

~\

Then, we study the regularity and uniqueness of the global weak solution $(\rho, u, f)$ given by Theorem \ref{theorem11}, and analyze the large time behavior of $(\rho, u, f)$ to the equilibrium state 
\begin{equation}\nonumber
\begin{split}
(\overline{\rho_{0}},u_{c},M_{\overline{n_{0}},u_{c}}(v)),
\end{split}
\end{equation}
 where $u_{c}$ is the constant
\begin{equation}\label{rhoinfty}
\begin{split}
& u_{c}:=\frac{\int_{\mathbb{T}}(m_{0}+j_{0})(x)dx}{\int_{\mathbb{T}}(\rho_{0}+n_{0})(x)dx},
\end{split}
\end{equation}
 $M_{\overline{n_{0}},u_{c}}(v)$ is the global Maxwellian
\begin{equation}\label{M}
\begin{split}
M_{\overline{n_{0}},u_{c}}(v):=\frac{\overline{n_{0}}}{\sqrt{2\pi}}e^{-\frac{|v-u_{c}|^2}{2}},
\end{split}
\end{equation}
and $\overline{n_{0}}$ is the constant 
\begin{equation}\label{n0}
\begin{split}
\overline{n_{0}}:=\int_{\mathbb{T}}n_{0}(x)dx.
\end{split}
\end{equation}

\begin{theorem}\label{theorem12}
Suppose that the initial data $(\rho_{0},m_{0},f_{0})$ satisfies
\begin{equation}
\begin{split}
&\inf_{x\in\mathbb{T}}\rho_{0}(x)>0, \quad \rho_{0}\in H^1(\mathbb{T}),\quad \frac{m_{0}}{\rho_{0}}\in H^1(\mathbb{T}),\quad 0\leq f_{0}\in L^{2}_{k_{0}}\cap L^{\infty}(\mathbb{T}\times\mathbb{R}), \label{a2}
\end{split}
\end{equation}
where $k_{0}>\frac{7}{2}$ is a constant. Then the global weak solution $(\rho,u,f)$ to the IVP $(\ref{m1})$-$(\ref{kappa})$ given by Theorem \ref{theorem11} is unique. In addition to the properties $(\ref{r1})$-$(\ref{b1})$, it holds for any $T>0$ that
 \begin{equation}\label{r2}
\left\{
\begin{split}
&~~\rho(x,t)\geq \rho_{-}>0,\quad\text{a.e.}~ (x,t)\in\mathbb{T}\times[0,T],\quad \sup_{t\in[0,T]}\|\rho(t)\|_{H^1(\mathbb{T})}\leq C_{0},\\
&~~\|u\|_{L^2(0,T;H^2(\mathbb{T})}+\|u_{t}\|_{L^2(0,T;L^2(\mathbb{T}))}+\|f_{v}\|_{L^2(0,T;L^2_{k_{0}}(\mathbb{T}\times\mathbb{R}))}\leq C_{T},\\
&\underset{t\in[0, T]}{{\rm{ess~sup}}}~\Big{(}\|u(t)\|_{H^1(\mathbb{T})}+t^{\frac{1}{2}}\|u(t)\|_{H^2(\mathbb{T})}+\|f(t)\|_{L^2_{k_{0}}(\mathbb{T}\times\mathbb{R})}\Big{)}\leq C_{T},\\
 &\underset{t\in[0, T]}{{\rm{ess~sup}}}~\Big{(} t^{\frac{1}{2}}\|f_{v}(t)\|_{L^2_{k_{0}-2}(\mathbb{T}\times\mathbb{R})}+t\|f_{vv}(t)\|_{L^2_{k_{0}-2}(\mathbb{T}\times\mathbb{R})}+t^{\frac{3}{2}}\|(f_{x},f_{vvv})(t)\|_{L^2_{k_{0}-2}(\mathbb{T}\times\mathbb{R})}\Big{)}\leq C_{T},
\end{split}
\right.
\end{equation}
 where $\rho_{-}>0$ and $C_{0}>0$ are two constants independent of the time $T>0$, and $C_{T}>0$ is a constant dependent of the time $T>0$.
 
Furthermore, the solution $(\rho,u,f)$ satisfies
\begin{equation}\label{b2}
\left\{
\begin{split}
&\lim_{t\rightarrow\infty}\Big{(}\|(\rho-\overline{\rho_{0}})(t)\|_{L^{\infty}(\mathbb{T})}+\|(u-u_{c})(t)\|_{L^2(\mathbb{T})} \Big{)}=0,\\
&\lim_{t\rightarrow\infty} \Big{(}\|(f-M_{\overline{n_{0}},u_{c}})(t)\|_{L^1(\mathbb{T}\times\mathbb{R})}+\|(n-\overline{n_{0}})(t)\|_{L^1(\mathbb{T})}+\|\big{(}n(w-u_{c})\big{)}(t)\|_{L^1(\mathbb{T})} \Big{)}=0,\\
\end{split}
\right.
\end{equation}
\end{theorem}

\vspace{1ex}

\begin{remark}\label{remark11}
Different from the compressible Navier-Stokes-Vlasov system {\rm{{\rm{\cite{lhl2}}}}}, the Vlasov-Fokker-Planck equation $(\ref{m1})_{3}$ has the regularizing effect $(\ref{r2})_{4}$ due to hypoellipticity of the nonlinear Fokker-Planck operator $v\partial_{x}-\rho \partial_{v}(v+\partial_{v})$. 
\end{remark}

\begin{remark}\label{remark12}
The distribution function $f$ lacks the uniform-time integrability under the assumptions of Theorem \ref{theorem12} since the term $f\log{f}$ in entropy for $(\ref{m1})$ may not be controlled by its dissipation, which is essentially different from the compressible Navier-Stokes-Vlasov model {\rm{\cite{lhl1}}}. To have the large time behavior $(\ref{b2})_{2}$, we employ the ideas by Bouchut-Dolbeault {\rm{\cite{bouchut1}}} to prove the $L^1(\mathbb{T}\times\mathbb{R})$-convergence of $f^{s}(x,v,t):=f(x,v,t+s)$ for $t\in(0,1)$ to the unique limit $M_{\overline{n_{0}},u_{c}}(v)$ as $s\rightarrow\infty$. It should be noted that the compactness lemmas in {\rm{\cite{bouchut1,diperna1}}} could not be applied here due to the non-smooth coefficients in $(\ref{m1})_{3}$. Indeed, because $(\sqrt{f^{s}})_{v}$ is uniformly bounded in $L^2(0,1;L^2(\mathbb{T}\times\mathbb{R}))$, we can apply the techniques developed by Ars${\acute{e}}$nio and Saint-Raymond {\rm{\cite{arsenio1}}} to obtain the strong convergence of $f^{s}$ in all variables.
\end{remark}

\begin{remark}\label{remark13}
By $(\ref{r2})$ and Theorem \ref{theorem11}, one can prove that the IVP $(\ref{m1})$-$(\ref{kappa})$ admits a unique global classical solution subject to regular initial data.
\end{remark}

\begin{remark}\label{remark14}
By Theorem \ref{theorem12} and similar arguments as used in {\rm{\cite{lhl2}}}, we are able to establish the global well-posedness of the IVP $(\ref{m1})$-$(\ref{kappa})$ in spatial real line.
\end{remark}

\begin{remark}\label{remark10}
Under the assumptions of Theorem \ref{theorem11}, if the initial density $\rho_{0}$ is strictly positive, then we can obtain the low bound of $\rho$ uniformly in time. In addition, as in Sections 4-5, the global weak solution $(\rho,u,f)$ satisfies the large time behavior $(\ref{b2})$ and the following entropy inequality:
\begin{equation}\label{entropyinequality0}
\begin{split}
&\int_{\mathbb{T}}\big{(}\frac{1}{2}\rho |u|^2+\frac{\rho^{\gamma}}{\gamma-1}\big{)}(x,t)dx+\int_{\mathbb{T}\times\mathbb{R}}\big{(}\frac{1}{2}f|v|^2+f\log{f}\big{)}(x,v,t)dvdx\\
&\quad\quad+\int_{0}^{t}\int_{\mathbb{T}}(\mu(\rho)|u_{x}|^2)(x,\tau)dxd\tau+\int_{0}^{t}\int_{\mathbb{T}\times\mathbb{R}}(\rho|(u-v)\sqrt{f}-2(\sqrt{f})_{v}|^2)(x,v,\tau)dvdxd\tau\\
&\quad\leq \int_{\mathbb{T}}\big{(}\frac{1}{2}\frac{|m_{0}|^2}{\rho_{0}} +\frac{\rho_{0}^{\gamma}}{\gamma-1}\big{)}(x)dx+\int_{\mathbb{T}\times\mathbb{R}}\big{(}\frac{1}{2}f_{0}|v|^2+f_{0}\log{f_{0}}\big{)}(x,v)dvdx.
\end{split}
\end{equation}
\end{remark}

\vspace{2mm}

The rest part of the paper is arranged as follows. In Section 2, we establish the a-priori estimates for the compressible NS-VFP system (\ref{m1}). The uniqueness of the weak solution will be proved in Section 3. Then, we  we analyze the large time behavior of global solutions in Section 4. In Section 5, we show the convergence of approximate sequence to the corresponding weak solution to the IVP (\ref{m1})-(\ref{kappa}). Finally, Section 6 is an appendix presenting some lemmas that are used in our analysis.

\section{The a-priori estimates}

\subsection{Basic estimates}

First,  by (\ref{m1}), we have the following properties:

\begin{lemma}\label{lemma21}
Let $T>0$, and $(\rho,u,f)$ be any regular solution to the IVP $(\ref{m1})$-$(\ref{kappa})$ for $t\in(0,T]$. Then it holds
\begin{align}
&\frac{d}{dt}\int_{\mathbb{T}}\rho(x,t) dx=0,\quad \frac{d}{dt}\int_{\mathbb{T}\times\mathbb{R}} f(x,v,t)dvdx=\frac{d}{dt}\int _{\mathbb{T}}n(x,t)dx=0,\label{mass}\\
&\frac{d}{dt}\int_{\mathbb{T}}(\rho u+nw)(x,t) dx=0,\label{momentum}\\
&\frac{d}{dt}\Big{(}\int_{\mathbb{T}}\big{(}\frac{1}{2}\rho |u|^2+\frac{\rho^{\gamma}}{\gamma-1}\big{)}(x,t)dx+\int_{\mathbb{T}\times\mathbb{R}} \big{(}\frac{1}{2}|v|^2f+f\log{f}\big{)}(x,v,t)dvdx\Big{)}\nonumber\\
&\quad\quad=-\int_{\mathbb{T}}(\mu(\rho)|u_{x}|^2)(x,t)dx-\int_{\mathbb{T}\times\mathbb{R}}\big{(}\rho|(u-v)\sqrt{f}-2(\sqrt{f})_{v}|^2\big{)}(x,v,t)dvdx,
\label{energyCNSVFP}
\end{align}
where $n$ and $nw$ are defined by $(\ref{nw})$.
\end{lemma}

Then, we have

\begin{lemma}\label{lemma22}
Let $T>0$, and $(\rho,u,f)$ be any regular solution to the IVP $(\ref{m1})$-$(\ref{kappa})$ for $t\in(0,T]$. Then, under the assumptions of Theorem \ref{theorem11}, it holds
\begin{equation}\label{basicCNSVFP}
\left\{
\begin{split}
&~~\rho(x,t)\geq0,\quad  f(x,v,t)\geq 0,\quad (x,v,t)\in \mathbb{T}\times\mathbb{R}\times [0,T],\\
&\sup_{t\in[0,T]}\big{(}\|(\sqrt{\rho}u)(t)\|_{L^2(\mathbb{T})}+\|f(t)\|_{L^1_{2}(\mathbb{T}\times\mathbb{R})}+\|(f\log{f})(t)\|_{L^1(\mathbb{T}\times\mathbb{R})}\big{)}\leq C,\\
&\|\sqrt{\mu(\rho)}u_{x}\|_{L^2(0,T;L^2(\mathbb{T}))}+\|\sqrt{\rho}\big{(}(u-v)\sqrt{f}-2(\sqrt{f})_{v}\big{)}\|_{L^2(0,T;L^2(\mathbb{T}\times\mathbb{R}))}\leq C,\\
&~~\|u(t)\|_{L^{\infty}(\mathbb{T})}\leq \|u_{x}(t)\|_{L^2(\mathbb{T})}+C,\quad t\in[0,T],\\
&~~ \rho(x,t)\leq \rho_{+},\quad (x,t)\in\mathbb{T}\times [0,T],\\
\end{split}
\right.
\end{equation}
where $C>0$ and $\rho_{+}$ are two constants independent of the time $T>0$.
\end{lemma}
\begin{proof}
By standard maximum principle for the transport equation $(\ref{m1})_{1}$, we get $\rho\geq 0$. Multiplying $(\ref{m1})_{3}$ by $f_{-}:=-\min{\{f,0\}}\leq 0$ and integrating the resulting equation by parts over $\mathbb{T}\times\mathbb{R}$, one obtains
$$
\frac{1}{2}\frac{d}{dt}\|f_{-}(t)\|_{L^2(\mathbb{T}\times\mathbb{R})}^2+\|(\sqrt{\rho}f_{-})_{v}(t)\|_{L^2(\mathbb{T}\times\mathbb{R})}^2\leq\frac{\|\rho(t)\|_{L^{\infty}(\mathbb{T})}}{2}\|f_{-}(t)\|_{L^2(\mathbb{T}\times\mathbb{R})}^2,
$$
which together with the Gr${\rm{\ddot{o}}}$nwall inequality and $f_{-}|_{t=0}=0$ implies $f\geq 0$. $(\ref{basicCNSVFP})_{2}$-$(\ref{basicCNSVFP})_{4}$ can be derived by (\ref{mass}), (\ref{energyCNSVFP}), $(\ref{basicCNSVFP})_{1}$, and the fact $f|\log f|\leq f\log{f}+\frac{1}{2}|v|^2 f+Ce^{-\frac{|v|^2}{2}}$.

Then, similarly to \cite{lhl4},  we introduce the effective velocity $u+\mathcal{I}(n)$ to re-write the momentum equation $(\ref{m1})_{2}$ as
\begin{equation}\label{newm1}
\begin{split}
&\big{(}\rho (u+\mathcal{I}(n))\big{)}_{t}+\big{(}\rho u(u+\mathcal{I}(n))\big{)}_{x}+(\rho^{\gamma})_{x}=(\mu(\rho)u_{x})_{x}+\rho \int_{0}^{1}nw(y,t)dy,
\end{split}
\end{equation}
  where the operator $\mathcal{I}:L^1(\mathbb{T}) \rightarrow L^{\infty}(\mathbb{T}) $ is defined as
  \begin{equation}\label{j}
\begin{split}
\mathcal{I}(g)(x):=\int^{x}_{0}g(y)dy-\int_{0}^{1}\int_{0}^{y}g(z) dzdy,\quad \forall g\in L^1(\mathbb{T}).
\end{split}
\end{equation}
Applying the operator $\mathcal{I}$ to (\ref{newm1}), one can obtain after a direct computation that 
\begin{equation}
\begin{split}
&\big{[}\mathcal{I}\big{(}\rho u+\rho\mathcal{I}(n)\big{)}\big{]}_{t}+ u\big{[}\mathcal{I}\big{(}\rho u+\rho\mathcal{I}(n)\big{)}\big{]}_{x}+\rho^{\gamma}-\mu(\rho)u_{x}\\
&\quad=\mathcal{I}(\rho)\int_{0}^{1}nw(y,t)dy+ \int_{0}^{1}\big{[}\rho |u|^2+\rho u\mathcal{I}(n)+\rho^{\gamma}-\mu(\rho)u_{y}\big{]}(y,t)dy.\label{newm2}
\end{split}
\end{equation}
 In addition,  it holds by the equation $(\ref{m1})_{1}$ that
\begin{equation}\label{massrhox}
\begin{split}
\frac{d}{dt}\theta(\rho)+u\theta(\rho)_{x}=-\mu(\rho)u_{x},
\end{split}
\end{equation}
where $\theta(\rho)$ is denoted by
\begin{eqnarray}\label{213}
\theta(\rho):=\int_{1}^{\rho}\frac{\mu(s)}{s}ds=
\begin{cases}
2\log{\rho},
& \mbox{if $\beta=0,$ } \\
\log{\rho}+\frac{\rho^{\beta}-1}{\beta},
& \mbox{if $\beta>0.$}
\end{cases}
\end{eqnarray}
Substituting (\ref{massrhox}) into (\ref{newm2}) and re-writing the resulting equation along the particle path $\mathcal{X}^{x,t}(s)$ for any $(x,t)\in \mathbb{T}\times [0,T]$ defined through
\begin{equation}\label{overlinex}
\left\{ \begin{split}
&\frac{d}{ds}\mathcal{X}^{x,t}(s)=u(\mathcal{X}^{x,t}(s),x),\quad  s\in [0,t],\\
&\quad \mathcal{X}^{x,t}(t)=x,
\end{split}
\right.
\end{equation}
we have
 \begin{equation}\label{newm4}
 \begin{split}
&\frac{d}{ds}\theta(\mathcal{X}^{x,t}(s),s)=-\rho^{\gamma}(\mathcal{X}^{x,t}(s),s)+\frac{d}{ds}F(\mathcal{X}^{x,t}(s),s),
\end{split}
\end{equation}
where $F(x,t)$ is given by
\begin{equation}\label{Fxt}
\begin{split}
&F(x,t):=-\mathcal{I}\big{(}\rho u+\rho\mathcal{I}(n)\big{)}(x,t)+\int_{0}^{t}\Big{(}\mathcal{I}(\rho)(x,\tau)\int_{0}^{1}nw(y,\tau)dy\\
&\quad\quad\quad\quad+ \int_{0}^{1}\big{[}\rho |u|^2+\rho u\mathcal{I}(n)+\rho^{\gamma}\big{]}(y,\tau)dy+\int_{0}^{1}(\mu(\rho)u_{y})(y,\tau)dy\Big{)}d\tau.
\end{split}
\end{equation}
To employ Lemma \ref{lemma61} below, we need to estimate $F(\mathcal{X}^{x,t}(s),s)$. It is easy to verify
\begin{equation}\label{FF}
\left\{
\begin{split}
&\big{|}\mathcal{I}\big{(}\rho u+\rho\mathcal{I}(n)\big{)}(x,t)\big{|}\leq \|\rho(t)\|_{L^1(\mathbb{T})}^{\frac{1}{2}}\|(\sqrt{\rho}u)(t)\|_{L^2(\mathbb{T})}+\|\rho(t)\|_{L^1(\mathbb{T})}\|n(t)\|_{L^1(\mathbb{T})},\\
&\big{|}\mathcal{I}(\rho)(x,\tau)\int_{0}^{1}nw(y,t)dy\big{|}\leq\|\rho(t)\|_{L^1(\mathbb{T})} \|f(t)\|_{L_{2}^1(\mathbb{T}\times\mathbb{R})},\\
&\big{|}\int_{0}^{1}\big{[}\rho |u|^2+\rho u\mathcal{I}(n)+\rho^{\gamma}\big{]}(y,t)dy\big{|}\\
&\quad\leq  \|(\sqrt{\rho}u)(t)\|_{L^2(\mathbb{T})}^2+\|n(t)\|_{L^1(\mathbb{T})}\|\rho(t)\|_{L^1(\mathbb{T})}^{\frac{1}{2}}\|(\sqrt{\rho} u)(t)\|_{L^2(\mathbb{T})}+\|\rho^{\gamma}(t)\|_{L^{1}(\mathbb{T})}.
\end{split}
\right.
\end{equation}
Since it holds
\begin{equation}\nonumber
 \left\{
 \begin{split}
 &\mu(\rho)\leq 2,\quad\text{if}~0\leq \rho\leq 1,\\
 &\mu(\rho)\leq \beta\theta(\rho)+2,\quad\text{if}~\rho\geq 1,
 \end{split}
 \right.
 \end{equation}
 we also have
  \begin{align}
 &\Big{|}\int_{0}^{1}(\mu(\rho)u_{y})(y,t)dy\Big{|}\nonumber\\
 &\quad\leq \Big{(}\int_{0}^{1}\mu(\rho)(y,t)dy\Big{)}^{\frac{1}{2}}\Big{(}\int_{0}^{1}(\mu(\rho)|u_{y}|^2)(y,t)dy\Big{)}^{\frac{1}{2}}\nonumber\\
&\quad=1+\frac{1}{4} \Big{(}\int_{\{y\in(0,1)|\rho(y,t)\leq 1\}}+\int_{\{y\in(0,1)|\rho(y,t)\geq 1\}}\Big{)}\mu(\rho)(y,t)dy\int_{0}^{1}(\mu(\rho)|u_{y}|^2)(y,t)dy\nonumber\\
&\quad\leq 1+\int_{0}^{1}(\mu(\rho)|u_{y}|^2)(y,t)dy+\frac{\beta}{4}\sup_{x\in\{y\in(0,1)|\rho(y,t)\geq 1\}}\theta(\rho)(x,t)\int_{0}^{1}(\mu(\rho)|u_{y}|^2)(y,t)dy.\label{ux11}
  \end{align}
By virtue of (\ref{mass}), (\ref{basicCNSVFP}) and (\ref{Fxt})-(\ref{ux11}), we deduce for any $0\leq t_{1}\leq t_{2}\leq T$ that
\begin{equation}\label{Flip}
\begin{split}
|F(\mathcal{X}^{x,t}(t_{2}),t_{2})-F(\mathcal{X}^{x,t}(t_{1}),t_{1})|\leq C_{2}+\frac{\beta C_{1}^2}{4} \int_{0}^{T}\sup_{x\in\{y\in(0,1)|\rho(y,t)\geq 1\}}\theta(\rho)(x,t)dt+C_{3}(t_{2}-t_{1}),
\end{split}
\end{equation}
where $C_{i}>0$, i=2,3, are two constants given by
\begin{equation}\nonumber
\left\{
\begin{split}
&C_{2}:=2\|\rho_{0}\|_{L^1(\mathbb{T})}^{\frac{1}{2}}C_{1}+2\|\rho_{0}\|_{L^1(\mathbb{T})}\|n_{0}\|_{L^1(\mathbb{T})}+C_{1},\\
&C_{3}:=\|\rho_{0}\|_{L^1(\mathbb{T})}C_{1}+C_{1}^2+\|n_{0}\|_{L^1(\mathbb{T})}\|\rho_{0}\|_{L^1(\mathbb{T})}^{\frac{1}{2}}C_{1}+C_{1}+1.
\end{split}
\right.
\end{equation}
Thus, we apply  Lemma \ref{lemma61} for $N_{0}=C_{2}+\frac{\beta C_{1}^2}{4} \int_{0}^{T}\sup_{x\in\{y\in(0,1)|\rho(y,t)\geq 1\}}\theta(\rho)(x,t)dt, N_{1}=C_{3}$ and $\xi_{*}=\theta(C_{3}^{\frac{1}{\gamma}})$ to have
\begin{equation}\nonumber
\begin{split}
&\sup_{x\in\{y\in(0,1)|\rho(y,t)\geq 1\}}\theta(\rho)(x,t)\leq \max\Big{\{}\sup_{x\in\mathbb{T}}\theta(\rho_{0}), \theta(C_{3}^{\frac{1}{\gamma}})\Big{\}}+C_{2}+\frac{\beta C_{1}^2}{4} \int_{0}^{T}\sup_{x\in\{y\in(0,1)|\rho(y,t)\geq 1\}}\theta(\rho)(x,t)dt,
\end{split}
\end{equation}
which together with the Gr${\rm{\ddot{o}}}$nwall inequality and (\ref{213}) yields $(\ref{basicCNSVFP})_{5}$. The proof of Lemma \ref{lemma22} is completed.
\end{proof}

Next, we have the time-dependent estimates of the distribution function $f$:

\begin{lemma}\label{lemma23}
Let $T>0$, and $(\rho,u,f)$ be any regular solution to the IVP $(\ref{m1})$-$(\ref{kappa})$ for $t\in(0,T]$. Then, under the assumptions of Theorem \ref{theorem11}, it holds
\begin{equation}
\left\{
\begin{split}
&~~f(x,v,t)\leq e^{\rho_{+}T}\|f_{0}\|_{L^{\infty}(\mathbb{T}\times\mathbb{R})},\quad (x,v,t)\in\mathbb{T}\times\mathbb{R}\times[0,T],\\
&\sup_{t\in[0,T]}\|f(t)\|_{L^1_{3}(\mathbb{T}\times\mathbb{R})}+\|\sqrt{\rho}f_{v}\|_{L^2(0,T;L^2(\mathbb{T}\times\mathbb{R}))}\leq C_{T},\\
&\sup_{t\in[0,T]}\big{(}\|n(t)\|_{L^{4}(\mathbb{T})}+\|nw(t)\|_{L^2(\mathbb{T})}\big{)}\leq C_{T},\label{N1time}
\end{split}
\right.
\end{equation}
where the constant $\rho_{+}$ is given by $(\ref{basicCNSVFP})_{5}$, and $C_{T}>0$ is a constant dependent of the time $T>0$.
\end{lemma}
\begin{proof}
First, we multiply $(\ref{m1})_{3}$ by $pf^{p-1}$ for any $p\in [2,\infty)$ and integrate the resulting equation by parts over $\mathbb{T}\times\mathbb{R}$ to obtain
\begin{equation}\nonumber
\begin{split}
&\frac{d}{dt}\|f(t)\|_{L^{p}(\mathbb{T}\times\mathbb{R})}^{p}+p(p-1)\| (\sqrt{\rho} f^{\frac{p-2}{2}}  f_{v})(t)\|_{L^2(\mathbb{T}\times\mathbb{R})}^2\leq (p-1)\|\rho(t)\|_{L^{\infty}(\mathbb{T})}\|f(t)\|_{L^{p}(\mathbb{T}\times\mathbb{R})}^{p},
\end{split}
\end{equation}
which together with $(\ref{basicCNSVFP})_{5}$ and the Gr${\rm{\ddot{o}}}$nwall inequality implies
\begin{equation}\label{flp}
\left\{
\begin{split}
&\|\sqrt{\rho}f_{v}\|_{L^2(0,T;L^2(\mathbb{T}\times\mathbb{R}))}\leq C_{T},\\
& \sup_{t\in [0,T]}\|f(t)\|_{L^{p}(\mathbb{T}\times\mathbb{R})}\leq e^{\rho_{+}T}\|f_{0}\|_{L^{\infty}(\mathbb{T}\times\mathbb{R})}^{1-\frac{1}{p}}\|f_{0}\|_{L^1(\mathbb{T}\times\mathbb{R})}^{\frac{1}{p}}.
\end{split}
\right.
\end{equation}
We get $(\ref{N1time})_{1}$ after taking the limit in $(\ref{flp})_{2}$ as $p\rightarrow\infty$. Then multiplying the equation $(\ref{m1})_{3}$ by $\langle v\rangle^{3}$ with $\langle v\rangle:=(1+|v|^2)^{\frac{1}{2}}$, integrating the resulting equation by parts over $\mathbb{T}\times\mathbb{R}$, and applying the Gr${\rm{\ddot{o}}}$nwall inequality, we derive
\begin{equation}\label{N11}
\begin{split}
&\sup_{t\in[0,T]}\|f(t)\|_{L^1_{3}(\mathbb{T}\times\mathbb{R})}\leq e^{C\rho_{+}(T+\|u\|_{L^1(0,T;L^{\infty}(\mathbb{T}))})}\|f_{0}\|_{L^1_{3}(\mathbb{T}\times\mathbb{R})}\leq C_{T},
\end{split}
\end{equation}
where one has used the estimates $(\ref{basicCNSVFP})$.

We are going to estimate $n$ and $nw$. For any $R>0$, it is easy to verify
\begin{align}
&nw(x,t)\leq \Big{(}\int_{\{|v|\leq R\}}+\int_{\{|v|\geq R\}}\Big{)}vf(x,v,t)dv\leq 2\|f(x,\cdot,t)\|_{L^{\infty}(\mathbb{R})}R^2+\frac{\int_{\mathbb{R}}|v|^3 f(x,v,t)dv}{R^2}.\label{3r}
\end{align}
Choosing $R= (\int_{\mathbb{R}} |v|^3f(x,v,t) dv)^{\frac{1}{4}}$ in (\ref{3r}) and making use of $(\ref{N1time})_{1}$ and $(\ref{N11})$, we obtain
\begin{equation}\label{N14}
\begin{split}
&\sup_{t\in[0,T]}\|nw(t)\|_{L^2(\mathbb{T})}\leq \sup_{t\in[0,T]}\big{(}2\|f(t)\|_{L^{\infty}(\mathbb{T}\times\mathbb{R})}+1\big{)}^{\frac{1}{2}}\|f(t)\|_{L^1_{3}(\mathbb{T}\times\mathbb{R})}^{\frac{1}{2}}\leq C_{T}.
\end{split}
\end{equation}
Similarly, one has
\begin{equation}\label{N13}
\begin{split}
&\sup_{t\in[0,T]}\|n(t)\|_{L^4(\mathbb{T})}\leq \sup_{t\in[0,T]}\big{(}2\|f(t)\|_{L^{\infty}(\mathbb{T}\times\mathbb{R})}+1\big{)}^{\frac{1}{4}}\|f(t)\|_{L^1_{3}(\mathbb{T}\times\mathbb{R})}^{\frac{1}{4}}\leq C_{T}.
\end{split}
\end{equation}
The combination of $(\ref{flp})_{1}$, (\ref{N11}), and (\ref{N14})-(\ref{N13}) gives rise to $(\ref{N1time})_{2}$-$(\ref{N1time})_{3}$. The proof of Lemma \ref{lemma23} is completed.
\end{proof}

The estimates (\ref{basicCNSVFP}) and (\ref{N1time}) are not enough to obtain the compactness of the density $\rho$ in the framework of Lions \cite{lions2}. To overcome the difficulty caused by the density-dependent viscosity coefficient $\mu(\rho)$, we need to establish the additional $L^1(0,T;L^{\infty})$-estimate of the effective viscous flux $\rho^{\gamma}-\mu(\rho)u_{x}$:
\begin{lemma}\label{lemma24}
Let $T>0$, and $(\rho,u,f)$ be any regular solution to the IVP $(\ref{m1})$-$(\ref{kappa})$ for $t\in(0,T]$. Then, under the assumptions of Theorem \ref{theorem11}, it holds
\begin{align}
\sup_{t\in[0,T]}t^{\frac{1}{2}}\|u(t)\|_{H^1(\mathbb{T})}+\|\rho^{\gamma}-\mu(\rho)u_{x}\|_{L^{p}(0,T;L^{\infty}(\mathbb{T}))}\leq C_{T},\quad p\in[1,\frac{4}{3}),\label{N1timew}
\end{align}
for $C_{T}>0$ a constant.
\end{lemma}
\begin{proof}
Multiplying $(\ref{m1})_{2}$ by $u_{t}$ and integrating the resulting equation by parts over $\mathbb{T}$, we have
\begin{align}
&\frac{1}{2}\frac{d}{dt}\|(\sqrt{\mu(\rho)}u_{x})(t)\|_{L^2(\mathbb{T})}^2+\|(\sqrt{\rho}u_{t})(t)\|_{L^2(\mathbb{T})}^2=\sum_{i=1}^3 I_{i}^1,\label{tux}
\end{align}
where $I_{1}^{1}, i=1,2,3,$ are given by
\begin{equation}\nonumber
\begin{split}
&I_{1}^{1}:=\int_{\mathbb{T}}[(\rho nw-\rho un-\rho u u_{x})u_{t}](x,t)dx,\\\
&I_{2}^{1}:=\int_{\mathbb{T}}(\rho^{\gamma} u_{xt})(x,t)dx,\\
&I_{3}^{1}:=\frac{1}{2}\int_{\mathbb{T}}[(\mu(\rho))_{t}|u_{x}|^2](x,t)dx.
\end{split}
\end{equation}
We deal with the terms $I_{i}^1$ $(i=1,2,3)$ as follows. One deduces by $(\ref{basicCNSVFP})_{5}$ and $(\ref{N1time})_{3}$ that
\begin{equation}\nonumber
\begin{split}
&I_{1}^{1}\leq \rho_{+}^{\frac{1}{2}}\|(\sqrt{\rho}u_{t})(t)\|_{L^2(\mathbb{T})}\big{(}\|nw(t)\|_{L^2(\mathbb{T})}+\|u(t)\|_{L^{\infty}(\mathbb{T})}\|n(t)\|_{L^2(\mathbb{T})}+\|u_{x}(t)\|_{L^2(\mathbb{T})}\|u(t)\|_{L^{\infty}(\mathbb{T})}\big{)}\\
&\quad\leq\frac{1}{8}\|(\sqrt{\rho}u_{t})(t)\|_{L^2(\mathbb{T})}^2+C_{T}\|u(t)\|_{L^{\infty}(\mathbb{T})}^2\big{(}1+\|u_{x}(t)\|_{L^2(\mathbb{T})}^2\big{)}+C_{T}.
\end{split}
\end{equation}
Since it holds
\begin{equation}\nonumber
\begin{split}
&-\int_{\mathbb{T}}[(\rho^{\gamma})_{t}u_{x}](x,t)dx\nonumber\\
&\quad=\gamma\int_{\mathbb{T}}\big{[}(\rho^{\gamma}u_{x}+u\rho^{\gamma-1}\rho_{x})u_{x}\big{]}(x,t)dx\nonumber\\
&\quad=\gamma\int_{\mathbb{T}}(\rho^{\gamma}|u_{x}|^2)(x,t)dx+\gamma\int_{\mathbb{T}} \Big{(}u\frac{\rho^{\gamma-1}}{\mu(\rho)}\rho_{x}(\mu(\rho)u_{x}-\rho^{\gamma})\Big{)}(x,t)dx+\gamma\int_{\mathbb{T}}\Big{(}u \frac{\rho^{2\gamma-1}}{\mu(\rho)}\rho_{x}\Big{)}(x,t)dx\nonumber\\
&\quad=\gamma\int_{\mathbb{T}}(\rho^{\gamma}|u_{x}|^2)(x,t)dx-\gamma\int_{\mathbb{T}} \Big{(}\int_{1}^{\rho}\frac{s^{\gamma-1}}{\mu(s)}ds[u(\mu(\rho)u_{x}-\rho^{\gamma})]_{x}\Big{)}(x,t)dx-\gamma\int_{\mathbb{T}}\Big{(}u_{x}\int_{1}^{\rho}\frac{s^{2\gamma-1}}{\mu(s)}ds\Big{)}(x,t)dx\nonumber\\
&\quad=\gamma\int_{\mathbb{T}}\Big{[}\Big{(}\rho^{\gamma}-\mu(\rho)\int_{1}^{\rho}\frac{s^{\gamma-1}}{\mu(s)}ds\Big{)}|u_{x}|^2\Big{]}(x,t)dx\\
&\quad\quad-\gamma\int_{\mathbb{T}} \Big{(}u\int_{1}^{\rho}\frac{s^{\gamma-1}}{\mu(s)}ds(\rho u_{t}+\rho u u_{x}-\rho nw+\rho nu)\Big{)}(x,t)dx\\
&\quad\quad-\gamma\int_{\mathbb{T}}\Big{[}u_{x}\Big{(}\int_{1}^{\rho}\frac{s^{2\gamma-1}}{\mu(s)}ds-\rho^{\gamma}\int_{1}^{\rho}\frac{s^{\gamma-1}}{\mu(s)}ds\Big{)}\Big{]}(x,t)dx,
\end{split}
\end{equation}
derived from the equations $(\ref{m1})_{1}$-$(\ref{m1})_{2}$, we obtain by $(\ref{basicCNSVFP})$ and $(\ref{N1time}) _{2}$ that
\begin{equation}\nonumber
\begin{split}
&I_{2}^{1}=\frac{d}{dt}\int_{\mathbb{T}}(\rho^{\gamma}u_{x})(x,t)dx-\int_{\mathbb{T}}[(\rho^{\gamma})_{t}u_{x}](x,t)dx\\
&\quad\leq \frac{d}{dt}\int_{\mathbb{T}}(\rho^{\gamma}u_{x})(x,t)dx+\frac{1}{8}\|(\sqrt{\rho}u_{t})(t)\|_{L^2(\mathbb{T})}^2+C\|u_{x}(t)\|_{L^2(\mathbb{T})}^2+C_{T}\|u(t)\|_{L^{\infty}(\mathbb{T})}^2+C_{T}.
\end{split}
\end{equation}
It follows from the Gagliardo-Nirenberg inequality, $(\ref{m1})_{2}$, $(\ref{basicCNSVFP})_{5}$, and $(\ref{N1time})_{3}$ that
\begin{align}
&\|u_{x}(t)\|_{L^{\infty}(\mathbb{T})}\nonumber\\
&\quad\leq \|(\mu(\rho)u_{x}-\rho^{\gamma})(t)\|_{L^{\infty}(\mathbb{T})}+\|\rho(t)\|_{L^{\infty}(\mathbb{T})}^{\gamma} \nonumber\\
&\quad\leq C\|(\mu(\rho)u_{x}-\rho^{\gamma})(t)\|_{L^{2}(\mathbb{T})}^{\frac{1}{2}}\|(\mu(\rho)u_{x}-\rho^{\gamma})_{x}(t)\|_{L^{2}(\mathbb{T})}^{\frac{1}{2}}+\|(\mu(\rho)u_{x}-\rho^{\gamma})(t)\|_{L^{1}(\mathbb{T})}+\rho_{+}^{\gamma}\nonumber\\
&\quad\leq C\|(\rho u_{t}+\rho u u_{x}+\rho un-\rho wn)(t)\|_{L^2(\mathbb{T})}^{\frac{1}{2}}\|(\mu(\rho)u_{x}-\rho^{\gamma})(t)\|_{L^2(\mathbb{T})}^{\frac{1}{2}}+C\|u_{x}(t)\|_{L^2(\mathbb{T})}+C\nonumber \\
&\quad\leq C_{T}\|(\sqrt{\rho}u_{t})(t)\|_{L^2(\mathbb{T})}+C_{T}\|u(t)\|_{L^{\infty}(\mathbb{T})}(\|u_{x}(t)\|_{L^2(\mathbb{T})}+1)+C_{T}.\label{uinftyx}
\end{align}
Similarly to the estimate of $I_{2}^{1}$, one deduces by $(\ref{m1})_{1}$-$(\ref{m1})_{2}$, (\ref{basicCNSVFP})-(\ref{N1time}), and (\ref{uinftyx}) that
\begin{equation}\nonumber
\begin{split}
&I_{3}^{1}=-\frac{\beta}{2}\int_{\mathbb{T}}(\rho^{\beta}|u_{x}|^2u_{x})(x,t)dx-\frac{\beta}{2}\int_{\mathbb{T}} \big{(}u\rho^{\beta-1}\rho_{x} |u_{x}|^2\big{)}(x,t)dx\\
&\quad=-\frac{\beta}{2}\int_{\mathbb{T}}(\rho^{\beta}|u_{x}|^2u_{x})(x,t)dx-\frac{\beta}{2}\int_{\mathbb{T}} \Big{[}u\Big{(}\int_{1}^{\rho}\frac{s^{\beta-1}}{\mu^2(s)}ds\Big{)}_{x}(\mu(\rho)u_{x}-\rho^{\gamma})^2\Big{]}(x,t)dxdt\\
&\quad\quad-\beta\int_{\mathbb{T}} \Big{(}u \Big{(}\int_{1}^{\rho}\frac{s^{\beta-1}}{\mu^2(s)}ds\Big{)}_{x}\mu(\rho)u_{x}\rho^{\gamma}\Big{)}(x,t)dx+\frac{\beta}{2}\int_{\mathbb{T}}\Big{[}u\Big{(}\int_{1}^{\rho}\frac{s^{\beta-1}}{\mu^2(s)}ds\Big{)}_{x}\rho^{2\gamma}\Big{]}(x,t)dx\\
&\quad=-\frac{\beta}{2}\int_{\mathbb{T}}(\rho^{\beta}|u_{x}|^2u_{x})(x,t)dx+\beta\int_{\mathbb{T}}\Big{(} u\int_{1}^{\rho}\frac{s^{\beta-1}}{\mu^2(s)}ds(\mu(\rho)u_{x}-\rho^{\gamma})(\rho u_{t}+\rho u u_{x}-\rho nw+\rho un) \Big{)}(x,t)dx\\
&\quad\quad+\frac{\beta}{2}\int_{\mathbb{T}}u_{x}\int_{1}^{\rho}\frac{s^{\beta-1}}{\mu^2(s)}(\mu(\rho)u_{x}-\rho^{\gamma})^2ds-\beta\int_{\mathbb{T}} \Big{[}u \Big{(}\int_{1}^{\rho}\frac{s^{\beta-1}}{\mu^2(s)}ds\Big{)}_{x}(\mu(\rho)u_{x}-\rho^{\gamma})\rho^{\gamma}\Big{]}(x,t)dx\\
&\quad\quad-\frac{\beta}{2}\int_{\mathbb{T}}\Big{[}u\Big{(}\int_{1}^{\rho}\frac{s^{\beta-1}}{\mu^2(s)}ds\Big{)}_{x}\rho^{2\gamma}\Big{]}(x,t)dx\\
&\quad\leq  \frac{1}{8}\|(\sqrt{\rho}u_{t})(t)\|_{L^2(\mathbb{T})}^2+C_{T}\Big{(}\|u_{x}(t)\|_{L^2(\mathbb{T})}^2+\|u(t)\|_{L^{\infty}(\mathbb{T})}^2+1\Big{)}\|u_{x}(t)\|_{L^2(\mathbb{T})}^2+C_{T}.
\end{split}
\end{equation}
Substituting the above estimates of $I^1_{i}$ $(i=1,2,3)$ into (\ref{tux}), multiplying the resulting inequality by $t$, and then integrating it over $[0,t]$, we obtain
\begin{equation}\nonumber
\begin{split}
&\frac{1}{2}t\|u_{x}(t)\|_{L^2(\mathbb{T})}^2+\frac{1}{2}\int_{0}^{t}\tau\|(\sqrt{\rho}u_{t})(\tau)\|_{L^2(\mathbb{T})}^2d\tau\\
&\quad\leq \frac{1}{2}\int_{0}^{t}\tau\|u_{x}(\tau)\|_{L^2(\mathbb{T})}^2d\tau+t\int_{\mathbb{T}}(\rho^{\gamma}u_{x})(x,t)dx-\int_{0}^{t}\tau\int_{\mathbb{T}}(\rho^{\gamma}u_{x})(x,\tau)dxd\tau\\
&\quad\quad+C_{T}\int_{0}^{t}\Big{(}1+\|u(\tau)\|_{L^{\infty}(\mathbb{T})}^2+\|u_{x}(\tau)\|_{L^2(\mathbb{T})}^2\Big{)}\tau\|u_{x}(\tau)\|_{L^2(\mathbb{T})}^2d\tau+C_{T},
\end{split}
\end{equation}
which together with the Gr${\rm{\ddot{o}}}$nwall inequality, $(\ref{basicCNSVFP})$, and (\ref{uinftyx}) gives rise to
\begin{equation}\label{weightp1}
\begin{split}
\sup_{t\in[0,T]}t\|u_{x}(t)\|_{L^2(\mathbb{T})}^2+\int_{0}^{t}\tau\Big{(}\|(\sqrt{\rho}u_{t})(\tau)\|_{L^2(\mathbb{T})}^2+\|u_{x}(\tau)\|_{L^{\infty}(\mathbb{T})}^2\Big{)}d\tau\leq C_{T}.
\end{split}
\end{equation}
By $(\ref{m1})_{2}$, $(\ref{basicCNSVFP})_{5}$, $(\ref{N1time})$, (\ref{weightp1}), and the Gagliardo-Nirenberg inequality, we have
\begin{equation}\nonumber
\begin{split}
&\int_{0}^{T}\|(\rho^{\gamma}-\mu(\rho)u_{x})(t)\|_{L^{\infty}(\mathbb{T})}^{p}dt\\
&\quad\leq  C\int_{0}^{T}\big{(} \|(\rho^{\gamma}-\mu(\rho)u_{x})(t)\|_{L^2(\mathbb{T})}^{\frac{p}{2}}\|(\rho u_{t}+\rho u u_{x}-\rho nw+\rho un)(t)\|_{L^2(\mathbb{T})}^{\frac{p}{2}}+\|(\rho^{\gamma}-\mu(\rho)u_{x})(t)\|_{L^{1}(\mathbb{T})}^{p}\big{)}dt\\
&\quad\leq C_{T}\sup_{t\in[0,T]}\Big{(}t^{\frac{1}{2}}\|u_{x}(t)\|_{L^2(\mathbb{T})}\Big{)}^{\frac{p}{2}}\Big{(}\int_{0}^{T}\big{(}t\|(\sqrt{\rho}u_{t})(t)\|_{L^2(\mathbb{T})}^2+t\|u_{x}(t)\|_{L^{\infty}(\mathbb{T})}^2\big{)}dt\Big{)}^{\frac{p}{4}}\\
&\quad\quad\times\Big{(}\int_{0}^{T}t^{-\frac{2p}{4-p}}dt\Big{)}^{\frac{4-p}{4}}+C_{T}\Big{(}\int_{0}^{T}t\|(\sqrt{\rho}u_{t})(t)\|_{L^2(\mathbb{T})}^2dt\Big{)}^{\frac{p}{4}}\Big{(}\int_{0}^{T}t^{-\frac{p}{4-p}}dt\Big{)}^{\frac{4-p}{4}}+C_{T}\leq C_{T},
\end{split}
\end{equation}
provided $p\in [1,\frac{4}{3})$. Due to $(\ref{basicCNSVFP})$ and (\ref{weightp1}), it also holds \begin{equation}\nonumber
\begin{split}
\sup_{t\in[0,T]}t\|u(t)\|_{H^1(\mathbb{T})}^2\leq C_{T}.
\end{split}
\end{equation}
The proof of Lemma \ref{lemma24} is completed.
\end{proof}

\subsection{Higher-order estimates}

We are ready to establish the higher-order estimates of the solution $(\rho,u,f)$.

\begin{lemma}\label{lemma25}
Let $T>0$, and $(\rho,u,f)$ be any regular solution to the IVP $(\ref{m1})$-$(\ref{kappa})$ for $t\in(0,T]$. Then, under the assumptions of Theorem \ref{theorem12}, it holds
\begin{equation}
\left\{
\begin{split}
&~~\rho(x,t)\geq \rho_{T}>0,\quad (x,t)\in\mathbb{T}\times[0,T],\\
&\sup_{t\in[0,T]}\|(\rho,u)(t)\|_{H^1(\mathbb{T})}+\|u\|_{L^2(0,T;H^2(\mathbb{T}))}+\|u_{t}\|_{L^2(0,T;L^2(\mathbb{T}))}\leq C_{T},\\
&\sup_{t\in[0,T]}\|f(t)\|_{L^2_{k_{0}}(\mathbb{T}\times\mathbb{R})}+\|f_{v}\|_{L^2(0,T;L^2_{k_{0}}(\mathbb{T}\times\mathbb{R}))}\leq C_{T},\\
&\sup_{t\in[0,T]}t^{\frac{1}{2}}\|(u_{xx},u_{t})(t)\|_{L^2(\mathbb{T})}\leq C_{T},\label{N2time}
\end{split}
\right.
\end{equation}
where $\rho_{T}>0$ and $C_{T}>0$ are two constants.
\end{lemma}
\begin{proof}
Let $\theta(\rho)$ be given by (\ref{213}), and the particle paths $\mathcal{X}^{x,t}(s)$ be defined by (\ref{overlinex}) . Integrating the equation (\ref{newm4}) over $[0,t]$ along the particle paths $\mathcal{X}^{x,t}(s)$ for any $(x,t)\in\mathbb{T}\times[0,T_{0}]$ and $s\in[0,t]$, we have by (\ref{basicCNSVFP}) and (\ref{newm4})-(\ref{Flip}) that
\begin{equation}\nonumber
\begin{split}
&\theta(\rho)(x,t)\geq\theta(\rho_{0})(\mathcal{X}^{x,t}(0))-T_{0}\rho_{+}^{\gamma}-\big{|}F(\mathcal{X}^{x,t}(\tau),\tau)|^{\tau=t}_{\tau=0}\big{|}\\
&\quad\quad\quad\quad\geq \inf_{x\in\mathbb{T}}\theta(\rho_{0})(x)-C_{2}-T_{0}\Big{(}\rho_{+}^{\gamma}+C_{3}+\frac{\beta C_{1}^2 \theta(\rho_{+})}{4}\Big{)},
\end{split}
\end{equation}
which proves $(\ref{N2time})_{1}$.

As in \cite{lhl4}, denote the effective velocity
\begin{equation}\label{m}
\begin{split}
U:=u+\mathcal{I}(n)+\rho^{-2}\mu(\rho)\rho_{x},
\end{split}
\end{equation}
so that the equation (\ref{newm1}) can be re-written as
 \begin{equation}
\begin{split}
\rho(U_{t}+uU_{x})+\gamma\rho^{\gamma+1}\mu(\rho)^{-1}U=\gamma\rho^{\gamma+1}\mu(\rho)^{-1}(u+\mathcal{I}(n))+\rho\int_{0}^{1}nw(y,t)dy.\label{bd}
\end{split}
\end{equation}
Multiplying (\ref{bd}) by $U$ and integrating the resulting equation by parts over $\mathbb{T}$, we deduce by (\ref{basicCNSVFP}) that
\begin{equation}\nonumber
\begin{split}
&\frac{1}{2}\frac{d}{dt}\int_{\mathbb{T}}(\rho |U|^2)(x,t)dx+\gamma\int_{\mathbb{T}}(\rho^{\gamma+1}\mu(\rho)^{-1}|U|^2)(x,t)dx\\
&\quad\leq \Big{(}\gamma \rho_{+}^{\gamma}\|(\sqrt{\rho}u)(t)\|_{L^2(\mathbb{T})}+\gamma\rho_{+}^{\gamma}\|\rho(t)\|_{L^1(\mathbb{T})}^{\frac{1}{2}}\|n(t)\|_{L^1(\mathbb{T})}\\
&\quad\quad+\|\rho(t)\|_{L^1(\mathbb{T})}^{\frac{1}{2}}\|f(t)\|_{L^1_{2}(\mathbb{T}\times\mathbb{R})}\Big{)}\Big{(}\int_{\mathbb{T}}(\rho |U|^2)(x,t)dx\Big{)}^{\frac{1}{2}}\\
&\quad\leq C\Big{(}\int_{\mathbb{T}}(\rho |U|^2)(x,t)dx\Big{)}^{\frac{1}{2}},
\end{split}
\end{equation}
which leads to
\begin{equation}\label{U2}
\begin{split}
&\sup_{t\in [0,T]}\|(\sqrt{\rho} U)(t)\|_{L^2(\mathbb{T})}\leq \|(\sqrt{\rho} U)(0)\|_{L^2(\mathbb{T})}+CT.
\end{split}
\end{equation}
By $(\ref{basicCNSVFP})_{5}$, (\ref{m}), and (\ref{U2}), we have
\begin{equation}\label{rhoH1}
\begin{split}
&\sup_{t\in[0,T]}\|\rho_{x}(t)\|_{L^2} \leq \rho_{+}^{\frac{3}{2}}\sup_{t\in[0,T]}\|\big{(}\sqrt{\rho}(U-u-\mathcal{I}(n))\big{)}(t)\|_{L^2}\\
&\quad\quad\quad\quad\quad\quad\quad\leq \rho_{+}^{\frac{3}{2}} \sup_{t\in[0,T]}\Big{(}\|(\sqrt{\rho}U)(t)\|_{L^2}+\|(\sqrt{\rho}u)(t)\|_{L^2}+\|\rho(t)\|_{L^1}^{\frac{1}{2}}\|n(t)\|_{L^1}\Big{)}\leq C_{T}.
\end{split}
\end{equation}
Similarly to (\ref{tux})-(\ref{weightp1}), one can obtain from $(\ref{basicCNSVFP})$, (\ref{N1time}), $(\ref{N2time})$, and the fact $\frac{m_{0}}{\rho_{0}}\in H^1(\mathbb{T})$ that
\begin{equation}\label{N2time12}
\begin{split}
\sup_{t\in[0,T]}\|u(t)\|_{H^1(\mathbb{T})}+\|u_{t}\|_{L^2(0,T;L^2(\mathbb{T}))}\leq C_{T},
\end{split}
\end{equation}
which together with $(\ref{m1})_{2}$, $(\ref{basicCNSVFP})$,  (\ref{rhoH1}) and the embedding $H^1(\mathbb{T})\hookrightarrow L^{\infty}(\mathbb{T})$ implies
\begin{equation}\label{N2time13}
\begin{split}
&\|u_{xx}\|_{L^2(0,T;L^2(\mathbb{T}))}\\
&\quad\leq\|\rho u_{t}+\rho u u_{x}+(\rho^{\gamma})_{x}-\rho nw+\rho nu\|_{L^2(0,T;L^2(\mathbb{T}))}\\
&\quad\leq C_{T}\Big{(}\|u_{t}\|_{L^2(0,T;L^2(\mathbb{T}))}+\|u\|_{L^{\infty}(0,T;L^{\infty}(\mathbb{T}))}\|u_{x}\|_{L^{2}(0,T;L^{2}(\mathbb{T}))}+\|\rho_{x}\|_{L^{\infty}(0,T;L^2(\mathbb{T}))}\\
&\quad\quad+\|nw\|_{L^{\infty}(0,T;L^2(\mathbb{T}))}+\|u\|_{L^{\infty}(0,T;L^{\infty}(\mathbb{T}))}\|n\|_{L^{\infty}(0,T;L^2(\mathbb{T}))}\Big{)}\leq C_{T}.
\end{split}
\end{equation}

Then, we multiply $(\ref{m1})_{3}$ by $\langle v\rangle^{2k_{0}}f$ and integrate the resulting equation by parts over $\mathbb{T}\times\mathbb{R}$ to get
\begin{equation}\nonumber
\begin{split}
&\frac{d}{dt}\|f(t)\|_{L^2_{k_{0}}(\mathbb{T}\times\mathbb{R})}^2+\|(\sqrt{\rho}f_{v})(t)\|_{L^2_{k_{0}}(\mathbb{T}\times\mathbb{R})}^2\leq C\|\rho(t)\|_{L^{\infty}(\mathbb{T})}\big{(}1+\|u(t)\|_{L^{\infty}(\mathbb{T})}\big{)}\|f(t)\|_{L^2_{k_{0}}(\mathbb{T}\times\mathbb{R})}^2,
\end{split}
\end{equation}
which together with $(\ref{basicCNSVFP})$, (\ref{N2time12}), and the Gr${\rm{\ddot{o}}}$nwall inequality leads to
\begin{equation}
\begin{split}
\sup_{t\in[0, T]}\|f(t)\|_{L^2_{k_{0}}(\mathbb{T}\times\mathbb{R})}^2+\|f_{v}\|_{L^2(0,T;L^2_{k_{0}}(\mathbb{T}\times\mathbb{R}))}^2\leq C_{T}.\label{N2time11}
\end{split}
\end{equation}
Due to (\ref{N2time11}) and the fact $k_{0}-2>\frac{1}{2}$, it also holds
\begin{equation}\label{N2time121}
\begin{split}
\sup_{t\in[0,T]}\|\int_{\mathbb{R}}|v|^2f(\cdot,v,t)dv\|_{L^2(\mathbb{T})}\leq \Big{(}\int_{\mathbb{R}}\frac{1}{\langle v\rangle^{2(k_{0}-2)}}dv\Big{)}^{\frac{1}{2}}\sup_{t\in[0,T]}\|f(t)\|_{L^2_{k_{0}}(\mathbb{T}\times\mathbb{R})}\leq C_{T}.
\end{split}
\end{equation}

Finally, we differentiate the equation $(\ref{m1})_{2}$ with respect to $t$ to obtain
\begin{equation}\label{ut}
\begin{split}
&\rho(u_{t}+u u_{xt})-(\mu(\rho)u_{xt} )_{x}+\rho nu_{t}\\
&\quad=-\rho_{t}(u_{t}+u u_{x})-\rho u_{t}u_{x}-(\rho^{\gamma})_{xt}+((\mu(\rho))_{t}u_{x})_{x}+\rho_{t}(nw-nu)-\rho u n_{t}+\rho (nw)_{t}.
\end{split}
\end{equation}
One deduces after a direct computation that
\begin{equation}
\begin{split}
&\frac{1}{2}\frac{d}{dt}\Big{(}t\|(\sqrt{\rho}u_{t})(t)\|_{L^2(\mathbb{T})}^2\Big{)}+t\|(\sqrt{\mu(\rho)}u_{xt})(t)\|_{L^2(\mathbb{T})}^2+t\|(\sqrt{\rho n}u_{t})(t)\|_{L^2(\mathbb{T})}^2\\
&\quad=\frac{1}{2}\|(\sqrt{\rho}u_{t})(t)\|_{L^2(\mathbb{T})}^2+\sum_{i=1}^3 I^2_{i},\label{tut}
\end{split}
\end{equation}
where $I^2_{i}, i=1,2,3,$ are given by
\begin{equation}\nonumber
\begin{split}
&I^2_{1}:=t\int_{\mathbb{T}}\Big{(}\big{[}-\rho_{t}(u_{t}+u u_{x})-\rho u_{t}u_{x}-(\rho^{\gamma})_{xt}+((\mu(\rho))_{t}u_{x})_{x}+\rho_{t}(nw-nu)\big{]}u_{t} \Big{)}(x,t)dx,\\
&I^2_{2}:=t\int_{\mathbb{T}}( \rho un_{t} u_{t})(x,t)dx,\\
&I^2_{3}:=-t\int_{\mathbb{T}}[\rho(nw)_{t} u_{t}](x,t)dx.
\end{split}
\end{equation}
The right-hand side term $I^2_{i}$ $(i=1,2,3)$ can be estimated as follows. It holds by $(\ref{basicCNSVFP})$, (\ref{N2time12}), and (\ref{N2time121}) that
\begin{equation}\nonumber
\begin{split}
&I^2_{1}=t\int_{\mathbb{T}}\Big{(}-\rho_{t}(u_{t}+uu_{x})u_{t}-\rho u_{x}|u_{t}|^2+(\rho^{\gamma})_{t}u_{tx}-(\mu(\rho))_{t}u_{x}u_{xt}+\rho_{t}(nw-nu)u_{t}\Big{)}(x,t)dx\\
&\quad\leq t\|\rho_{t}(t)\|_{L^2(\mathbb{T})}\|(u_{t}+u u_{x})(t)\|_{L^2(\mathbb{T})}\|u_{t}(x)\|_{L^{\infty}(\mathbb{T})}+t\|\rho(t)\|_{L^{\infty}(\mathbb{T})}\|u_{x}(t)\|_{L^{\infty}(\mathbb{T})}\|u_{t}(t)\|_{L^2(\mathbb{T})}^2\\
&\quad\quad+t\big{(}\|(\rho^{\gamma})_{t}(t)\|_{L^2(\mathbb{T})}+\|(\mu(\rho))_{t}(t)\|_{L^2(\mathbb{T})}\big{)}\|u_{xt}(t)\|_{L^2(\mathbb{T})}\\
&\quad\quad+t\|\rho_{t}(t)\|_{L^2(\mathbb{T})}\big{(}\|nw(t)\|_{L^2(\mathbb{T})}+\|u(t)\|_{L^{\infty}(\mathbb{T})}\|n(t)\|_{L^2(\mathbb{T})}\big{)}\|u_{t}(t)\|_{L^{\infty}(\mathbb{T})}\\
&\quad\leq \frac{t}{4}\|(\sqrt{\mu(\rho)}u_{xt})(t)\|_{L^2(\mathbb{T})}^2+C_{T}(1+\|u_{x}(t)\|_{L^{\infty}(\mathbb{T})})t\|u_{t}(t)\|_{L^2(\mathbb{T})}^2+C_{T}.
\end{split}
\end{equation}
By virtue of (\ref{uinftyx}), (\ref{rhoH1})-(\ref{N2time12}) and the fact $n_{t}=-(nw)_{x}$,
we have
\begin{equation}\nonumber
\begin{split}
&I^2_{2}=t\int_{\mathbb{T}}[\rho(u_{x}u_{t}+u u_{xt})nw+\rho_{x} u u_{t} ](x,t)dx\\
&\quad\leq t\big{(}\|u_{x}(t)\|_{L^{\infty}(\mathbb{T})}\|u_{t}(t)\|_{L^2(\mathbb{T})}+\|u(t)\|_{L^{\infty}(\mathbb{T})}\|u_{xt}(t)\|_{L^2(\mathbb{T})} \\
&\quad\quad+\|\rho_{x}(t)\|_{L^2(\mathbb{T})}\|u(t)\|_{L^{\infty}(\mathbb{T})}\|u_{t}(t)\|_{L^2(\mathbb{T})} \big{)}\|nw(t)\|_{L^2(\mathbb{T})}\\
&\quad\leq \frac{t}{4}\|(\sqrt{\mu(\rho)}u_{xt})(t)\|_{L^2(\mathbb{T})}^2+ C_{T}(1+\|u_{x}(t)\|_{L^{\infty}(\mathbb{T})})t\|u_{t}(t)\|_{L^2(\mathbb{T})}^2+C_{T}.
\end{split}
\end{equation}
In addition, it is easy to verify 
$$
(nw)_{t}=-(\int_{\mathbb{R}} |v|^2fdv)_{x}+\rho un -\rho nw,
$$
which together with $(\ref{basicCNSVFP})$, $(\ref{N2time12})$ and  (\ref{N2time121}) yields 
\begin{equation}\nonumber
\begin{split}
&I^2_{3}=t\int_{\mathbb{T}}(\int_{\mathbb{R}} |v|^2f(x,v,t)dv) (\rho u_{t})_{x}(x,t)dx-t\int_{\mathbb{T}}[(\rho un -\rho nw)\rho u_{t}](x,t)dx\\
&\quad\leq t\|\int_{\mathbb{R}} |v|^2f(\cdot,v,t)dv\|_{L^2(\mathbb{T})}\big{(}\|\rho(t)\|_{L^{\infty}(\mathbb{T})}\|u_{xt}(t)\|_{L^2(\mathbb{T})}+\|\rho_{x}(t)\|_{L^2(\mathbb{T})}\|u_{t}(t)\|_{L^{\infty}(\mathbb{T})}\big{)}\\
&\quad\quad+t\|(\sqrt{\rho}u)(t)\|_{L^2(\mathbb{T})} \|\rho(t)\|_{L^{\infty}(\mathbb{T})}^{\frac{3}{2}}\big{(}\|u(t)\|_{L^{\infty}(\mathbb{T})}\|n(t)\|_{L^2(\mathbb{T})}+\|nw(t)\|_{L^2(\mathbb{T})}\big{)}\|u_{t}(t)\|_{L^2(\mathbb{T})}\\
&\quad\leq \frac{t}{4}\|(\sqrt{\mu(\rho)}u_{xt})(t)\|_{L^2(\mathbb{T})}^2+C_{T}t\|u_{t}(t)\|_{L^2(\mathbb{T})}^2+C_{T}.
\end{split}
\end{equation}
Substituting the above estimates of $I^2_{i}$ $(i=1,2,3)$ into (\ref{tut}) and applying the Gr${\rm{\ddot{o}}}$nwall inequality, we obtain
\begin{equation}\label{tut1}
\begin{split}
&\sup_{t\in [0,T]}t\|u_{t}(t)\|_{L^2(\mathbb{T})}^2\leq C_{T}.
\end{split}
\end{equation}
By $(\ref{m1})_{2}$, $(\ref{basicCNSVFP})$, (\ref{N2time12})-(\ref{rhotl2}) and (\ref{tut1}), it also holds
\begin{equation}\label{uxxtt}
\begin{split}
&\sup_{t\in[0,T]}t^{\frac{1}{2}}\|u_{xx}(t)\|_{L^2(\mathbb{T})}\leq\sup_{t\in [0,T]}t^{\frac{1}{2}}\|(\rho u_{t}+\rho u u_{x}+(\rho^{\gamma})_{x}-\rho nw+\rho nu)(t)\|_{L^2(\mathbb{T})}\leq C_{T}.
\end{split}
\end{equation}
The combination of (\ref{N2time12})-(\ref{N2time11}) and (\ref{tut1})-(\ref{uxxtt}) gives rise to (\ref{N2time}). The proof of Lemma \ref{lemma25} is completed.
\end{proof}

Inspired by \cite{chen1,herau1,villani1}, we have the following hypoellipticity estimates of the distribution function $f$ for the Vlasov-Fokker-Planck equation $(\ref{m1})_{3}$:
\begin{lemma}\label{lemma26}
Let $T>0$, and $(\rho,u,f)$ be any regular solution to the IVP $(\ref{m1})$-$(\ref{kappa})$ for $t\in(0,T]$. Then, under the assumptions of Theorem \ref{theorem12}, it holds
\begin{equation}
\begin{split}
&\sup_{t\in[0,T]}\big{(}t\|f_{v}(t)\|_{L^2_{k_{0}-2}(\mathbb{T}\times\mathbb{R})}^2+t^2\|f_{vv}(t)\|_{L^2_{k_{0}-2}(\mathbb{T}\times\mathbb{R})}^2+t^{3}\|(f_{x},f_{vvv})(t)\|_{L^2_{k_{0}-2}(\mathbb{T}\times\mathbb{R})}^2\big{)}\leq C_{T},\label{N2timew2}
\end{split}
\end{equation}
for $C_{T}>0$ a constant.
\end{lemma}
\begin{proof}
First, we differentiate the Vlasov-Fokker-Planck equation $(\ref{m1})_{3}$ with respect to $v$ to obtain
\begin{equation}\label{fv}
\begin{split}
(f_{v})_{t}+v(f_{v})_{x}+ (\rho(u-v)f_{v}-\rho (f_{v})_{v})_{v}=\rho  f_{v}+f_{x}.
\end{split}
\end{equation}
Multiplying (\ref{fv}) by $\langle v\rangle^{2(k_{0}-2)}f_{v}$, integrating the resulting equation by parts over $\mathbb{T}\times\mathbb{R}$, and then making use of $(\ref{basicCNSVFP})_{5}$ and (\ref{N2time}), we have
\begin{align}
&\frac{1}{2}\frac{d}{dt}\|f_{v}(t)\|_{L^2_{k_{0}-2}(\mathbb{T}\times\mathbb{R})}^2+\|(\sqrt{\rho}f_{vv})(t)\|_{L_{k_{0}-2}^2(\mathbb{T}\times\mathbb{R})}^2\nonumber\\
&~~\leq C_{T}\|f_{v}(t)\|_{L^2_{k_{0}-2}(\mathbb{T}\times\mathbb{R})}^2+\|f_{x}(t)\|_{L^2_{k_{0}-2}(\mathbb{T}\times\mathbb{R})}\|f_{v}(t)\|_{L^2_{k_{0}-2}(\mathbb{T}\times\mathbb{R})},\label{fvwp}
\end{align}
from which we infer
\begin{align}
&\frac{1}{2}\frac{d}{dt}\Big{(}t\|f_{v}(t)\|_{L^2_{k_{0}-2}(\mathbb{T}\times\mathbb{R})}^2\Big{)}+t\|(\sqrt{\rho}f_{vv})(t)\|_{L_{k_{0}-2}^2(\mathbb{T}\times\mathbb{R})}^2\nonumber\\
&~~~\leq C_{T}\|f_{v}(t)\|_{L^2_{k_{0}-2}(\mathbb{T}\times\mathbb{R})}^2+t\|f_{v}(t)\|_{L^2_{k_{0}-2}(\mathbb{T}\times\mathbb{R})}\|f_{x}(t)\|_{L^2_{k_{0}-2}(\mathbb{T}\times\mathbb{R})}.\label{fvw}
\end{align}
 Similarly, due to the equation
\begin{equation}\label{fx}
\begin{split}
(f_{x})_{t}+v(f_{x})_{x}+ (\rho(u-v)f_{x}-\rho(f_{x})_{v})_{v}=-(\rho _{x}(u-v)f+\rho  u_{x}f-\rho _{x}f_{v})_{v},
\end{split}
\end{equation}
one has
\begin{align}
&\frac{1}{2}\frac{d}{dt}\|f_{x}(t)\|_{L^2_{k_{0}-2}(\mathbb{T}\times\mathbb{R})}^2+\|(\sqrt{\rho}f_{vx})(t)\|_{L_{k_{0}-2}^2(\mathbb{T}\times\mathbb{R})}^2\nonumber\\
&~~~\leq C_{T}\|f_{x}(t)\|_{L^2_{k_{0}-2}(\mathbb{T}\times\mathbb{R})}^2\nonumber\\
&\quad~~~+C_{T}\|(\rho _{x},u_{x})(t)\|_{L^2(\mathbb{T})}\sup_{x\in\mathbb{T}}\|(\langle v\rangle^{k_{0}-1}f,\langle v\rangle^{k_{0}-2}f_{v})(x,t)\|_{L^2(\mathbb{R})}\|f_{vx}(t)\|_{L_{k_{0}-2}^2(\mathbb{T}\times\mathbb{R})}.\label{fx1}
\end{align}
 To overcome the difficulties caused by the loss of one velocity weight $v$ and the low regularity of $\rho_{x}$ on the right-hand side of (\ref{fx1}), we need
\begin{align}
&\sup_{x\in\mathbb{T}}\int_{\mathbb{R}}\langle v\rangle^{2(k_{0}-1)} f^2(x,v,t)dv\nonumber\\
&\leq \int_{\mathbb{T}\times\mathbb{R}} \langle v\rangle^{2(k_{0}-1)} f^2(x,v,t)dvdx+2\|f(t)\|_{L^2_{k_{0}}(\mathbb{T}\times\mathbb{R})}\|f_{x}(t)\|_{L^2_{k_{0}-2}(\mathbb{T}\times\mathbb{R})},\label{f5}
\end{align}
where one has used the fact $\sup_{x\in\mathbb{T}}|g(x)|\leq \|g\|_{L^1(\mathbb{T})}+\|g_{x}\|_{L^1(\mathbb{T})}$ for any $g\in W^{1,1}(\mathbb{T})$. Similarly, we get 
\begin{align}
&\sup_{x\in\mathbb{T}}\int_{\mathbb{R}}\langle v\rangle^{2(k_{0}-2)}|f_{v}|^2(x,v,t)dv\nonumber\\
&\leq \int_{\mathbb{T}\times\mathbb{R}}\langle v\rangle^{2(k_{0}-2)}|f_{v}|^2(x,v,t)dvdx+2\|f_{v}(t)\|_{L^2_{k_{0}-2}(\mathbb{T}\times\mathbb{R})}\|f_{vx}(t)\|_{L^2_{k_{0}-2}(\mathbb{T}\times\mathbb{R})}.\label{fv3}
\end{align}
We combine (\ref{fx1})-(\ref{fv3}) together to have
\begin{align}
&\frac{1}{2}\frac{d}{dt}\Big{(}t^3\|f_{x}(t)\|_{L^2_{k_{0}-2}(\mathbb{T}\times\mathbb{R})}^2\Big{)}+t^3\|(\sqrt{\rho}f_{vx})(t)\|_{L_{k_{0}-2}^2(\mathbb{T}\times\mathbb{R})}^2\nonumber\\
&~~\leq C_{T}t^2\|f_{x}(t)\|_{L^2_{k_{0}-2}(\mathbb{T}\times\mathbb{R})}^2\nonumber\\
&~~\quad+C_{T}\Big{(}\|f(t)\|_{L^{2}_{k_{0}-1}(\mathbb{T}\times\mathbb{R})}+\|f_{v}(t)\|_{L^{2}_{k_{0}-2}(\mathbb{T}\times\mathbb{R})}+\|f(t)\|_{L^2_{k_{0}}(\mathbb{T}\times\mathbb{R})}^{\frac{1}{2}}t^{\frac{1}{2}}\|f_{x}(t)\|_{L^2_{k_{0}-2}(\mathbb{T}\times\mathbb{R})}^{\frac{1}{2}}\nonumber\\
&~~\quad + \|f_{v}(t)\|_{L^2_{k_{0}-2}(\mathbb{T}\times\mathbb{R})}^{\frac{1}{2}}t^{\frac{3}{4}} \|f_{vx}(t)\|_{L^2_{k_{0}-2}(\mathbb{T}\times\mathbb{R})}^{\frac{1}{2}}\Big{)}t^{\frac{3}{2}}\|f_{vx}(t)\|_{L_{k_{0}-2}^2(\mathbb{T}\times\mathbb{R})}.\label{fxw}
\end{align}
To obtain the dissipation $t^2\|f_{x}(t)\|_{L^2_{k_{0}-2}(\mathbb{T}\times\mathbb{R})}^2$, it follows by (\ref{fv}) and (\ref{fx}) that
\begin{equation}\label{fvfx}
\begin{split}
&\frac{d}{dt}(f_{v}f_{x})+v(f_{v}f_{x})_{x}+|f_{x}|^2=-f_{x}(\rho(u-v)f-\rho f_{v})_{vv}-f_{v}(\rho(u-v)f-\rho f_{v})_{vx}.
\end{split}
\end{equation}
Multiplying (\ref{fvfx}) by $t^2\langle v\rangle^{2(k_{0}-2)}$ and integrating the resulting equation by parts over $\mathbb{T}\times\mathbb{R}$, we obtain
\begin{align}
&\frac{d}{dt}\Big{(}t^2\int_{\mathbb{T}\times\mathbb{R}} \langle v\rangle^{2(k_{0}-2)}(f_{v}f_{x})(x,v,t)dvdx\Big{)}+t^2\|f_{x}(t)\|_{L^{2}_{k_{0}-2}(\mathbb{T}\times\mathbb{R})}^2\nonumber\\
&=2t \int_{\mathbb{T}\times\mathbb{R}} \langle v\rangle^{2(k_{0}-2)}(f_{v}f_{x})(x,v,t)dvdx\nonumber\\
&\quad\quad+t^2\int_{\mathbb{T}\times\mathbb{R}}\Big{[}\Big{(} \big{(} \langle v\rangle^{2(k_{0}-2)}f_{x} \big{)}_{v}+\langle v\rangle^{2(k_{0}-2)}f_{vx}\Big{)}\Big{(}\rho(u-v)f_{v}-\rho f-\rho f_{vv}\Big{)}\Big{]}(x,v,t)dvdx\nonumber\\
&\leq C_{T}\Big{(}\|f(t)\|_{L^2_{k_{0}-2}(\mathbb{T}\times\mathbb{R})}+\|f_{v}(t)\|_{L^2_{k_{0}-1}(\mathbb{T}\times\mathbb{R})}+t^{\frac{1}{2}}\|f_{vv}(t)\|_{L^2_{k_{0}-2}(\mathbb{T}\times\mathbb{R})}\Big{)}\nonumber\\
&\quad~\times\Big{(} t\|f_{x}(t)\|_{L^2_{k_{0}-2}(\mathbb{T}\times\mathbb{R})}+t^{\frac{3}{2}} \|f_{vx}(t)\|_{L^2_{k_{0}-2}(\mathbb{T}\times\mathbb{R})}\Big{)}.\label{fxfvw}
\end{align}
For the constant $\eta\in (0,1)$ to be determined, introducing 
\begin{equation}\nonumber
\left\{
\begin{split}
&L^{\eta}(t):=\frac{1}{2}t\|f_{v}(t)\|_{L^2_{k_{0}-2}(\mathbb{T}\times\mathbb{R})}^2+\frac{1}{2}\eta^3t^3\|f_{x}(t)\|_{L^2_{k_{0}-2}(\mathbb{T}\times\mathbb{R})}^2+\eta^{2}t^2\int_{\mathbb{T}\times\mathbb{R}} \langle v\rangle^{2(k_{0}-2)}(f_{v}f_{x})(x,v,t)dvdx,\\
&D^{\eta}(t):=\|(\sqrt{\rho} f_{vx})(t)\|_{L_{k_{0}-2}^2(\mathbb{T}\times\mathbb{R})}^2+\eta^3 t^3\|(\sqrt{\rho}f_{vx})(t)\|_{L_{k_{0}-2}^2(\mathbb{T}\times\mathbb{R})}^2+\eta^{2}t^2\|f_{x}(t)\|_{L^{2}_{k_{0}-2}(\mathbb{T}\times\mathbb{R})}^2,
\end{split}
\right.
\end{equation}
 we obtain from  (\ref{fvw}), (\ref{fxw}), and (\ref{fxfvw}) that
\begin{align}
&\frac{d}{dt}L^{\eta}(t)+D^{\eta}(t) \nonumber\\
&\leq C_{T}\|f_{v}(t)\|_{L^2_{k_{0}-2}(\mathbb{T}\times\mathbb{R})}^2+t\|f_{v}(t)\|_{L^2_{k_{0}-2}(\mathbb{T}\times\mathbb{R})}\|f_{x}(t)\|_{L^2_{k_{0}-2}(\mathbb{T}\times\mathbb{R})} +C_{T}\eta^3t^2\|f_{x}(t)\|_{L^2_{k_{0}-2}(\mathbb{T}\times\mathbb{R})}^2\nonumber\\
&\quad+C_{T}\eta^{\frac{3}{4}}\Big{(}\|f(t)\|_{L^{2}_{k_{0}-1}(\mathbb{T}\times\mathbb{R})}+\|f_{v}(t)\|_{L^{2}_{k_{0}-2}(\mathbb{T}\times\mathbb{R})}+\|f(t)\|_{L^2_{k_{0}}(\mathbb{T}\times\mathbb{R})}^{\frac{1}{2}}\eta^{\frac{1}{2}}t^{\frac{1}{2}}\|f_{x}(t)\|_{L^2_{k_{0}-2}(\mathbb{T}\times\mathbb{R})}^{\frac{1}{2}}\nonumber\\
&\quad + \|f_{v}(t)\|_{L^2_{k_{0}-2}(\mathbb{T}\times\mathbb{R})}^{\frac{1}{2}}\eta^{\frac{3}{4}}t^{\frac{3}{4}} \|f_{vx}(t)\|_{L^2_{k_{0}-2}(\mathbb{T}\times\mathbb{R})}^{\frac{1}{2}}\Big{)} \eta^{\frac{3}{2}}t^{\frac{3}{2}}\|f_{vx}(t)\|_{L_{k_{0}-2}^2(\mathbb{T}\times\mathbb{R})}\nonumber\\
&\quad+C_{T}\eta^{\frac{1}{2}}\Big{(}\|f(t)\|_{L^2_{k_{0}-2}(\mathbb{T}\times\mathbb{R})}+\|f_{v}(t)\|_{L^2_{k_{0}-1}(\mathbb{T}\times\mathbb{R})}+t^{\frac{1}{2}}\|f_{vv}(t)\|_{L^2_{k_{0}-2}(\mathbb{T}\times\mathbb{R})}\Big{)} \nonumber\\
&\quad\quad\times\Big{(} \eta t\|f_{x}(t)\|_{L^2_{k_{0}-2}(\mathbb{T}\times\mathbb{R})}+\eta^{\frac{3}{2}}t^{\frac{3}{2}} \|f_{vx}(t)\|_{L^2_{k_{0}-2}(\mathbb{T}\times\mathbb{R})}\Big{)}.   \label{Lppp}
\end{align}
It is easy to verify
\begin{equation}\label{Lsim}
\left\{
\begin{split}
&L^{\eta}(t)\geq (1-\frac{1}{2}\eta^{\frac{1}{2}})\Big{(}t\|f_{v}(t)\|_{L^2_{k_{0}-2}(\mathbb{T}\times\mathbb{R})}^2+\eta^3t^3\|f_{x}(t)\|_{L^2_{k_{0}-2}(\mathbb{T}\times\mathbb{R})}^2\Big{)},\\
&D^{\eta}(t)\geq \min\{1,\rho_{T}\}\Big{(} t\|f_{vv}(t)\|_{L^{2}_{k_{0}-2}(\mathbb{T}\times\mathbb{R})}^2+\eta^3t^3\|f_{vx}(t)\|_{L^{2}_{k_{0}-2}(\mathbb{T}\times\mathbb{R})}^2+\eta^{2}t^2\|f_{x}(t)\|_{L^{2}_{k_{0}-2}(\mathbb{T}\times\mathbb{R})}^2\Big{)},
\end{split}
\right.
\end{equation}
where the constant $\rho_{T}>0$ is given by $(\ref{N2time})_{1}$. Then for any $\eta\in(0,1)$, one has by (\ref{Lppp})-$(\ref{Lsim})$ that
 \begin{align}
&\frac{d}{dt}L^{\eta}(t)+D^{\eta}(t)\nonumber\\
&\leq C_{T}\|f_{v}(t)\|_{L^2_{k_{0}-2}(\mathbb{T}\times\mathbb{R})}^2+\frac{C}{\eta}\|f_{v}(t)\|_{L^2_{k_{0}-2}(\mathbb{T}\times\mathbb{R})} \big{(}D^{\eta}(t) \big{)}^{\frac{1}{2}}+C_{T}\eta D^{\eta}(t)\nonumber\\
&\quad+C_{T}\eta^{\frac{3}{4}}\Big{(}\|(f,f_{v})(t)\|_{L^2_{k_{0}-1}(\mathbb{T}\times\mathbb{R})}+\|(f,f_{v})(t)\|_{L^2_{k_{0}}(\mathbb{T}\times\mathbb{R})}^{\frac{1}{2}}\big{(}D^{\eta}(t) \big{)}^{\frac{1}{4}}\Big{)} \big{(}D^{\eta}(t) \big{)}^{\frac{1}{2}}\nonumber\\
&\quad+C_{T}\eta^{\frac{1}{2}}\Big{(}\|(f,f_{v})(t)\|_{L^{2}_{k_{0}-2}(\mathbb{T}\times\mathbb{R})}+\big{(}D^{\eta}(t) \big{)}^{\frac{1}{2}}\Big{)}\big{(}D^{\eta}(t) \big{)}^{\frac{1}{2}}\nonumber\\
&\leq C_{T} \eta^{\frac{1}{2}}D^{\eta}(t)+\frac{C_{T}}{\eta^{\frac{5}{2}}}\|(f,f_{v})(t)\|_{L^{2}_{k_{0}}(\mathbb{T}\times\mathbb{R})}^2.\label{Lppp1}
\end{align}
Since it follows $L^{\eta}(0)=0$, we choose $
\eta=\min\Big{\{}1,\frac{1}{4C_{T}^2}\Big{\}}$ and make use of $(\ref{N2time})$ and (\ref{Lppp})-(\ref{Lppp1}) to get
\begin{align}
&\sup_{t\in[0,T]}\Big{(}t\|f_{v}(t)\|_{L^2_{k_{0}-2}(\mathbb{T}\times\mathbb{R})}^2+t^3\|f_{x}(t)\|_{L^2_{k_{0}-2}(\mathbb{T}\times\mathbb{R})}^2\Big{)}\nonumber\\
&\quad+\int_{0}^{T}\Big{(}t\|f_{vv}(t)\|_{L^2_{k_{0}-2}(\mathbb{T}\times\mathbb{R})}^2+t^3\|f_{vx}(t)\|_{L^2_{k_{0}-2}(\mathbb{T}\times\mathbb{R})}^2+t^2\|f_{x}(t)\|_{L^{2}_{k_{0}-2}(\mathbb{T}\times\mathbb{R})}^2\Big{)}dt\leq C_{T}.\label{N2timew21}
\end{align}

Finally, differentiating $(\ref{fv})$ thrice with respect to $v$ and multiplying the resulting equation by $t^3$, we have
$$
(t^3f_{vvv})_{t}+v(t^3f_{vvv})_{x}+(\rho(u-v)t^3f_{vvv}-\rho(t^3f_{vvv})_{v})_{v}=(3t^2f_{vv}+3t^3\rho f_{vv}+3t^3f_{xv})_{v}.
$$
Similarly to (\ref{fv})-(\ref{fx1}), one can show
\begin{align}
&\sup_{t\in[0,T]}t^3\|f_{vvv}(t)\|_{L^2_{k_{0}-2}(\mathbb{T}\times\mathbb{R})}^2\nonumber\\
&\leq C_{T}\int_{0}^{T}\|(3t^2f_{vv}+3t^3\rho f_{vv}+3t^3f_{xv})(t)\|_{L^{2}_{k_{0}-2}(\mathbb{T}\times\mathbb{R})}^2dt\nonumber\\
&\leq C_{T}\int_{0}^{T}\Big{(}t^2\|f_{vv}(t)\|_{L^{2}_{k_{0}-2}(\mathbb{T}\times\mathbb{R})}^2+t^3\|f_{xv}(t)\|_{L^{2}_{k_{0}-2}(\mathbb{T}\times\mathbb{R})}^2\Big{)}dt\leq C_{T}.\label{N2timew22}
\end{align}
By the Gagliardo-Nirenberg inequality, we also have
\begin{align}
&\sup_{t\in[0,T]}t\|f_{vv}(t)\|_{L^2_{k_{0}-2}(\mathbb{T}\times\mathbb{R})}\leq C\sup_{t\in[0,T]}\Big{(}t^{\frac{1}{2}}\|f_{v}(t)\|_{L^2_{k_{0}-2}(\mathbb{T}\times\mathbb{R})}\Big{)}^{\frac{1}{2}}\Big{(}t^{\frac{3}{2}}\|f_{vvv}(t)\|_{L^2_{k_{0}-2}(\mathbb{T}\times\mathbb{R})}\Big{)}^{\frac{1}{2}}\leq C_{T}.\label{fvvw}
\end{align}
The combination of (\ref{N2timew21})-(\ref{fvvw}) leads to (\ref{N2timew2}). The proof of Lemma \ref{lemma26} is completed.
\end{proof}

\section{Uniquenss}

\begin{prop}\label{prop31}
Let $T>0$, and $(\rho_{i},u_{i},f_{i})$, $i=1,2,$ be two weak solutions to the IVP $(\ref{m1})$-$(\ref{kappa})$ on $[0,T]$ in the sense of Definition \ref{defn11} satisfying $(\ref{r0})$, $(\ref{r1})$, and $(\ref{r2})$. Then it holds $(\rho_{1},u_{1},f_{1})=(\rho_{2},u_{2},f_{2})$ a.e. in $(x,v,t)\in\mathbb{T}\times\mathbb{R}\times[0,T]$.
\end{prop}

\begin{proof}
We first estimate $(\rho_{1}-\rho_{2})$. It is easy to verify
\begin{equation}
\begin{split}
&(\rho_{1}-\rho_{2})_{t}+((\rho_{1}-\rho_{2})u_{1})_{x}=-\rho_{2x}(u_{1}-u_{2})-\rho_{2}(u_{1}-u_{2})_{x}.\label{widetilderho}
\end{split}
\end{equation}
Multiplying $(\ref{widetilderho})$ by $(\rho_{1}-\rho_{2})$ and integrating the resulting equation by parts over $\mathbb{T}$, we have
\begin{align}
&\frac{1}{2}\frac{d}{dt}\|(\rho_{1}-\rho_{2})(t)\|_{L^2(\mathbb{T})}^2\nonumber\\
&\quad\leq C_{T}\big{(}\|(u_{1}-u_{2})_{x}(t)\|_{L^2(\mathbb{T})}+\|\big{(}\sqrt{\rho_{1}}(u_{1}-u_{2})\big{)}(t)\|_{L^2(\mathbb{T})}\big{)}\|(\rho_{1}-\rho_{2})(t)\|_{L^2(\mathbb{T})}\nonumber\\
&\quad\quad+C_{T}\|u_{2x}(t)\|_{H^1(\mathbb{T})}\|(\rho_{1}-\rho_{2})(t)\|_{L^2(\mathbb{T})}^2,\label{3111}
\end{align}
where one has used (\ref{r2}) and the fact
\begin{equation}\label{uinfty11}
\begin{split}
&\|(u_{1}-u_{2})(t)\|_{L^{\infty}(\mathbb{T})}\leq \|(u_{1}-u_{2})_{x}(t)\|_{L^2(\mathbb{T})}+\frac{\|\big{(}\sqrt{\rho_{1}}(u_{1}-u_{2})\big{)}(t)\|_{L^2(\mathbb{T})}}{\|\rho_{0}\|_{L^1(\mathbb{T})}^{\frac{1}{2}}},\quad\forall t\in[0,T].
\end{split}
\end{equation}
By (\ref{3111}), we obtain
\begin{equation}\nonumber
\begin{split}
&\frac{d}{dt}\|(\rho_{1}-\rho_{2})(t)\|_{L^2(\mathbb{T})}\\\
&\leq C_{T}\Big{(}\|(u_{1}-u_{2})_{x}(t)\|_{L^2(\mathbb{T})}+\|\big{(}\sqrt{\rho_{1}}(u_{1}-u_{2})\big{)}(t)\|_{L^2(\mathbb{T})}\Big{)}+\|u_{2}(t)\|_{H^2(\mathbb{T})}\|(\rho_{1}-\rho_{2})(t)\|_{L^2(\mathbb{T})},
\end{split}
\end{equation}
which together with the Gr${\rm{\ddot{o}}}$nwall inequality leads to
\begin{align}
&\|(\rho_{1}-\rho_{2})(t)\|_{L^2(\mathbb{T})}\leq C_{T}t^{\frac{1}{2}}\Big{(}\sup_{\tau\in [0,t]}\|\big{(}\sqrt{\rho_{1}}(u_{1}-u_{2})\big{)}(\tau)\|_{L^2(\mathbb{T})}+\|(u_{1}-u_{2})_{x}\|_{L^2(0,t;L^2(\mathbb{T}))}\Big{)}.\label{widetilderho1}
\end{align}

To estimate $(u_{1}-u_{2})$, noticing that it holds
\begin{align}
&\rho _{1}(u_{1}-u_{2})_{t}+\rho_{1}u_{1}(u_{1}-u_{2})_{x}-[\mu(\rho_{1})(u_{1}-u_{2})_{x}]_{x}+\rho _{1}\int_{\mathbb{R}} f_{1}dv(u_{1}-u_{2}) \nonumber\\
&\quad=-(\rho_{1}^{\gamma}-\rho_{2}^{\gamma})_{x}-(\rho_{1}-\rho_{2})(u_{2t}+u_{2}u_{2x})-\rho _{1}(u_{1}-u_{2})u_{2x}+[(\mu(\rho_{1})-\mu(\rho_{2}))u_{2x}]_{x} \nonumber\\
&\quad\quad+\int_{\mathbb{R}}(\rho _{1}-\rho _{2})(v-u_{2})f_{1}dv+\int_{\mathbb{R}}\rho _{2}(v-u_{2})(f_{1}-f_{2})dv,\label{widetildeu}
\end{align}
we multiply $(\ref{widetildeu})$ by $(u_{1}-u_{2})$ and integrate the resulting equation by parts over $\mathbb{T}\times[0,t]$ to derive
\begin{align}
&\frac{1}{2}\|\big{(}\sqrt{\rho _{1}}(u_{1}-u_{2})\big{)}(t)\|_{L^2(\mathbb{T})}^2+\int_{0}^{t}\|\big{(}\sqrt{\mu(\rho_{1})}(u_{1}-u_{2})_{x}\big{)}(\tau)\|_{L^2(\mathbb{T})}^2d\tau\nonumber\\
&\quad\quad+\int_{0}^{t}\|\big{(}\sqrt{\rho _{1}f_{1}}(u_{1}-u_{2})\big{)}(\tau)\|_{L^2(\mathbb{T}\times\mathbb{R})}^2d\tau=\sum_{i=1}^{4} I^3_{i},\label{uwidetildepp}
\end{align}
where $I^3_{i}, i=1,...,4,$ are given by
\begin{equation}\nonumber
\begin{split}
&I_{1}^{3}:=\int_{0}^{t}\int_{\mathbb{T}}\big{[}(\rho_{1}^{\gamma}-\rho_{2}^{\gamma})(u_{1}-u_{2})_{x}\big{]}(x,\tau)dxd\tau,\\
&I_{2}^{3}:=-\int_{0}^{t}\int_{\mathbb{T}}\big{[}(\rho_{1}-\rho_{2})(u_{2t}+u_{2}u_{2x})(u_{1}-u_{2})+\rho_{1} u_{2x}|u_{1}-u_{2}|^2\big{]}(x,\tau)dxd\tau,\\
&I_{3}^{3}:=\int_{0}^{t}\int_{\mathbb{T}}\big{[}-(\mu(\rho_{1})-\mu(\rho_{2}))u_{2x}(u_{1}-u_{2})_{x}+(\rho_{1}-\rho_{2})(u_{1}-u_{2})(v-u_{2})f_{1}\big{]}(x,\tau)dxd\tau,\\
&I_{4}^{3}:=\int_{0}^{t}\int_{\mathbb{T}\times\mathbb{R}}\big{[}\rho_{2}(v-u_{2})(f_{1}-f_{2})(u_{1}-u_{2})\big{]}(x,v,\tau)dxdvd\tau.
\end{split}
\end{equation}
By $(\ref{r1})$ and the Young inequality,  one has
\begin{equation}\nonumber
\begin{split}
&I_{1}^{3}\leq C\int_{0}^{t}\|(\rho_{1}-\rho_{2})(\tau)\|_{L^2(\mathbb{T})}^2d\tau+\frac{1}{100}\int_{0}^{t}\|(u_{1}-u_{2})_{x}(\tau)\|_{L^2(\mathbb{T})}^2d\tau.
\end{split}
\end{equation}
For the term $I_{2}^3$, it follows from $(\ref{r2})$ and (\ref{uinfty11}) that
\begin{equation}\nonumber
\begin{split}
&I_{2}^{3}\leq C_{T}\int_{0}^{t}\Big{(}1+\|u_{2t}(\tau)\|_{L^2(\mathbb{T})}^2+\|u_{2}(\tau)\|_{H^2(\mathbb{T})}\Big{)}\Big{(}\|\big{(}\sqrt{\rho_{1}}(u_{1}-u_{2})\big{)}(\tau)\|_{L^2(\mathbb{T})}^2\|(\rho_{1}-\rho_{2})(\tau)\|_{L^2(\mathbb{T})}^2\Big{)}d\tau\\
&\quad\quad+\frac{1}{100}\int_{0}^{t}\|(u_{1}-u_{2})_{x}(\tau)\|_{L^2(\mathbb{T})}^2d\tau.
\end{split}
\end{equation}
Due to $(\ref{r1})$, (\ref{r2}), and (\ref{uinfty11}), it also holds
\begin{equation}\nonumber
\begin{split}
&I_{3}^{3}\leq C_{T}\int_{0}^{t}\big{(}1+\|u_{2}(\tau)\|_{H^2(\mathbb{T})}\big{)}\Big{(}\|\big{(}\sqrt{\rho_{1}}(u_{1}-u_{2})\big{)}(\tau)\|_{L^2(\mathbb{T})}^2+\|(\rho_{1}-\rho_{2})(\tau)\|_{L^2(\mathbb{T})}^2\Big{)}d\tau\\
&\quad\quad+\frac{1}{100}\int_{0}^{t}\|(u_{1}-u_{2})_{x}(\tau)\|_{L^2(\mathbb{T})}^2d\tau.
\end{split}
\end{equation}
Finally, the term $I_{4}^{3}$ can be estimated as follows:
\begin{equation}\nonumber
\begin{split}
&I_{4}^{3}\leq C \int_{0}^{t}\big{(}1+\|u_{2}(\tau)\|_{L^{\infty}(\mathbb{T})}\big{)}\|(u_{1}-u_{2})(\tau)\|_{L^{\infty}(\mathbb{T})}\|(f_{1}-f_{2})(\tau)\|_{L^1_{1}(\mathbb{T}\times\mathbb{R})}d\tau\\
&\quad\leq C_{T}\int_{0}^{t}\Big{(}\|\big{(}\sqrt{\rho_{1}}(u_{1}-u_{2})\big{)}(\tau)\|_{L^2(\mathbb{T})}^2d\tau+\|(f_{1}-f_{2})(\tau)\|_{L^1_{1}(\mathbb{T}\times\mathbb{R})}^2\Big{)}d\tau\\
&\quad\quad+\frac{1}{100}\int_{0}^{t}\|(u_{1}-u_{2})_{x}(\tau)\|_{L^2(\mathbb{T})}^2d\tau.
\end{split}
\end{equation}
Substituting the above estimates of $I_{i}^{3}$ $(i=1,2,3,4)$ into (\ref{uwidetildepp}), we get
\begin{align}
&\sup_{\tau\in [0,t]}\|\big{(}\sqrt{\rho _{1}}(u_{1}-u_{2})\big{)}(\tau)\|_{L^2(\mathbb{T})}^2+\int_{0}^{t}\|(u_{1}-u_{2})_{x}(\tau)\|_{L^2(\mathbb{T})}^2d\tau \nonumber\\
&\quad\leq \int_{0}^{t}\Big{(}1+\|u_{2t}(\tau)\|_{L^2(\mathbb{T})}^2+\|u_{2}(\tau)\|_{H^2(\mathbb{T})}\Big{)}\Big{(}\|\big{(}\sqrt{\rho_{1}}(u_{1}-u_{2})\big{)}(\tau)\|_{L^2(\mathbb{T})}^2+\|(\rho_{1}-\rho_{2})(\tau)\|_{L^2(\mathbb{T})}^2\Big{)}d\tau\nonumber\\
&\quad\quad+\int_{0}^{t}\|(f_{1}-f_{2})(\tau)\|_{L^1_{1}(\mathbb{T}\times\mathbb{R})}^2d\tau.\label{widetildeu1}
\end{align}

We are ready to estimate $(f_{1}-f_{2})$. It can be verified by $(\ref{m1})_{3}$ that
\begin{equation}\label{widetildef}
\begin{split}
&(f_{1}-f_{2})_{t}+v(f_{1}-f_{2})_{x}+[\rho_{1}(u_{1}-v)(f_{1}-f_{2})-\rho_{1}(f_{1}-f_{2})_{v}]_{v}\\
&\quad=-\rho_{1}(u_{1}-u_{2})f_{2v}-(\rho_{1}-\rho_{2})(u_{2}-v)f_{2v}+(\rho_{1}-\rho_{2})f_{2}+(\rho_{1}-\rho_{2})f_{2vv}.
\end{split}
\end{equation}
Employing Lemma \ref{lemma610} below to the equation (\ref{widetildef}) for $b(s)=b^{\delta}(s)=(s^2+\delta^2)^{\frac{1}{2}}$ satisfying $ |(b^{\delta}(s))'|=|s(s^2+\delta^2)^{-\frac{1}{2}}|\leq 1$ and $(b^{\delta}(s))''=\delta^2(s^2+\delta^2)^{-\frac{3}{2}}\geq 0$ for any $\delta>0$, we obtain
\begin{equation}\label{leqf1f2}
\begin{split}
&(b^{\delta}(f_{1}-f_{2}))_{t}+v(b^{\delta}(f_{1}-f_{2}))_{x}+[\rho_{1}(u_{1}-v)b_{\delta}(f_{1}-f_{2})]_{v}-\rho[b^{\delta}(f_{1}-f_{2})]_{vv} \\
&\quad\leq |-\rho_{1}(u_{1}-u_{2})f_{2v}-(\rho_{1}-\rho_{2})(u_{2}-v)f_{2v}+(\rho_{1}-\rho_{2})f_{2} +(\rho_{1}-\rho_{2})f_{2vv}|.
\end{split}
\end{equation}
Since it follows $b^{\delta}(s)\rightarrow |s|$ as $\delta\rightarrow 0$, we choose the text function $\psi^{\delta}(x,v,t)\in \mathcal{D}(\mathbb{T}\times\mathbb{R}\times(0,T))$ in (\ref{leqf1f2}) satisfying $\psi^{\delta}(x,v,t)\rightarrow \langle v\rangle$ as $\delta\rightarrow 0$ and employ the dominated convergence theorem to have
\begin{align}
&\|(f_{1}-f_{2})(t)\|_{L^1_{1}(\mathbb{T}\times\mathbb{R})}\leq \sum_{i=1}^{3}I^4_{i},\label{remardf}
\end{align}
where $I^4_{i}, i=1,2,3,$ are given by
\begin{equation}\nonumber
\begin{split}
&I_{1}^{4}:=\int_{0}^{t}\|\big{(}\rho_{1}(u_{1}-v)(f_{1}-f_{2})\big{)}(\tau)\|_{L^1(\mathbb{T}\times\mathbb{R})}d\tau,\\
&I_{2}^{4}:=\int_{0}^{t}\|\big{(}\rho_{1}(u_{1}-u_{2})f_{2v}+(\rho_{1}-\rho_{2})(u_{2}-v)f_{2v}-(\rho_{1}-\rho_{2})f_{2} \big{)}(\tau)\|_{L^1_{1}(\mathbb{T}\times\mathbb{R})}d\tau,\\
&I_{3}^{4}:=\int_{0}^{t}\|\big{(}(\rho_{1}-\rho_{2})f_{2vv}\big{)}(\tau)\|_{L^1_{1}(\mathbb{T}\times\mathbb{R})}d\tau.
\end{split}
\end{equation}
It follows from $(\ref{r1})_{1}$ that
\begin{equation}\nonumber
\begin{split}
&I_{1}^{4}\leq \rho_{+}\int_{0}^{t}\big{(}\|u_{1}(\tau)\|_{L^{\infty}(\mathbb{T})}+1\big{)}\|(f_{1}-f_{2})(\tau)\|_{L^1_{1}(\mathbb{T}\times\mathbb{R})}d\tau.
\end{split}
\end{equation}
Due to $(\ref{r2})$, (\ref{widetilderho1}) and the fact $k_{0}-2>\frac{1}{2}$, we have
\begin{equation}\nonumber
\begin{split}
&I_{2}^{4}\leq \rho_{+}^{\frac{1}{2}}\int_{0}^{t}\|\big{(}\sqrt{\rho_{1}}(u_{1}-u_{2})\big{)}(\tau)\|_{L^2(\mathbb{T})}\Big{(}\int_{\mathbb{R}}\frac{1}{\langle v\rangle^{2(k_{0}-1)}}dv\Big{)}^{\frac{1}{2}}\|f_{2v}(\tau)\|_{L^2_{k_{0}}(\mathbb{T}\times\mathbb{R})}d\tau\\
&\quad\quad+\int_{0}^{t}\|(\rho_{1}-\rho_{2})(\tau)\|_{L^2(\mathbb{T})}\big{(}1+\|u_{2}(\tau)\|_{L^{\infty}(\mathbb{T})}\big{)}\Big{(}\int_{\mathbb{R}}\frac{1}{\langle v\rangle^{2(k_{0}-2)}}dv\Big{)}^{\frac{1}{2}}\|f_{2v}(\tau)\|_{L^2_{k_{0}}(\mathbb{T}\times\mathbb{R})}d\tau\\
&\quad\quad+\int_{0}^{t}\|(\rho_{1}-\rho_{2})(\tau)\|_{L^2(\mathbb{T})}\Big{(}\int_{\mathbb{R}}\frac{1}{\langle v\rangle^{2(k_{0}-1)}}dv\Big{)}^{\frac{1}{2}}\|f_{2}(\tau)\|_{L^2_{k_{0}}(\mathbb{T}\times\mathbb{R})}d\tau\\
&\quad\leq C_{T}\sup_{\tau\in [0,t]}\|\big{(}\sqrt{\rho_{1}}(u_{1}-u_{2})\big{)}(\tau)\|_{L^2(\mathbb{T})}+C_{T}\|(u_{1}-u_{2})_{x}\|_{L^2(0,t;L^2(\mathbb{T}))}.
\end{split}
\end{equation}
For the key term $I_{4}^{4}$, we make use of $(\ref{r2})_{5}$, (\ref{widetilderho1}) and the fact $k_{0}-3>\frac{1}{2}$ to obtain
\begin{equation}\nonumber
\begin{split}
&I_{4}^{4}\leq\int_{0}^{t}\|(\rho_{1}-\rho_{2})(\tau)\|_{L^2(\mathbb{T})}\Big{(}\int_{\mathbb{R}}\frac{1}{\langle v\rangle^{2(k_{0}-3)}}dv\Big{)}^{\frac{1}{2}}\|f_{2vv}(\tau)\|_{L^2_{k_{0}-2}(\mathbb{T}\times\mathbb{R})}d\tau\\
&\quad\leq \Big{(}\sup_{\tau\in [0,t]}\frac{1}{\tau^{\frac{1}{2}}}\|(\rho_{1}-\rho_{2})(\tau)\|_{L^2(\mathbb{T})}\Big{)}\Big{(}\sup_{\tau\in [0,t]} \tau \|f_{2vv}(\tau)\|_{L^2_{k_{0}-2}(\mathbb{T}\times\mathbb{R})}\Big{)}\int_{0}^{t}\frac{1}{\tau^{\frac{1}{2}}}d\tau\\
&\quad\leq C_{T}\sup_{\tau\in [0,t]}\|\big{(}\sqrt{\rho_{1}}(u_{1}-u_{2})\big{)}(\tau)\|_{L^2(\mathbb{T})}+C_{T}\|(u_{1}-u_{2})_{x}\|_{L^2(0,t;L^2(\mathbb{T}))}.
\end{split}
\end{equation}
Substituting the above estimates of $I_{i}^{4}$ $(i=1,2,3,4)$ into (\ref{remardf}) and applying the Gr${\rm{\ddot{o}}}$nwall inequality, we have
\begin{align}
&\sup_{\tau\in [0,t]}\|(f_{1}-f_{2})(\tau)\|_{L^1_{1}(\mathbb{T}\times\mathbb{R})}\leq C_{T}\sup_{\tau\in [0,t]}\|\big{(}\sqrt{\rho_{1}}(u_{1}-u_{2})\big{)}(\tau)\|_{L^2(\mathbb{T})}+C_{T}\|(u_{1}-u_{2})_{x}\|_{L^2(0,t;L^2(\mathbb{T}))},\label{remardfff}
\end{align}
which together with (\ref{widetilderho1})-(\ref{widetildeu}) leads to
\begin{align}
&\sup_{\tau\in [0,t]}\|\big{(}\sqrt{\rho_{1}}(u_{1}-u_{2})\big{)}(\tau)\|_{L^2(\mathbb{T})}^2+\|(u_{1}-u_{2})_{x}\|_{L^2(0,t;L^2(\mathbb{T}))}^2\nonumber\\
&\quad\leq C_{T}\int_{0}^{t}\Big{(}1+\|u_{2t}(\tau)\|_{L^2(\mathbb{T})}^2+\|u_{2}(\tau)\|_{H^2(\mathbb{T})}\Big{)}\Big{(}\sup_{\omega\in [0,\tau]}\|\big{(}\sqrt{\rho_{1}}(u_{1}-u_{2})\big{)}(\omega)\|_{L^2(\mathbb{T})}^2\nonumber\\
&\quad\quad+\|(u_{1}-u_{2})_{x}\|_{L^2(0,\tau;L^2(\mathbb{T}))}^2\Big{)}d\tau.\label{pprr}
\end{align}
The combination of $(\ref{r2})$, (\ref{widetilderho1}), (\ref{remardfff})-(\ref{pprr}), and the Gr${\rm{\ddot{o}}}$nwall inequality gives rise to $(\rho _{1},u_{1},f_{1})=(\rho _{2},u_{2},f_{2})$ a.e. in $\mathbb{T}\times\mathbb{R}\times [0,T]$. The proof of Proposition \ref{prop31} is completed.

\end{proof}

\section{Long time behavior}

 In this section, we study the large time behavior of global solutions to the IVP (\ref{m1})-(\ref{kappa}). First, based on relative entropy estimates for the compressible Navier-Stokes equation $(\ref{m1})_{1}$-$(\ref{m1})_{2}$, we have
\begin{prop}\label{prop41}
Let $(\rho,u,f)$ be any global weak solution to the IVP $(\ref{m1})$-$(\ref{kappa})$ in the sense of Definition \ref{defn11} satisfying $(\ref{r1})$-$(\ref{massmomentum})$ for any $T>0$. Then, under the assumptions of Theorem \ref{theorem11}, it holds
\begin{equation}
\begin{split}
\lim_{t\rightarrow \infty}\Big{(}\|(\rho-\overline{\rho_{0}})(t)\|_{L^{p}(\mathbb{T})}+\|\big{(}\sqrt{\rho}(u-\overline{u})\big{)}(t)\|_{L^2(\mathbb{T})}\Big{)}=0,\quad p\in[1,\infty),\label{b21jia}
\end{split}
\end{equation}
where $\overline{\rho_{0}}$ and $\overline{u}(t)$ are given by $(\ref{rhoinfty})$.
\end{prop}
\begin{proof}
Define
 \begin{equation}\nonumber
\left\{
\begin{split}
&\mathcal{E}_{NS}^{\eta}(t):=\int_{\mathbb{T}}\big{(}\frac{1}{2}\rho|u-m_{1}|^2+\Pi_{\gamma}(\rho|\overline{\rho_{0}})\big{)}(x,t)dx-\eta\int_{\mathbb{T}}[\rho(u-m_{1})\mathcal{I}(\rho-\overline{\rho_{0}})](x,t)dx,\\
 &\mathcal{D}_{NS}^{\eta}(t):=\int_{\mathbb{T}}(\mu(\rho)|u_{x}|^2)(x,t)dx+\eta\int_{\mathbb{T}}[(\rho^{\gamma}-\overline{\rho_{0}}^{\gamma})(\rho-\overline{\rho_{0}})](x,t)dx\\
&\quad\quad\quad\quad-\eta\int_{\mathbb{T}}\big{[}\overline{\rho_{0}}\rho|u-m_{1}|^2+\mu(\rho)u_{x}(\rho-\overline{\rho_{0}})\big{]}(x,t)dx\\
&\quad\quad\quad\quad+\eta\int_{\mathbb{T}}\mathcal{I}(\rho-\overline{\rho_{0}})(x,t)\Big{(}-\int_{\mathbb{R}}\big{[}\rho \sqrt{f}\big{(} (u-v)\sqrt{f}-2(\sqrt{f})_{v}\big{)}\big{]}(x,v,t)dv\\
&\quad\quad\quad\quad+\frac{\rho(x,t)}{\overline{\rho_{0}}}\int_{\mathbb{T}\times\mathbb{R}} [\rho \sqrt{f} \big{(} (u-v)\sqrt{f}-2(\sqrt{f})_{v}](y,v,t)dvdy\Big{)}dx,
\end{split}
\right.
\end{equation}
where the operator $\mathcal{I}$ is denoted as (\ref{j}), and $m_{1}(t)$ and $\Pi_{\gamma}(\rho|\overline{\rho_{0}})$ are given by
\begin{equation}\label{m1m2pp}
\left\{
\begin{split}
&m_{1}(t):=\frac{\int_{\mathbb{T}}\rho u(x,t)dx}{\int_{\mathbb{T}}\rho_{0}(x)dx},\\
&\Pi_{\gamma}(\rho|\overline{\rho_{0}}):=\frac{\rho^{\gamma}}{\gamma-1}-\frac{\overline{\rho_{0}}^{\gamma}}{\gamma-1}-\frac{\gamma\overline{\rho_{0}}^{\gamma-1}}{\gamma-1}(\rho-\overline{\rho_{0}}).
\end{split}
\right.
\end{equation}
For a.e. $0\leq s\leq t<\infty$ (including $s=0$), one can show
\begin{equation}\label{entropyfluid22}
\begin{aligned}
&\mathcal{E}_{NS}^{\eta}(t)+\int_{s}^{t}\mathcal{D}_{NS}^{\eta}(\tau)d\tau\\
&\quad=\mathcal{E}_{NS}^{\eta}(s)++\int_{s}^{t}\int_{\mathbb{T}\times\mathbb{R}}\rho (u-v)f (u-m_{1})dxdv\\
&\quad\leq \mathcal{E}_{NS}^{\eta}(s)+\int_{s}^{t}\int_{\mathbb{T}\times\mathbb{R}} \big{[}\rho\sqrt{f} \big{|}  u-m_{1}\big{|} \big{|} (u-v)\sqrt{f}-2(\sqrt{f})_{v} \big{|} \big{]}(x,v,t)dvdx.
\end{aligned}
\end{equation}
Choosing a suitably small constant $\eta>0$, we can deduce after direct computations that
\begin{equation}\label{ENSsim}
\left\{
\begin{split}
&\frac{1}{C}\mathcal{E}_{NS}^{\eta}(t)\leq \|\big{(}\sqrt{\rho}(u-m_{1})\big{)}(t)\|_{L^2(\mathbb{T})}^2+\|(\rho-\overline{\rho_{0}})(t)\|_{L^2(\mathbb{T})}^2\leq C\mathcal{E}_{NS}^{\eta}(t),\\
&\mathcal{D}^{\eta}_{NS}(t) \geq  \frac{1}{2}\|u_{x}(t)\|_{L^2(\mathbb{T})}^2+\frac{1}{C}\mathcal{E}_{NS}^{\eta}(t)-C\|\big{[}\sqrt{\rho}\big{(} (u-v)\sqrt{f}-2(\sqrt{f})_{v}\big{)}\big{]}(t)\|_{L^2(\mathbb{T}\times\mathbb{R})}^2,
\end{split}
\right.
\end{equation}
where $C>1$ is a constant independent of time. The interested reader can refer to \cite{lhl1,lhl2} for details.

Then, by virtue of $(\ref{r1})$-$(\ref{massmomentum})$,  and the fact
\begin{equation}\label{um1infty}
\begin{split}
&(u-m_{1})(x,t)=\frac{\int_{\mathbb{T}} \rho(y,t)\int^{x}_{y}u_{z}(z,t)dzdy}{\int_{\mathbb{T}} \rho(y,t)dy}\leq \|u_{x}(t)\|_{L^2(\mathbb{T})},\quad (x,t)\in\mathbb{T}\times[0,T]
\end{split}
\end{equation}
 it holds
\begin{align}
&\int_{s}^{t}\int_{\mathbb{T}\times\mathbb{R}} \big{[}\rho\sqrt{f} \big{|}  u-m_{1}\big{|} \big{|} (u-v)\sqrt{f}-2(\sqrt{f})_{v} \big{|} \big{]}(x,v,t)dvdx \nonumber\\
&\quad\leq \rho_{+}^{\frac{1}{2}}\int_{s}^{t}\|(u-m_{1})(\tau)\|_{L^{\infty}(\mathbb{T})}\|f(\tau)\|_{L^1(\mathbb{T}\times\mathbb{R})}^{\frac{1}{2}}\|\big{[}\sqrt{\rho}\big{(} (u-v)\sqrt{f}-2(\sqrt{f})_{v}\big{)}\big{]}(\tau)\|_{L(\mathbb{T}\times\mathbb{R})}d\tau\nonumber\\
&\quad\leq\frac{1}{2}\int_{s}^{t}\|u_{x}(\tau)\|_{L^2(\mathbb{T})}^2d\tau+\frac{\rho_{+}\|n_{0}\|_{L^1(\mathbb{T})}}{2}\int_{s}^{t}\|\big{[}\sqrt{\rho}\big{(} (u-v)\sqrt{f}-2(\sqrt{f})_{v}\big{)}\big{]}(\tau)\|_{L(\mathbb{T}\times\mathbb{R})}^2d\tau,\label{pp3}
\end{align}
which together with (\ref{entropyinequality}) and (\ref{entropyfluid22}) implies for a.e. $0\leq s\leq t<\infty$ that
\begin{equation}\label{ENS1p}
\left\{
\begin{split}
&\int_{0}^{\infty}\mathcal{E}_{NS}^{\eta}(\tau)d\tau<\infty,\\
&\mathcal{E}_{NS}^{\eta}(t)\leq \mathcal{E}_{NS}^{\eta}(s)+\frac{\rho_{+}\|n_{0}\|_{L^1(\mathbb{T})}+2C}{2}\int_{s}^{t}\|\big{[}\sqrt{\rho}\big{(} (u-v)\sqrt{f}-2(\sqrt{f})_{v}\big{)}\big{]}(\tau)\|_{L(\mathbb{T}\times\mathbb{R})}^2d\tau.
\end{split}
\right.
\end{equation}
Integrating $(\ref{ENS1p})_{2}$ with respect to $s$ over $ [t-1,t]$, we derive
\begin{equation}\label{00oo}
\begin{split}
&\mathcal{E}_{NS}^{\eta}(t)\leq \int_{t-1}^{t}\mathcal{E}_{NS}^{\eta}(s)ds+\frac{\rho_{+}\|n_{0}\|_{L^1(\mathbb{T})}+2C}{2}\int_{t-1}^{t}\|\big{[}\sqrt{\rho}\big{(} (u-v)\sqrt{f}-2(\sqrt{f})_{v}\big{)}\big{]}(\tau)\|_{L(\mathbb{T}\times\mathbb{R})}^2d\tau\\
&\quad\quad\quad\rightarrow 0,\quad \text{as} ~t\rightarrow \infty,
\end{split}
\end{equation}
where we have used (\ref{entropyinequality}) and $(\ref{ENS1p})$. By $(\ref{ENSsim})_{1}$ and (\ref{00oo}), (\ref{b21jia}) holds. The proof of Proposition \ref{prop41} is completed.
\end{proof}

Then, we derive
\begin{prop}\label{prop42}
Let $(\rho,u,f)$ be any global weak solution to the IVP $(\ref{m1})$-$(\ref{kappa})$ in the sense of Definition \ref{defn11} satisfying $(\ref{r1})$-$(\ref{massmomentum})$  for any $T>0$. Then, under the assumptions of Theorem \ref{theorem12}, it holds
\begin{equation}
\left\{
\begin{split}
&~~\rho(x,t)\geq \rho_{-},\quad (x,t)\in\mathbb{T}\times [0,T],\\
&\sup_{t\in[0,T]}\|\rho_{x}(t)\|_{L^2(\mathbb{T})}\leq C,\label{lowerweak}
\end{split}
\right.
\end{equation}
where $\rho_{-}>0$ and $C>0$ are two constant independent of the time $T>0$.

Moreover, we have 
\begin{align}
&\lim_{t\rightarrow \infty}\Big{(}\|(\rho-\overline{\rho_{0}})(t)\|_{L^{\infty}(\mathbb{T})}+\|(u-u_{c})(t)\|_{L^2(\mathbb{T})}+|(m_{i}-u_{c})(t)|\Big{)}=0,\quad i=1,2,\label{b21}
\end{align}
where $\overline{\rho_{0}}$, $u_{c}$, and $m_{1}(t)$ are given by $(\ref{overrho})$, $(\ref{rhoinfty})$, and $(\ref{m1m2pp})_{1}$ respectively, and $m_{2}(t),$ is denoted by
\begin{equation}\label{m1m2}
\begin{split}
&m_{2}(t):=\frac{\int_{\mathbb{T}}n w(x,t)dx}{\int_{\mathbb{T}}n_{0}(x)dx}.
\end{split}
\end{equation}
\end{prop}

\begin{proof}
Repeating the arguments as in \cite{lhl2} with few modifications, we can show the uniform estimates $(\ref{lowerweak})$. To get (\ref{b21}), we first prove $|(m_{1}-m_{2})(t)|^2\in L^1(\mathbb{R}_{+})$. Indeed, a straightforward computation yields
\begin{equation}\nonumber
\begin{split}
&\|\big{[}\sqrt{\rho}\big{(}(u-v)\sqrt{f}-2(\sqrt{f})_{v}\big{)}\big{]}(t)\|_{L^2(\mathbb{T}\times\mathbb{R})}^2\\
&\quad\geq \rho_{-}\|\big{[}(u-m_{1})\sqrt{f}+(m_{1}-m_{2})\sqrt{f}+(m_{2}-v)\sqrt{f}-2(\sqrt{f})_{v}\big{]}(t)\|_{L^2(\mathbb{T}\times\mathbb{R})}^2\\
&\quad=\rho_{-}\|\big{(}\sqrt{f}(u-m_{1})\big{)}(t)\|_{L^2(\mathbb{T}\times\mathbb{R})}^2+\rho_{-}|(m_{1}-m_{2})(t)|^2\|f(t)\|_{L^1(\mathbb{T}\times\mathbb{R})}^{\frac{1}{2}}\\
&\quad\quad+\rho_{-}\|\big{(}(m_{2}-v)\sqrt{f}-2(\sqrt{f})_{v}\big{)}(t)\|_{L^2(\mathbb{T}\times\mathbb{R})}^2\\
&\quad\quad+2\rho_{-}(m_{1}-m_{2})(t)\int_{\mathbb{T}\times\mathbb{R}} [f (u-m_{1})](x,v,t)dvdx\\
&\quad\quad+2\rho_{-}\int_{\mathbb{T}\times\mathbb{R}} \big{[}\sqrt{f} (u-m_{1})((m_{2}-v)\sqrt{f}-(\sqrt{f})_{v})\big{]}(x,v,t)dvdx,
\end{split}
\end{equation}
which together with $(\ref{mass})$ and the fact $\|(u-m_{1})(t)\|_{L^{\infty}(\mathbb{T})}\leq \|u_{x}(t)\|_{L^2(\mathbb{T})}$ gives rise to
\begin{equation}\nonumber
\begin{split}
&\|\big{[}\sqrt{\rho}\big{(}(u-v)\sqrt{f}-2(\sqrt{f})_{v}\big{)}\big{]}(t)\|_{L^2(\mathbb{T}\times\mathbb{R})}^2\\
&\quad\geq -3\rho_{-}\|n_{0}\|_{L^1(\mathbb{T})}\|u_{x}(t)\|_{L^2(\mathbb{T})}^2+\frac{\rho_{-}\|n_{0}\|_{L^1(\mathbb{T})}}{2}|(m_{1}-m_{2})(t)|^2\\
&\quad\quad+\frac{\rho_{-}}{2}\|\big{(}(m_{2}-v)\sqrt{f}-2(\sqrt{f})_{v}\big{)}(t)\|_{L^2(\mathbb{T}\times\mathbb{R})}^2.
\end{split}
\end{equation}
Therefore, we have
\begin{align}
&\int_{0}^{\infty}|(m_{1}-m_{2})(t)|^2dt\nonumber\\
&\quad\leq \frac{2(3\rho_{-}\|n_{0}\|_{L^1(\mathbb{T})}+1)}{\|n_{0}\|_{L^1(\mathbb{T})}\rho_{-}}\int_{0}^{\infty}\Big{(}\|u_{x}(t)\|_{L^2(\mathbb{T})}^2+\|\big{[}\sqrt{\rho}\big{(}(u-v)\sqrt{f}-2(\sqrt{f})_{v}\big{)}\big{]}(t)\|_{L^2(\mathbb{T}\times\mathbb{R})}^2\Big{)}dt.\label{m1m2leq}
\end{align}

Next, it can be verified by $(\ref{m1})_{2}$-$(\ref{m1})_{3}$ that
 \begin{align}
&\frac{1}{2}\frac{d}{dt}|(m_{2}-m_{1})(t)|^2\nonumber\\
&\quad=C_{1}(m_{2}-m_{1})(t)\int_{\mathbb{T}\times\mathbb{R}}\Big{[}\rho\sqrt{f}\big{(} (u-v)\sqrt{f}-2(\sqrt{f})_{v}\big{)}\big{]}(x,v,t)dvdx,\label{m1m2t}
\end{align}
where the constant $C_{1}>0$ is given by 
$$
C_{1}:=\frac{\|\rho_{0}\|_{L^1(\mathbb{T})}+\|n_{0}\|_{L^1(\mathbb{T})}}{\|\rho_{0}\|_{L^1(\mathbb{T})}\|n_{0}\|_{L^1(\mathbb{T})}}.
$$
For any $0\leq s\leq t<\infty$, we integrate (\ref{m1m2t}) over $[s,t]$ to derive
 \begin{align}
&\frac{1}{2}|(m_{1}-m_{2})(t)|^2\nonumber\\
&\quad\leq \frac{1}{2}|(m_{1}-m_{2})(s)|^2+\rho_{+}^{\frac{1}{2}}\|n_{0}\|_{L^1(\mathbb{T})}^{\frac{1}{2}}C_{1}\int_{s}^{t}\Big{(} |(m_{1}-m_{2})(\tau)|^2\nonumber\\
&\quad\quad+\|\big{[}\sqrt{\rho}\big{(}(u-v)\sqrt{f}-2(\sqrt{f})_{v}\big{)}\big{]}(\tau)\|_{L^2(\mathbb{T}\times\mathbb{R})}^2\Big{)}d\tau,\label{whoswho}
 \end{align}
and thence integrate (\ref{whoswho}) with respect to $s$ over $ [t-1,t]$ to have
\begin{align}
&\frac{1}{2}|(m_{1}-m_{2})(t)|^2\nonumber\\
&\quad\leq \frac{1}{2}\int_{t-1}^{t}|(m_{1}-m_{2})(s)|^2ds+\rho_{+}^{\frac{1}{2}}\|n_{0}\|_{L^1(\mathbb{T})}^{\frac{1}{2}}C_{1}\int_{t-1}^{t}\Big{(}|(m_{1}-m_{2})(\tau)|^2\nonumber\\
&\quad\quad+\|\big{[}\sqrt{\rho}\big{(}(u-v)\sqrt{f}-2(\sqrt{f})_{v}\big{)}\big{]}(\tau)\|_{L^2(\mathbb{T}\times\mathbb{R})}^2\Big{)}d\tau.\label{whoswho1}
\end{align}
Due to $(\ref{entropyinequality})$ and (\ref{m1m2leq}), the right-hand side of (\ref{whoswho1}) tends to $0$ as $t\rightarrow \infty$. Thus, it follows
\begin{equation}\label{m1m2behavior}
\begin{split}
\lim_{t\rightarrow\infty}|(m_{1}-m_{2})(t)|^2=0.
\end{split}
\end{equation}

Finally, it holds by the conservation laws $(\ref{massmomentum})$ that
$$
m_{1}(t)\int_{\mathbb{T}}\rho_{0}(x)dx+m_{2}(t)\int_{\mathbb{T}}n_{0}(x)dx=\int_{\mathbb{T}}(m_{0}+j_{0})(x)dx,
$$
which leads to
\begin{equation}\label{m1m2m1m2}
\begin{split}
&|(m_{1}-m_{2})(t)|=\frac{\int_{\mathbb{T}}(\rho_{0}+n_{0})(x)dx}{\int_{\mathbb{T}}\rho_{0}(x)dx}|(m_{2}-u_{c})(t)|,
\end{split}
\end{equation}
where the constant $u_{c}$ is given by $(\ref{rhoinfty})$. One deduces by (\ref{m1m2behavior})-(\ref{m1m2m1m2}) that
\begin{align}
&\lim_{t\rightarrow\infty}\big{(}|(m_{1}-u_{c})|+|(m_{2}-u_{c})(t)|\big{)}\leq\lim_{t\rightarrow\infty}\big{(} |(m_{1}-m_{2})(t)|+2|(m_{2}-u_{c})(t)|\big{)}=0.\label{m11m22}
\end{align}
Combining (\ref{b21jia})-(\ref{lowerweak}) and (\ref{m11m22}) together, we have
\begin{equation}\nonumber
\begin{split}
&\lim_{t\rightarrow\infty}\|(u-u_{c})(t)\|_{L^2(\mathbb{T})}\leq \lim_{t\rightarrow\infty}\big{(}\frac{1}{\rho_{-}^{\frac{1}{2}}}\|\big{(}\sqrt{\rho}(u-m_{1})\big{)}(t)\|_{L^2(\mathbb{T})}+|(m_{1}-u_{c})(t)|\big{)}=0.
\end{split}
\end{equation}
It follows from $(\ref{r1})_{1}$, $(\ref{b21jia})$, and the Gagliardo-Nirenberg inequality that
\begin{equation}\nonumber
\begin{split}
&\lim_{t\rightarrow\infty}\|(\rho-\overline{\rho_{0}})(t)\|_{L^{\infty}(\mathbb{T})}\leq \sqrt{2}\lim_{t\rightarrow\infty}\|\rho_{x}(t)\|_{L^2(\mathbb{T})}^{\frac{1}{2}}\|(\rho-\overline{\rho_{0}})(t)\|_{L^2(\mathbb{T})}^{\frac{1}{2}}= 0.
\end{split}
\end{equation}
 The proof of Proposition \ref{prop42} is completed.
\end{proof}

Furthermore, the time convergence of the distribution function $f$ to the global Maxwellian $M_{\overline{n_{0}},u_{c}}(v)$ is proved based on the ideas inspired by \cite{bouchut1} and the compactness techniques developed in \cite{arsenio1}.

\begin{prop}\label{prop43}
Let $(\rho,u,f)$ be any global weak solution to the IVP $(\ref{m1})$-$(\ref{kappa})$ in the sense of Definition \ref{defn11} satisfying $(\ref{r1})$-$(\ref{massmomentum})$  for any $T>0$. Then, under the assumptions of Theorem \ref{theorem12}, it holds
\begin{equation}
\left\{
\begin{split}
&\lim_{t\rightarrow \infty} \|(f-M_{\overline{n_{0}},u_{c}})(t)\|_{L^1(\mathbb{T}\times\mathbb{R})}=0,\\
&\lim_{t\rightarrow \infty}\Big{(}\|(n-\overline{n_{0}})(t)\|_{L^1(\mathbb{T})}+\|\big{(}n(w-u_{c})\big{)}(t)\|_{L^1(\mathbb{T})}\Big{)}=0,\label{b22}
\end{split}
\right.
\end{equation}
where  the global Maxwellian $M_{\overline{n_{0}},u_{c}}(v)$ is denoted through $(\ref{M})$, and the constants $u_{c}$ and $\overline{n_{0}}$ are given by $(\ref{rhoinfty})$ and $(\ref{n0})$ respectively.
\end{prop}
\begin{proof}

For any $(x,v,t)\in \mathbb{T}\times\mathbb{R}\times(0,1)$ and $s\geq 0$, define
$$
\rho ^{s}(x,t):=\rho (x,t+s),\quad u^{s}(x,t):=u(x,t+s),\quad f^{s}(x,v,t):=f(x,v,t+s),
$$
so that it satisfies
\begin{equation}
\begin{split}
(f^{s})_{t}+v(f^{s})_{x}+\rho^{s}[(u^{s}-v)f^{s}-(f^{s})_{v}]_{v}=0.\label{fs}
\end{split}
\end{equation}
By $(\ref{entropyinequality})$, $(\ref{r1})_{1}$, (\ref{massmomentum}), and $(\ref{lowerweak})$, we have 
\begin{equation}\label{uniformf}
\left\{
\begin{split}
&\|\rho^{s}(t)\|_{L^1(\mathbb{T})}=\|\rho_{0}\|_{L^1(\mathbb{T})},\quad \|f^{s}(t)\|_{L^1(\mathbb{T}\times\mathbb{R})}=\|n_{0}\|_{L^1(\mathbb{T})},\\
&0<\rho_{-}\leq \rho^{s}(x,t)\leq \rho_{+}<\infty,\\
&\sup_{s\geq0}\Big{(}\|f^s\|_{L^{\infty}(0,1;L^1_{2}(\mathbb{T}\times\mathbb{R}))}+ \|f^{s}\log{f^{s}}\|_{L^{\infty}(0,1;L^1(\mathbb{T}\times\mathbb{R}))}+\|u^{s}\|_{L^{2}(0,1;H^{1}(\mathbb{T}))}\\
&\quad~~+\|(u^{s}-v)\sqrt{f^{s}}-2(\sqrt{f^{s}})_{v}\|_{L^2(0,1;L^2(\mathbb{T}\times\mathbb{R}))}\Big{)}<\infty.
\end{split}
\right.
\end{equation}
Thence, one deduces from $(\ref{basicCNSVFP})_{4}$, (\ref{b21jia}), (\ref{b21}), and (\ref{uniformf}) as $s\rightarrow \infty$ that
\begin{equation}
\left\{
\begin{split}
&\rho^{s}\overset{\ast}{\rightharpoonup} \overline{\rho_{0}}\quad\text{ in}~ L^{\infty}(0,1;L^{\infty}(\mathbb{T})),\\
&u^{s}\rightarrow u_{c}\quad\text{ in}~L^2(0,1;H^{1}(\mathbb{T}))\hookrightarrow L^2(0,1;L^{\infty}(\mathbb{T})).\label{strongcompactrhou}
\end{split}
\right.
\end{equation}
By (\ref{uniformf}) and the fact
\begin{equation}\nonumber
\begin{split}
&-4\int_{0}^{1}\int_{\mathbb{T}\times\mathbb{R}}[(u^{s}-v)\sqrt{f^{s}}(\sqrt{f^{s}})_{v}](x,v,t)dvdxdt= -2\|n_{0}\|_{L^1(\mathbb{T})},
\end{split}
\end{equation}
we have
\begin{equation}\nonumber
\begin{split}
&\|(u^{s}-v)\sqrt{f^{s}}-2(\sqrt{f^{s}})_{v}\|_{L^2(0,1;L^2(\mathbb{T}\times\mathbb{R}))}^2\\
&\quad=\|(u^{s}-v)\sqrt{f^{s}}\|_{L^2(0,1;L^2(\mathbb{T}\times\mathbb{R}))}^2+4\|(\sqrt{f^{s}})_{v}\|_{L^2(0,1;L^2(\mathbb{T}\times\mathbb{R}))}^2-2\|n_{0}\|_{L^1(\mathbb{T})},
\end{split}
\end{equation}
which implies
\begin{equation}\label{fsv}
\begin{split}
\sup_{s\geq0}\|(\sqrt{f^{s}})_{v}\|_{L^2(0,1;L^2(\mathbb{T}\times\mathbb{R}))}<\infty.
\end{split}
\end{equation}

We are ready to establish the strong $L^1(\mathbb{T}\times\mathbb{R})$-compactness of $f^{s}(x,v,t)$ for a.e. $t\in(0,1)$. It follows from (\ref{fs}) and Lemma \ref{lemma610} for any $\delta\in (0,1)$ and $s\in (0,\infty)$ that
\begin{equation}\label{sqrtfs}
\begin{split}
&(\sqrt{f^{s}+\delta})_{t}+v(\sqrt{f^{s}+\delta})_{x}+[\rho^{s}(u^{s}-v)(\sqrt{f^{s}+\delta})]_{v}-[\rho^{s}(\sqrt{f^{s}+\delta})_{v}]_{v}\\
&\quad\quad+\frac{\rho^{s}(f^{s}+2\delta)}{2\sqrt{f^{s}+\delta}}-\frac{\rho^{s}|(f^{s})_{v}|^2}{4(f^{s}+\delta)^{\frac{3}{2}}}=0, \quad\text{in}~\mathcal{D}'(\mathbb{T}\times\mathbb{R}\times(0,1)).
\end{split}
\end{equation}
In order to employ Lemma \ref{lemma67} below, for any $(x,v,t)\in\mathbb{R}^3$, we prolong $(\rho^{s},u^{s},f^{s})$ to be zero outside $(0,1)\times\mathbb{R}\times(0,1)$:
\begin{equation}\nonumber
\begin{split}
&(\widetilde{\rho^{s}}(x,t),\widetilde{u^{s}}(x,t),\widetilde{f^{s}}(x,v,t)):=
\begin{cases}
(\rho^{s}(x,t),u^{s}(x,t),f^{s}(x,v,t)),
& \mbox{ $(x,v,t)\in (0,1)\times\mathbb{R}\times (0,1),$ } \\
0,
& \mbox{ $(x,v,t) \notin (0,1)\times\mathbb{R}\times (0,1).$}
\end{cases}
\end{split}
\end{equation}
 According to the uniform estimates (\ref{uniformf}) and Lemma \ref{lemma64}, there is a limit $f^{\infty}$ satisfying $f^{\infty}\in L^{\infty}(0,1;L^1_{2}(\mathbb{T}\times\mathbb{R}))$ and $(\sqrt{f^{\infty}})_{v}\in L^2(0,1;L^2(\mathbb{T}\times\mathbb{R}))$ such that it holds as $s\rightarrow\infty$,
\begin{equation}\label{weakf}
\left\{
\begin{split}
&\widetilde{f^{s}}\rightharpoonup \widetilde{f^{\infty}}\quad\text{in $L^1(\mathbb{R}^2)$}\quad \text{ a.e.}~t\in \mathbb{R},\\
&(\sqrt{\widetilde{f^{s}}})_{v}\rightharpoonup (\sqrt{\widetilde{f^{\infty}}})_{v}\quad\text{ in $L^2(\mathbb{R}^3)$},
\end{split}
\right.
\end{equation}
where $\widetilde{f^{\infty}}$ is given by
\begin{eqnarray}\nonumber
\widetilde{f^{\infty}}(x,v,t):=
\begin{cases}
f^{\infty}(x,v,t),
& \mbox{ $(x,v,t)\in (0,1)\times\mathbb{R}\times (0,1),$ } \\
0,
& \mbox{ $(x,v,t) \notin(0,1)\times\mathbb{R}\times (0,1).$}
\end{cases}
\end{eqnarray}
Denote
\begin{equation}\label{vartheta1}
\begin{split}
\vartheta^{s,\delta,r}(t,x,v):=\mathbf{1}_{|v|\leq r}\sqrt{\widetilde{f^{s}}(x,v,t)+\delta},\quad s>0, \quad \delta\in(0,1),\quad r>0,
\end{split}
\end{equation}
where $\mathbf{1}_{|v|\leq r}\in \mathcal{D}(\mathbb{R})$ stands for the cut-off function satisfying $\mathbf{1}_{|v|\leq r}=1$ for $|v|\leq r$ and $\mathbf{1}_{|v|\leq r}=0$ for $|v|\geq 2r$. Obviously, $\vartheta^{s,\delta,r}$ has compact support in $(-2r,2r)\times(0,1)\times(0,1)$, and we obtain by (\ref{sqrtfs}) that 
\begin{equation}\label{vartheta}
\begin{split}
&(\vartheta^{s,\delta,r})_{t}+v(\vartheta^{s,\delta,r})_{v}=(I-\partial_{tt}^2-\partial_{xx}^2)^{\frac{9}{22}}(I-\partial_{vv}^2)^{\frac{31}{22}}(g_{1}^{s,\delta}+g_{2}^{s,\delta})\quad\text{in}~\mathcal{D}'(\mathbb{R}^3),
\end{split}
\end{equation}
with
\begin{equation}\nonumber
\left\{
\begin{split}
&g_{1}^{s,\delta}:=(I-\partial_{tt}^2-\partial_{xx}^2)^{-\frac{9}{22}}(I-\partial_{vv}^2)^{-\frac{31}{22}}\partial_{vv}^2(\widetilde{\rho^{s}}\vartheta^{s,\delta,r}),\\
&g_{2}^{s,\delta}:=(I-\partial_{tt}^2-\partial_{xx}^2)^{-\frac{9}{22}}(I-\partial_{vv}^2)^{-\frac{31}{22}}\Big{(}-(\mathbf{1}_{|v|\leq r})_{v}\widetilde{\rho^s}(\widetilde{u^s}-v)\sqrt{\widetilde{f^s}+\delta}\\
&\quad\quad\quad-(\mathbf{1}_{|v|\leq r})_{vv}\widetilde{\rho^{s}}\sqrt{\widetilde{f^{s}}+\delta}-2(\mathbf{1}_{|v|\leq r})_{v}\widetilde{\rho^s} (\sqrt{\widetilde{f^s}+\delta})_{v}-\frac{\mathbf{1}_{|v|\leq r}\widetilde{\rho^{s}}(\widetilde{f^{s}}+2\delta)}{2\sqrt{\widetilde{f^{s}}+\delta}}+\frac{\mathbf{1}_{|v|\leq r}\widetilde{\rho^{s}}|(\widetilde{f^{s}})_{v}|^2}{4(\widetilde{f^{s}}+\delta)^{\frac{3}{2}}}\Big{)}.
\end{split}
\right.
\end{equation}
By virtue of (\ref{uniformf}) and Lemmas \ref{lemma67}-\ref{lemma68}, we have
\begin{align}
&\|g_{1}^{s,\delta}\|_{L^{\frac{5}{4}}(\mathbb{R}^3)}\leq C_{r}\|(I-\partial_{tt}^2-\partial_{xx}^2)^{-\frac{9}{22}}(I-\partial_{vv}^2)^{-\frac{9}{22}}(\widetilde{\rho^{s}}\vartheta^{s,\delta,r})\|_{L^{\frac{5}{4}}(\mathbb{R}^3)}\nonumber\\
&\quad\quad\quad\quad\quad~\leq C_{r}\|(I-\partial_{tt}^2-\partial_{xx}^2-\partial_{vv}^2)^{-\frac{9}{22}}(\widetilde{\rho^{s}}\vartheta^{s,\delta,r})\|_{L^{\frac{11}{8},\infty}(\mathbb{R}^3)}\nonumber\\
&\quad\quad\quad\quad\quad~\leq C_{r}\|\widetilde{\rho^{s}}\vartheta^{s,\delta,r}\|_{L^{1}(\mathbb{R}^3)}\leq C_{r}\|\rho^{s}\|_{L^{\infty}(0,1;L^{\infty}(\mathbb{T}))}\|f^{s}\|_{L^{\infty}(0,1;L^1(\mathbb{T}\times\mathbb{R}))}^{\frac{1}{2}}\leq C_{r},\label{g1e}
\end{align}
where $C_{r}>0$ denotes a sufficiently large constant only dependent of $r>0$.  Similarly, it follows by (\ref{uniformf}), (\ref{fsv}), and Lemmas \ref{lemma67}-\ref{lemma68} that
\begin{align}
&\|g_{2}^{s,\delta}\|_{L^{\frac{5}{4}}(\mathbb{R}^3)}\leq C_{r}\|\rho^{s}\|_{L^{\infty}(0,1;L^{\infty}(\mathbb{T}))}\big{(}\|u^{s}\|_{L^2(0,1;L^{\infty}(\mathbb{T}))}+1\big{)}\Big{(}\|f^{s}\|_{L^{\infty}(0,1;L^1(\mathbb{T}\times\mathbb{R}))}^{\frac{1}{2}}\nonumber\\
&\quad\quad\quad\quad\quad~\quad+\|(\sqrt{f^{s}})_{v}\|_{L^2(0,1;L^2(\mathbb{T}\times\mathbb{R}))}+\delta^{-\frac{1}{2}}\|(\sqrt{f^{s}})_{v}\|_{L^2(0,1;L^2(\mathbb{T}\times\mathbb{R}))}^2\Big{)}\leq C_{r}.\label{g2e}
\end{align}
Due to the uniform estimate (\ref{fsv}), the sequence $\vartheta^{s,\delta,r}$ given by (\ref{vartheta}) is locally strongly compact with respect to $v$ in $L^{\frac{5}{4}}(\mathbb{R}^3)$. Thus, we make use of Lemma \ref{lemma66} below, (\ref{uniformf}), (\ref{fsv}), $(\ref{weakf})$, and (\ref{g1e})-(\ref{g2e}) to derive
\begin{equation}\nonumber
\begin{split}
&\vartheta^{s,\delta,r} \rightarrow \mathbf{1}_{|v|\leq r}\sqrt{\widetilde{f^{\infty}}+\delta}\quad\text{in}~L^{\frac{5}{4}}_{loc}(\mathbb{R}^3),\quad\delta\in(0,1),\quad r>0,\quad \text{as}~s\rightarrow\infty,
\end{split}
\end{equation}
which implies for any $\delta\in(0,1)$ and $r>0$ that
\begin{equation}\label{strongvartheta}
\begin{split}
 \mathbf{1}_{|v|\leq r}\sqrt{f^{s}+\delta}\rightarrow  \mathbf{1}_{|v|\leq r}\sqrt{f^{\infty}+\delta}\quad\text{in}~L^{\frac{5}{4}}(\mathbb{T}\times\mathbb{R}\times(0,1)),\quad \text{as}~s\rightarrow\infty.
\end{split}
\end{equation}
In addition, we have
\begin{align}
&\|\sqrt{f^{s}}-\sqrt{f^{\infty}}\|_{L^1(\mathbb{T}\times\mathbb{R}\times(0,1))}\nonumber\\
&\quad\leq \| \mathbf{1}_{|v|\leq r}(\sqrt{f^{s}+\delta}- \sqrt{f^{\infty}+\delta})\|_{L^1(\mathbb{T}\times\mathbb{R}\times(0,1))}+\|\sqrt{f^{s}}- \mathbf{1}_{|v|\leq r}\sqrt{f^{s}+\delta}\|_{L^1(\mathbb{T}\times\mathbb{R}\times(0,1))}\nonumber\\
&\quad\quad+\| \mathbf{1}_{|v|\leq r}\sqrt{f^{\infty}+\delta}-\sqrt{f^{\infty}}\|_{L^1(\mathbb{T}\times\mathbb{R}\times(0,1))}.\label{sqrtflimit1}
\end{align}
To control the right-hand side of (\ref{sqrtflimit1}), we deduce by (\ref{uniformf}) and (\ref{weakf})-(\ref{vartheta1}) that
\begin{align}
&\|\sqrt{f^{s}}- \mathbf{1}_{|v|\leq r}\sqrt{f^{s}+\delta}\|_{L^1(\mathbb{T}\times\mathbb{R}\times(0,1))}\nonumber\\
&\quad= \| (1-\mathbf{1}_{|v|\leq r}) \sqrt{f^{s}}+\mathbf{1}_{|v|\leq r}( \sqrt{f^{s}}-\sqrt{f^{s}+\delta})\|_{L^1(\mathbb{T}\times\mathbb{R}\times(0,1))}\nonumber\\
&\quad\leq \frac{1}{r}\|\langle v\rangle \sqrt{f^{s}}\|_{L^1(0,1;L^1(\mathbb{T}\times\mathbb{R}))}+ \|\mathbf{1}_{|v|\leq r} \frac{\delta}{\sqrt{f^{s}}+\sqrt{f^{s}+\delta}}\|_{L^1(\mathbb{T}\times\mathbb{R}\times(0,1))}\nonumber\\
&\quad\leq \frac{1}{r}\Big{(}\int_{\mathbb{R}}\frac{1}{\langle v\rangle^2}dv\Big{)}^{\frac{1}{2}}\|f^{s}\|_{L^{\infty}(0,1;L^1_{2}(\mathbb{T}\times\mathbb{R}))}^{\frac{1}{2}}+C_{r}\sqrt{\delta}\leq C_{r}\sqrt{\delta}+\frac{C_{2}}{r},\label{bbdr}
\end{align}
where $C_{2}>0$ stands for a constant uniformly in $s, \delta$, and $r$, and $C_{r}>0$ is a constant only dependent of $r>0$. Similarly, it holds
\begin{equation}\label{bbdr1}
\begin{split}
&\|\sqrt{f^{\infty}}- \mathbf{1}_{|v|\leq r}\sqrt{f^{\infty}+\delta}\|_{L^1(\mathbb{T}\times\mathbb{R}\times(0,1))}\leq C_{r}\sqrt{\delta}+\frac{C_{2}}{r}.
\end{split}
\end{equation}
Substituting (\ref{bbdr})-(\ref{bbdr1}) into (\ref{sqrtflimit1}) and using  (\ref{strongvartheta}), we obtain
\begin{equation}\nonumber
\begin{split}
&\lim_{s\rightarrow\infty}\|\sqrt{\widetilde{f^{s}}}-\sqrt{\widetilde{f^{\infty}}}\|_{L^1(\mathbb{R}^3)}\\
&\quad\leq\lim_{r\rightarrow\infty}\lim_{\delta\rightarrow0}\lim_{s\rightarrow\infty}\Big{(}C_{2}r^{\frac{1}{5}}\| \mathbf{1}_{|v|\leq r}(\sqrt{f^{s}+\delta}- \sqrt{f^{\infty}+\delta})\|_{L^{\frac{5}{4}}(\mathbb{R}^3)}+C_{r}\sqrt{\delta}+\frac{C_{2}}{r}\Big{)}=0,
\end{split}
\end{equation}
which gives rise to
\begin{equation}\label{fae}
\begin{split}
f^{s}\rightarrow f^{\infty}\quad\text{a.e. in}~\mathbb{T}\times\mathbb{R}\times(0,1),\quad\text{as}~s\rightarrow\infty.
\end{split}
\end{equation}
By virtue of (\ref{weakf}), (\ref{fae}), and Lemma \ref{lemma65} below, we have
\begin{equation}\label{strongf}
\begin{split}
&f^{s}(x,v,t)\rightarrow f^{\infty}(x,v,t)\quad \text{ in}~ L^1(\mathbb{T}\times\mathbb{R}),\quad \text{a.e.}~t\in(0,1),\quad\text{as}~s\rightarrow\infty.
\end{split}
\end{equation}

We are going to show $f^{\infty}(x,v,t)=M_{\overline{n_{0}}, u_{c}}(v)$. By (\ref{entropyinequality}), $(\ref{r1})_{1}$, and $(\ref{massmomentum})_{2}$, we have 
\begin{align}
&\lim_{t\rightarrow\infty}\|\rho^{s}(u^{s}-v)f^{s}-\rho^{s}(f^{s})_{v}\|_{L^1(0,1;L^1(\mathbb{T}\times\mathbb{R}))}\nonumber\\
&\quad=\lim_{t\rightarrow\infty}\|\rho(u-v)f-\rho f_{v}\|_{L^1(s,s+1;L^1(\mathbb{T}\times\mathbb{R}))}\nonumber\\
&\quad\leq \rho_{+}^{\frac{1}{2}}\|f_{0}\|_{L^1(\mathbb{T}\times\mathbb{R})}^{\frac{1}{2}} \lim_{t\rightarrow\infty}\|\sqrt{\rho}\big{(}(u-v)\sqrt{f}-(\sqrt{f})_{v}\big{)}\|_{L^2(s,s+1;L^2(\mathbb{T}\times\mathbb{R}))}= 0.\label{fffff}
\end{align}
Due to (\ref{fs}), (\ref{strongcompactrhou}), (\ref{weakf}), and (\ref{fffff}), the limit $f^{\infty}$ satisfies
\begin{equation}\label{equationf}
\left\{
\begin{split}
&(f^{\infty})_{v}-(u_{c}-v)f^{\infty}= e^{-\frac{|v-u_{c}|^2}{2}}\big{(}e^{\frac{|v-u_{c}|^2}{2}}f^{\infty}\big{)}_{v}=0\quad\text{in}~\mathcal{D}'(\mathbb{T}\times\mathbb{R}\times(0,1)),\\
&(f^{\infty})_{t}+v(f^{\infty})_{x}=0\quad \text{in}~\mathcal{D}'(\mathbb{T}\times\mathbb{R}\times(0,1)).
\end{split}
\right.
\end{equation}
With the help of $(\ref{equationf})_{1}$, we deduce that exists a function $g^{\infty}=g^{\infty}(x,t)\in L^{\infty}(0,1;L^1(\mathbb{T}))$ satisfying
\begin{equation}\label{fg}
\begin{split}
g^{\infty}(x,t):=e^{\frac{|v-u_{c}|^2}{2}}f^{\infty}(x,v,t).
\end{split}
\end{equation}
And by virtue of $(\ref{equationf})_{2}$ and (\ref{fg}), we have for any $\phi\in \mathcal{D}(\mathbb{T}\times(0,1))$ and $\chi\in\mathcal{D}(\mathbb{R})$ that
\begin{equation}\label{chichi}
\begin{split}
\int_{0}^{1}\int_{\mathbb{T}\times\mathbb{R}} g^{\infty}(x,t)\big{(}\phi_{t}(x,t)+v\phi_{x}(x,t)\big{)}\chi(v) dvdxdt=0.
\end{split}
\end{equation}
By a density argument,  (\ref{chichi}) indeed holds for any $\chi\in \mathcal{S}(\mathbb{R})$. Thus, we choose $\chi(v)=e^{-|v|^2}$ and $\chi(v)=ve^{-|v|^2}$ in (\ref{chichi}) to derive
\begin{equation}\nonumber
\begin{split}
&g^{\infty}_{t}(x,t)=g^{\infty}_{x}(x,t)=0\quad\text{in}~\mathcal{D}(\mathbb{T}\times(0,1)),
\end{split}
\end{equation}
 which together with  $(\ref{massmomentum})_{2}$ shows $g^{\infty}=\frac{\overline{n_{0}}}{\sqrt{2\pi}}$. Thus, we conclude that the unique formula of $f^{\infty}$ is $M_{\overline{n_{0}},u_{c}}(v)$.

Then, we have $(\ref{b22})_{1}$. Indeed, if $f^{s}(t)$ does not converge to $M_{\overline{n_{0}},u_{c}}$ in $L^1(\mathbb{T}\times\mathbb{R})$ for any $t\in(0,1)$ as $s\rightarrow\infty$, then there is a constant $\delta>0$ such that it holds
   \begin{equation}
\begin{split}
\|(f^{s}-M_{\overline{n_{0}},u_{c}})(t)\|_{L^1(\mathbb{T}\times\mathbb{R})}\geq\delta,\quad \text{as}~s\rightarrow\infty,\nonumber
\end{split}
\end{equation}
which contradicts (\ref{strongf}).

In addition, it holds
\begin{equation}\label{nb}
\begin{split}
&\lim_{t\rightarrow\infty}\|(n-\overline{n_{0}})(t)\|_{L^1(\mathbb{T})}\leq \lim_{t\rightarrow\infty}\|(f-M_{\overline{n_{0}},u_{c}})(t)\|_{L^1(\mathbb{T}\times\mathbb{R})}=0.
\end{split}
\end{equation}
One concludes from (\ref{nw})-(\ref{entropyinequality}) and $(\ref{b22})_{1}$ that
\begin{align}
&\lim_{t\rightarrow\infty}\|(nw-\overline{n_{0}}u_{c})(t)\|_{L^1(\mathbb{T})}\leq \lim_{t\rightarrow\infty}\| |v|^2(f+M_{\overline{n_{0}},u_{c}})(t)\|_{L^1(\mathbb{T}\times\mathbb{R})}^{\frac{1}{2}}\| (f-M_{\overline{n_{0}},u_{c}})(t)\|_{L^1(\mathbb{T}\times\mathbb{R})}^{\frac{1}{2}}=0.\label{nwb}
\end{align}
The combination of $(\ref{b22})_{1}$ and $(\ref{nb})$-(\ref{nwb}) gives rise to
\begin{equation}\nonumber
\begin{split}
&\lim_{t\rightarrow\infty}\|(n|w-u_{c}|)(t)\|_{L^1(\mathbb{T})}\leq \lim_{t\rightarrow\infty}\big{(} \|(nw-\overline{n_{0}}u_{c})(t)\|_{L^1(\mathbb{T})}+|u_{c}|\| (n-\overline{n_{0}})(t)\|_{L^1(\mathbb{T})} \big{)}= 0.
\end{split}
\end{equation}
The proof of Proposition \ref{prop43} is completed.
\end{proof}

\section{Proofs of main results}

\underline{\it\textbf{Proof of Theorem \ref{theorem11}:}} \emph{~Step~1: Construction of approximate sequence.}
For any $\varepsilon\in(0,1)$, we regularize the initial data as follows:
\begin{equation}\label{weakdata}
\begin{split}
&(\rho_{0}^{\varepsilon}(x),u_{0}^{\varepsilon}(x),f_{0}^{\varepsilon}(x,v)):=(J_{1}^{\varepsilon}\ast\rho_{0}(x)+\varepsilon,\frac{J_{1}^{\varepsilon}\ast(\frac{m_{0}}{\sqrt{\rho_{0}}})(x)}{\sqrt{J_{1}^{\varepsilon}\ast\rho_{0}(x)+\varepsilon}},J_{1}^{\varepsilon}\ast J_{2}^{\varepsilon}\ast(f_{0}\mathbf{1}_{|v|\leq \varepsilon^{-1}})(x,v)),
\end{split}
\end{equation}
where $J_{1}^{\varepsilon}(x)$ and $J_{2}^{\varepsilon}(v)$ are the Friedrichs mollifier with respect to the variables $x$ and $v$, and $\mathbf{1}_{|v|\leq \varepsilon^{-1}}\in \mathcal{D}(\mathbb{R})$ is the cut-off function satisfying $\mathbf{1}_{|v|\leq \varepsilon^{-1}}=1$ for $|v|\leq r$ and $\mathbf{1}_{|v|\leq \varepsilon^{-1}}=0$ for $|v|\geq 2\varepsilon^{-1}$. It is easy to verify that $(\rho_{0}^{\varepsilon},\rho_{0}^{\varepsilon}u_{0}^{\varepsilon},f_{0}^{\varepsilon})$ satisfies the assumptions $(\ref{a1})$ of Theorem \ref{theorem11} uniformly in $\varepsilon\in(0,1)$.

We are able to obtain a local regular solution $(\rho^{\varepsilon},u^{\varepsilon},f^{\varepsilon})$ for any $\varepsilon\in(0,1)$ to the IVP $(\ref{m1})$-$(\ref{kappa})$ for the initial data (\ref{weakdata}) in a standard way based on linearization techniques \cite{lhl2}, the details are omitted.

Then, by virtue of Remark \ref{remark13}, we can extend the local regular solution $(\rho^{\varepsilon},u^{\varepsilon},f^{\varepsilon})$ to a global one.

\emph{Step~2: Compactness and convergence.}
Let $T>0$ be any given time. It follows from the a-priori estimates established in Lemmas \ref{lemma22}-\ref{lemma23} and standard arguments as in \cite{lions2} that there exist a limit $(\rho,u, f)$ such that up to a subsequence (still labeled by $(\rho^{\varepsilon},u^{\varepsilon},f^{\varepsilon})$ here and in what follows), it holds as $\varepsilon\rightarrow 0$ that
\begin{equation}
\left\{
\begin{split}
&(\rho^{\varepsilon}, f^{\varepsilon})\overset{\ast}{\rightharpoonup}(\rho,f)\quad\text{ in}~ L^{\infty}(0,T;L^{\infty}(\mathbb{T}))\times L^{\infty}(0,T;L^{\infty}(\mathbb{T}\times\mathbb{R})),\\
&u^{\varepsilon}\rightharpoonup u\quad\text{ in}~ L^{2}(0,T;H^{1}(\mathbb{T})),\\
&\rho^{\varepsilon}\rightarrow \rho\quad \text{ in}~C([0,T];L^{p}_{weak}(\mathbb{T}))\cap C([0,T];H^{-1}(\mathbb{T})),\quad p\in (1,\infty),\\
&\rho^{\varepsilon}u^{\varepsilon}\rightarrow \rho u\quad\text{ in}~ C([0,T];L^2_{weak}(\mathbb{T}))\cap C([0,T];H^{-1}(\mathbb{T})),\\
&\rho^{\varepsilon}|u^{\varepsilon}|^2\rightarrow  \rho |u|^2 \quad\text{in}~ \mathcal{D}(\mathbb{T}\times(0,T)),\label{limit1}
\end{split}
\right.
\end{equation}
where $C([0,T];X_{weak})$ is the space of weak topology on $X$ defined by
\begin{equation}\nonumber
\begin{split}
C([0,T];X_{weak}):=\big{\{} g:[0,T]\rightarrow X~\big{|}~<g,\phi>_{X,X^{*}}\in C([0,T])~\text{for any $\phi\in X^{*}$}\big{\}}.
\end{split}
\end{equation}
 In addition, it follows from Lemmas \ref{lemma22}-\ref{lemma23}, $(\ref{m1})_{3}$, and Lemma \ref{lemma63} for any $\chi(v)\in\mathcal{D}(\mathbb{R})$ satisfying $|\chi(v)|\leq C(1+|v|)$ that
\begin{equation}\label{compactmoments}
\begin{split}
&\int_{\mathbb{R}}f^{\varepsilon}\chi(v)dv\rightarrow\int_{\mathbb{R}}f \chi(v)dv\quad\text{ in}~L^{1}(0,T;L^{1}(\mathbb{T})), \quad\text{as}~\varepsilon\rightarrow0,
\end{split}
\end{equation}
which together with $(\ref{N1time})_{3}$ implies as $\varepsilon\rightarrow0$ that 
\begin{equation}\label{strongnj}
\left\{
\begin{split}
&n^{\varepsilon}:=\int_{\mathbb{R}}f^{\varepsilon}dv\rightarrow n=\int_{\mathbb{R}}fdv\quad\text{ in}~ L^{q}(0,T;L^{p}(\mathbb{T})),\quad q\in[1,\infty),~p\in [1,4),\\
&n^{\varepsilon}w^{\varepsilon}:=\int_{\mathbb{R}}f^{\varepsilon}vdv\rightarrow nw=\int_{\mathbb{R}}fvdv~\text{in} ~L^{q}(0,T;L^{p}(\mathbb{T})),\quad q\in[1,\infty),~p\in [1,2).
\end{split}
\right.
\end{equation}

Then,  we claim that it holds
 \begin{equation}
 \begin{split}
 \rho^{\varepsilon}\rightarrow  \rho\quad\text{strongly in}~L^{p}(0,T;L^{p}(\mathbb{T})),\quad \text{as}~\varepsilon\rightarrow 0,\quad p\in[1,\infty).\label{limit7}
 \end{split}
 \end{equation}
 The proof of (\ref{limit7}) is similar to the arguments as used in \cite{huang1,lions2}. Indeed, let~$\underline{g}$~denote a weak (weak$^{*}$) limit of any sequence $g^{\varepsilon}$, and $\mathcal{G}^{\varepsilon}$ be the effect viscous flux defined as
\begin{equation}\label{effectflux}
\begin{split}
\mathcal{G}^{\varepsilon}:=(\rho^{\varepsilon})^{\gamma}-\mu(\rho^{\varepsilon})(u^{\varepsilon})_{x}.
\end{split}
\end{equation}
It follows from (\ref{N1timew}) as $\varepsilon\rightarrow0$ that
\begin{align}
&\mathcal{G}^{\varepsilon}\overset{\ast}{\rightharpoonup} \mathcal{G}:=\underline{\rho^{\gamma}}-\underline{\mu(\rho)u_{x}}\quad\text{weakly* in}~ L^{p}(0,T;L^{\infty}(\mathbb{T})),\quad p\in (1,\frac{4}{3}).\label{limit6}
\end{align}
Multiplying $(\ref{m1})_{1}$ by $b'(\rho^{\varepsilon})$ for any $b(s)\in C^{1}(\mathbb{R})$, we have
 \begin{equation}\label{rho2varepsilon}
 \begin{split}
(b(\rho^{\varepsilon}))_{t}+(u^{\varepsilon} b(\rho^{\varepsilon}) )_{x}+\big{(}b'(\rho^{\varepsilon})\rho^{\varepsilon}-b(\rho^{\varepsilon})\big{)}(u^{\varepsilon})_{x}=0\quad\text{in}~\mathcal{D}'(\mathbb{T}\times (0,T)).
  \end{split}
 \end{equation}
Due to $(\ref{limit1})_{2}$ and the fact 
\begin{equation}\label{rhol2weak}
\begin{split}
&b(\rho^{\varepsilon}) \rightarrow \underline{b(\rho)}\quad\text{ in}~C([0,T];H^{-1}(\mathbb{T}))\quad\text{as}~\varepsilon\rightarrow0,
\end{split}
\end{equation}
derived from (\ref{basicCNSVFP}), (\ref{rho2varepsilon}), and Lemma \ref{lemma62}, we take the limit in $(\ref{rho2varepsilon})$ as $\varepsilon\rightarrow 0$ to get
 \begin{equation}\label{rho2p}
 \begin{split}
 &(\underline{b(\rho)})_{t}+(\underline{b(\rho)}u)_{x}+\underline{(b'(\rho)\rho-b(\rho))u_{x}}=0\quad\text{in}~\mathcal{D}'(\mathbb{T}\times (0,T)).
 \end{split}
 \end{equation}
By the equations (\ref{newm2}) and (\ref{rho2varepsilon}) for $(\rho^{\varepsilon},u^{\varepsilon},f^{\varepsilon})$, we can prove
 \begin{equation}\label{Gvarepsilon}
 \begin{split}
 &b(\rho^{\varepsilon})\mathcal{G}^{\varepsilon}=b(\rho^{\varepsilon})\big{(}(\rho^{\varepsilon})^{\gamma}-\mu(\rho^{\varepsilon})(u^{\varepsilon})_{x}\big{)}\\
 &\quad\quad\quad~=b(\rho^{\varepsilon})\Big{(}-\big{[}\mathcal{I}\big{(}\rho^{\varepsilon}u^{\varepsilon}+\rho^{\varepsilon}\mathcal{I}(n^{\varepsilon})\big{)}\big{]}_{t}-u^{\varepsilon}\big{(} \mathcal{I}\big{(}\rho^{\varepsilon}u^{\varepsilon}+\rho^{\varepsilon}\mathcal{I}(n^{\varepsilon})\big{)}\big{)}_{x}\\
 &\quad\quad\quad\quad~+\mathcal{I}(\rho^{\varepsilon})\int_{0}^{1}n^{\varepsilon}w^{\varepsilon}(y,t)dy+\int_{0}^{1}\big{[}\rho^{\varepsilon}|u^{\varepsilon}|^2+\rho^{\varepsilon}u^{\varepsilon}\mathcal{I}n^{\varepsilon}+(\rho^{\varepsilon})^{\gamma}-\mu(\rho^{\varepsilon})(u^{\varepsilon})_{x}\big{]}(y,t)dx\Big{)}\\
 &\quad\quad\quad~=-\big{[}(b(\rho^{\varepsilon})\mathcal{I}\big{(}\rho^{\varepsilon}u^{\varepsilon}+\rho^{\varepsilon}\mathcal{I}(n^{\varepsilon})\big{)}\big{]}_{t}-\big{[}(b(\rho^{\varepsilon})u^{\varepsilon}\mathcal{I}\big{(}\rho^{\varepsilon}u^{\varepsilon}+\rho^{\varepsilon}\mathcal{I}(n^{\varepsilon})\big{)}\big{]}_{x}\\
 &\quad\quad\quad\quad+\big{(}b(\rho^{\varepsilon})-b'(\rho^{\varepsilon})\rho^{\varepsilon}\big{)}(u^{\varepsilon})_{x}\mathcal{I}\big{(}\rho^{\varepsilon}u^{\varepsilon}+\rho^{\varepsilon}\mathcal{I}(n^{\varepsilon})\big{)}+b(\rho^{\varepsilon})\mathcal{I}(\rho^{\varepsilon})\int_{0}^{1}n^{\varepsilon}w^{\varepsilon}(y,t)dy\\
 &\quad\quad\quad\quad+b(\rho^{\varepsilon})\int_{0}^{1}\big{[}\rho^{\varepsilon}|u^{\varepsilon}|^2+\rho^{\varepsilon}u^{\varepsilon}\mathcal{I}(n^{\varepsilon})+(\rho^{\varepsilon})^{\gamma}-\mu(\rho^{\varepsilon})(u^{\varepsilon})_{x}\big{]}(y,t)dx.
 \end{split}
 \end{equation}
 Combining (\ref{j}), (\ref{limit1})-(\ref{strongnj}) and (\ref{effectflux})-(\ref{Gvarepsilon}) together, for any $b(s)\in C^1(\mathbb{R})$, we can prove the following property of effective viscous flux:
\begin{equation}
\begin{split}
&\underline{b(\rho)\mathcal{G}}=-\big{[}\underline{b(\rho)}\mathcal{I}\big{(}\rho u+\rho\mathcal{I}(n)\big{)}\big{]}_{t}-\big{[}\underline{b(\rho)}u\mathcal{I}\big{(}\rho u+\rho\mathcal{I}(n)\big{)}\big{]}_{x}\\
&\quad\quad\quad~+\underline{\big{(}b(\rho)-b'(\rho)\rho\big{)}u_{x}}\mathcal{I}\big{(}\rho u+\rho\mathcal{I}(n)\big{)}+\underline{b(\rho)}\mathcal{I}(\rho)\int_{0}^{1}nw(y,t)dy\\
&\quad\quad\quad~+\underline{b(\rho)}\int_{0}^{1}\big{[}\rho|u|^2+\rho u\mathcal{I}(n)+\underline{\rho^{\gamma}}-\underline{\mu(\rho)u_{x}}\big{]}(x,t)dx=\underline{b(\rho)}\mathcal{G},\quad\text{in}~\mathcal{D}(\mathbb{T}\times(0,T)).\label{limitG1}
\end{split}
\end{equation}
By $(\ref{limit1})_{1}$, (\ref{rhol2weak}) and the fact
\begin{equation}\label{gumu}
\begin{split}
(u^{\varepsilon})_{x}=-\mathcal{G}^{\varepsilon}\mu(\rho^{\varepsilon})^{-1}+(\rho^{\varepsilon})^{\gamma}\mu(\rho^{\varepsilon})^{-1},
\end{split}
\end{equation}
we have
\begin{equation}\label{ux1}
\begin{split}
&u_{x}=-\mathcal{G}\underline{\mu(\rho)^{-1}}+\underline{\rho^{\gamma}\mu(\rho)^{-1}}\quad \text{a.e. in}~\mathbb{T}\times(0,T).
\end{split}
\end{equation}
Since the equation $(\ref{rho2varepsilon})$ for $b(s)=s^2$ can be re-written as
\begin{equation}\label{gumu1}
\begin{split}
[(\rho^{\varepsilon})^{2}]_{t}+[(\rho^{\varepsilon})^{2}u^{\varepsilon}]_{x}=\mathcal{G}^{\varepsilon}(\rho^{\varepsilon})^2\mu(\rho^{\varepsilon})^{-1}-(\rho^{\varepsilon})^{\gamma+2}\mu(\rho^{\varepsilon})^{-1},
\end{split}
\end{equation}
we obtain by $(\ref{limit1})_{1}$, (\ref{limit6}), (\ref{rhol2weak}) and (\ref{limitG1})-(\ref{gumu1}) that
\begin{equation}\label{rho2eq}
\begin{split}
(\underline{\rho^2})_{t}+(\underline{\rho^2}u)_{x}=\mathcal{G}\underline{\rho^2\mu(\rho)^{-1}}-\underline{\rho^{\gamma+2}\mu(\rho)^{-1}}\quad \text{in}~\mathcal{D}'(\mathbb{T}\times(0,T)).
\end{split}
\end{equation}
We deduce from Lemma \ref{lemma69} that the equation $(\ref{m1})_{1}$ is satisfied in the sense of renormalized solutions, and therefore it holds by (\ref{ux1}) that
\begin{equation}\label{rrho2}
\begin{split}
&\rho\in C([0,T];L^{\gamma}(\mathbb{T})),\quad (\rho^2)_{t}+(\rho^2u)_{x}=-\rho^2 u_{x}=\mathcal{G} \rho^2\underline{\mu(\rho)^{-1}}-\rho^2\underline{\rho^{\gamma}\mu(\rho)^{-1}} \quad \text{in}~\mathcal{D}'(\mathbb{T}\times(0,T)).
\end{split}
\end{equation}
It follows from (\ref{rho2eq})-(\ref{rrho2}) for $\Psi:=\underline{\rho^{2}}-\rho^2$ that
\begin{equation}
\begin{split}
&\Psi_{t}+(\Psi u)_{x}=\mathcal{G}(\underline{\rho^2\mu(\rho)^{-1}}-\rho^2\mu(\rho)^{-1})+\mathcal{G}\rho^2(\mu(\rho)^{-1}-\underline{\mu(\rho)^{-1}})\\
&\quad\quad\quad\quad\quad\quad-(\underline{\rho^{\gamma+2}\mu(\rho)^{-1}}-\rho^{\gamma+2}\mu(\rho)^{-1})-\rho^2(\rho^{\gamma}\mu(\rho)^{-1}-\underline{\rho^{\gamma}\mu(\rho)^{-1}})\quad\text{in}~\mathcal{D}'(\mathbb{T}\times(0,T)),\label{limitrhoerror}
\end{split}
\end{equation}
with the initial data $\Psi|_{t=0}=0$ a.e. in~$\mathbb{T}$. For any $p\geq 0$ and $B(s)\in C^{1}([0,\rho_{+}])\cap C^{2}((0,\rho_{+}])$ satisfying
$$
s^{p}B'(s)\in C([0,\rho_{+}]),\quad C_{p,B}:=\sup_{0\leq \rho,s\leq \rho_{+}}\Big{|}\rho^{p}\int_{0}^{1}\int_{0}^{1}B''(\rho+\theta_{1}\theta_{2}(s-\rho))\theta_{1}d\theta_{2}d\theta_{1}\Big{|}<\infty,
$$
we have
\begin{equation}\label{tayor}
\begin{split}
&\rho^{p}(B(\rho^{\varepsilon})-B(\rho))=\rho^{p}B'(\rho)(\rho^{\varepsilon}-\rho)+\rho^{p}\int_{0}^{1}\int_{0}^{1}B''(\rho+\theta_{1}\theta_{2}(\rho^{\varepsilon}-\rho))\theta_{1}d\theta_{2}d\theta_{1}|\rho^{\varepsilon}-\rho|^2,
\end{split}
\end{equation}
from which we deduce for any $H\in L^1(0,T;L^{\infty}(\mathbb{T}))$ and a.e. $t\in (0,T)$ that
\begin{equation}\nonumber
\begin{split}
&\int_{\mathbb{T}} |H\rho^{p}(\underline{B(\rho)}-B(\rho))|(x,t)dx\leq C_{p,B}\|H(t)\|_{L^{\infty}(\mathbb{T})}\underset{\varepsilon\rightarrow 0}{\sup\lim}\int_{\mathbb{T}}|\rho^{\varepsilon}-\rho|^2(x,t)dx.
\end{split}
\end{equation}
Thus, one gets for $B_{1}(s)=s^2\mu(s)^{-1}, B_{2}(s)=\mu(s)^{-1}, B_{3}(s)=s^{\gamma+2}\mu(s)^{-1}$, and $B_{4}(s)=s^{\gamma}\mu(s)^{-1}$ that
\begin{equation}\label{tayor1}
\left\{
\begin{split}
&\int_{\mathbb{T}}|\mathcal{G}(\underline{\rho^2\mu(\rho)^{-1}}-\rho^2\mu(\rho)^{-1})|(x,t)dx\leq C_{0,B_{1}}\|\mathcal{G}(t)\|_{L^{\infty}(\mathbb{T})}\underset{\varepsilon\rightarrow 0}{\sup\lim}\int_{\mathbb{T}}|\rho^{\varepsilon}-\rho|^2(x,t)dx,\\
&\int_{\mathbb{T}}|\mathcal{G}\rho^2(\mu(\rho)^{-1}-\underline{\mu(\rho)^{-1}})|(x,t)dx\leq C_{2,B_{2}}\|\mathcal{G}(t)\|_{L^{\infty}(\mathbb{T})}\underset{\varepsilon\rightarrow 0}{\sup\lim}\int_{\mathbb{T}}|\rho^{\varepsilon}-\rho|^2(x,t)dx,\\
&\int_{\mathbb{T}}|\underline{\rho^{\gamma+2}\mu(\rho)^{-1}}-\rho^{\gamma+2}\mu(\rho)^{-1}|(x,t)dx\leq C_{0,B_{3}}\underset{\varepsilon\rightarrow 0}{\sup\lim}\int_{\mathbb{T}}|\rho^{\varepsilon}-\rho|^2(x,t)dx,\\
&\int_{\mathbb{T}}\rho^2|\rho^{\gamma}\mu(\rho)^{-1}-\underline{\rho^{\gamma}\mu(\rho)^{-1}}|(x,t)dx\leq C_{2,B_{4}}\underset{\varepsilon\rightarrow 0}{\sup\lim}\int_{\mathbb{T}}|\rho^{\varepsilon}-\rho|^2(x,t)dx.
\end{split}
\right.
\end{equation}
In particular, we choose $B(s)=s^2$ and $p=0$ in (\ref{tayor}) to have
\begin{equation}\label{psigeq0}
\begin{split}
&\underset{\varepsilon\rightarrow 0}{\sup\lim}\int_{\mathbb{T}}|\rho^{\varepsilon}-\rho|^2(x,t)dx\leq\int_{\mathbb{T}}\Psi(x,t) dx.
\end{split}
\end{equation}
Integrating (\ref{limitrhoerror}) over $\mathbb{T}\times[0,t]$ and substituting (\ref{tayor1})-(\ref{psigeq0}) into the resulting inequality, we obtain
\begin{equation}\nonumber
\begin{split}
\underset{t\in[0, T]}{{\rm{ess~sup}}}\int_{\mathbb{T}}|\rho^{\varepsilon}-\rho|^2(x,t)dx\leq C\int_{0}^{T}\big{(}1+\|\mathcal{G}(t)\|_{L^{\infty}(\mathbb{T})}\big{)}\int_{\mathbb{T}}|\rho^{\varepsilon}-\rho|^2(x,t)dxdt,
\end{split}
\end{equation}
which together with $(\ref{basicCNSVFP})_{5}$, (\ref{limit6}), and the Gr${\rm{\ddot{o}}}$nwall inequality gives rise to (\ref{limit7}).

Furthermore, owing to Lemma \ref{lemma23}, $(\ref{limit1})_{1}$, $(\ref{limit7})$, and integration by parts, we obtain
\begin{equation}\label{fL2}
\begin{split}
&\sqrt{\rho^{\varepsilon}}(f^{\varepsilon})_{v}\rightharpoonup \sqrt{\rho} f_{v}\quad  \text{in}~L^2(0,T;L^2(\mathbb{T}\times\mathbb{R})),\quad \text{as}~\varepsilon\rightarrow 0.
\end{split}
\end{equation}
By virtue of (\ref{limit1})-(\ref{limit7}) and $(\ref{fL2})$, one can show that the limit $(\rho,u,f)$ indeed satisfies the equations $(\ref{m1})$ in the sense of distributions.

\emph{Step~3: The properties $(\ref{r0})$-$(\ref{entropyinequality})$ and $(\ref{r1})$-$(\ref{b1})$.} For any nonnegative function $\phi\in \mathcal{D}(0,T)$ and constant $R>0$, one deduces by  $(\ref{basicCNSVFP})_{2}$ and $(\ref{limit1})_{3}$ that
\begin{equation}\nonumber
\begin{split}
&\int_{0}^{T}\phi(t)\int_{\mathbb{T}\times\mathbb{R}} \langle v\rangle^3\mathrm{1}_{|v|\leq R}f(x,v,t)dvdxdt\leq \underset{\varepsilon\rightarrow 0}{\lim\sup}\int_{0}^{T}\phi(t)\int_{\mathbb{T}\times\mathbb{R}} \langle v\rangle^3f^{\varepsilon}(x,v,t)dvdxdt.
\end{split}
\end{equation}
We take the limit as $R\rightarrow \infty$ and apply the monotone convergence theorem to get
\begin{equation}
\begin{split}
f\in L^{\infty}(0,T;L^1_{3}(\mathbb{T}\times\mathbb{R})).\label{limit2}
\end{split}
\end{equation}
Therefore, it follows by Lemma \ref{lemma22}-\ref{lemma23}, (\ref{limit1}), and (\ref{limit2}) that $(\rho,u,f)$ satisfies (\ref{r1}). By $(\ref{m1})_{3}$ and (\ref{basicCNSVFP}), it holds for any $p\in(1,3]$ that
\begin{equation}\nonumber
\begin{split}
&\sup_{\varepsilon>0}\int_{0}^{T}\|\langle v\rangle ^{\frac{3}{p}-1}(f^{\varepsilon})_{t}(t)\|_{W^{-2,p}(\mathbb{T}\times\mathbb{R})}^2dt\\
&\quad=\sup_{\varepsilon>0}\int_{0}^{T}\Big{|}\sup_{\|\varphi\|_{W^{2,\frac{p}{p-1}}(\mathbb{T}\times\mathbb{R})}\leq 1}\int_{\mathbb{T}\times\mathbb{R}}\big{[}\langle v\rangle^{\frac{3}{p}-1}vf^{\varepsilon}\varphi_{x}+\rho^{\varepsilon}(u^{\varepsilon}-v)f^{\varepsilon}(\langle v\rangle^{\frac{3}{p}-1}\varphi)_{v}\\
&\quad\quad+\rho^{\varepsilon} f^{\varepsilon}(\langle v\rangle^{\frac{3}{p}-1}\varphi)_{vv}\big{]}(x,v,t)dvdx\Big{|}^2dt\\
&\quad\leq C\sup_{\varepsilon>0}\Big{(} \big{(}1+\|u^{\varepsilon}\|_{L^2(0,T;L^{\infty}(\mathbb{T}))}\big{)} \|f^{\varepsilon}\|_{L^{\infty}(0,T;L^1_{3}(\mathbb{T}\times\mathbb{R}))}^{\frac{1}{p}}\|f^{\varepsilon}\|_{L^{\infty}(0,T;L^{\infty}(\mathbb{T}\times\mathbb{R}))}^{1-\frac{1}{p}}\Big{)}^2<\infty,
\end{split}
\end{equation}
which together with Lemma C.1 in \cite{lions1} yields for any $p\in(1,3)$ and $q\in[0,\frac{3}{p}-1]$ that
\begin{equation}\label{fweaklp}
\begin{split}
&\langle v\rangle^{q}f^{\varepsilon}\rightarrow \langle v\rangle^{q}f\quad\text{ in }~C([0,T];L^{p}_{weak}(\mathbb{T}\times\mathbb{R})),\quad \text{as}~\varepsilon\rightarrow 0.
\end{split}
\end{equation}
Since we have $1\in L^2(\mathbb{T})$ and $\langle v\rangle^{-\frac{2}{5}}\in L^{5}(\mathbb{T}\times\mathbb{R})$, one deduces by $(\ref{limit1})_{3}$ and $(\ref{fweaklp})$ for $(p,q)=(\frac{5}{4},\frac{2}{5})$ that
\begin{equation}\label{limit10}
\left\{
\begin{split}
&\lim_{\varepsilon\rightarrow0}\int_{\mathbb{T}}\rho^{\varepsilon}(x,t)dx=\int_{\mathbb{T}}\rho(x,t)dx,\quad \forall t\in [0,T],\\
&\lim_{\varepsilon\rightarrow0}\int_{\mathbb{T}\times\mathbb{R}}\langle v\rangle^{-\frac{2}{5}}\langle v\rangle^{\frac{2}{5}}f^{\varepsilon}(x,v,t)dvdx=\int_{\mathbb{T}\times\mathbb{R}}f(x,v,t)dvdx,\quad \forall t\in [0,T].
\end{split}
\right.
\end{equation}
Similarly, it follows from $v\langle v\rangle^{-\frac{7}{5}}\in L^{5}(\mathbb{T}\times\mathbb{R})$, $(\ref{limit1})_{4}$, and $(\ref{fweaklp})$ for $(p,q)=(\frac{5}{4},\frac{7}{5})$ that
\begin{equation}\label{limit11}
\left\{
\begin{split}
&\lim_{\varepsilon\rightarrow0}\int_{\mathbb{T}}\rho^{\varepsilon}u^{\varepsilon}(x,t)dx=\int_{\mathbb{T}}\rho u(x,t)dx,\quad \forall t\in[0,T],\\
&\lim_{\varepsilon\rightarrow0}\int_{\mathbb{T}\times\mathbb{R}}v\langle v\rangle^{-\frac{7}{5}}\langle v\rangle^{\frac{7}{5}}f^{\varepsilon}(x,v,t)dvdx=\int_{\mathbb{T}\times\mathbb{R}}vf(x,v,t)dvdx,\quad \forall t\in [0,T].
\end{split}
\right.
\end{equation}
According to $(\ref{mass})$-$(\ref{momentum})$ and (\ref{limit10})-(\ref{limit11}), the conservation laws (\ref{massmomentum}) hold for any $t\in[0,T]$. We conclude from (\ref{limit1})-(\ref{limit7}), (\ref{fL2}), (\ref{fweaklp}), and the lower semi-continuity of the limit $(\rho,u,f)$ that the entropy inequality $(\ref{entropyinequality})$ holds for a.e. $t\in [0,T]$. By virtue of Proposition \ref{prop41}, one has the long time behavior $(\ref{b1})$. The proof of Theorem \ref{theorem11} is completed.

\begin{remark} Under the assumptions of Theorem \ref{theorem11}, if it further holds $\inf_{x\in\mathbb{T}}\rho_{0}(x)>0$, then we obtain $\|f_{v}\|_{L^2(0,T;L^2)}\leq C_{T}$ due to $(\ref{N1time})_{2}$, so we can apply Lemma \ref{lemma66} as in $(\ref{vartheta})$-$(\ref{strongf})$ to prove
\begin{equation}\nonumber
\begin{split}
&\sqrt{f^{\varepsilon}}\rightarrow \sqrt{f}\quad\text{in}~L^{2}(0,T;L^2(\mathbb{T}\times\mathbb{R})),\quad \text{as}~\varepsilon\rightarrow 0,
\end{split}
\end{equation}
which implies
\begin{equation}\nonumber
\begin{split}
&\sqrt{\rho^{\varepsilon}} (\sqrt{f^{\varepsilon}})_{v}\rightharpoonup  \sqrt{\rho} (\sqrt{f})_{v}\quad\text{in}~L^{2}(0,T;L^2(\mathbb{T}\times\mathbb{R})),\quad \text{as}~\varepsilon\rightarrow 0.
\end{split}
\end{equation}
Thus, one can show
\begin{equation}\nonumber
\begin{split}
&\int_{0}^{t}\int_{\mathbb{T}\times\mathbb{R}}(\rho|(u-v)\sqrt{f}-2(\sqrt{f})_{v}|^2)(x,v,\tau)dvdxd\tau\\
&\quad\leq \lim\sup_{\varepsilon\rightarrow0}\int_{0}^{t}\int_{\mathbb{T}\times\mathbb{R}}(\rho^{\varepsilon}|(u^{\varepsilon}-v)\sqrt{f^{\varepsilon}}-2(\sqrt{f^{\varepsilon}})_{v}|^2)(x,v,\tau)dvdxd\tau,
\end{split}
\end{equation}
so that the entropy inequality $(\ref{entropyinequality0})$ holds.
\end{remark}

~\

\underline{\it\textbf{Proof of Theorem \ref{theorem12}:}}~ Let $(\rho_{0},\frac{m_{0}}{\rho_{0}},f_{0})$ satisfy $(\ref{a2})$, and $(\rho_{0}^{\varepsilon},u_{0}^{\varepsilon},f_{0}^{\varepsilon})$ be given through (\ref{weakdata}) for $\varepsilon\in(0,1)$. Similarly, we can obtain an approximate sequence $(\rho^{\varepsilon},u^{\varepsilon},f^{\varepsilon})$ and show its convergence to a global weak solution $(\rho,u,f)$ to the IVP (\ref{m1})-(\ref{kappa}) in the sense of Definition \ref{defn11} as $\varepsilon\rightarrow0$. Thus, it follows from the a-priori estimates established in Lemmas \ref{lemma25}-\ref{lemma26} that $(\rho, u, f)$ satisfies the further properties (\ref{r2}). In addition, this weak solution $(\rho,u,f)$ is unique due to Proposition \ref{prop31}, and the time convergence $(\ref{b2})$ can be derived by Propositions \ref{prop42}-\ref{prop43}. The proof of Theorem \ref{theorem12} is completed.

\section{Appendix}

First, the following Zlonik's inequality is used to get the upper of the fluid density uniformly in time:

\begin{lemma}[\!\!\cite{zlotnik1}]\label{lemma61}
Let $b\in W^{1,1}(0,T)$ and $g\in C(\mathbb{R})$. Assume that the function $y\in W^{1,1}(0,T)$ solves the ODE system
$$
y'(t)=g(y)+b'(t)\quad\text{on}~[0,T],\quad y(0)=y_{0}.
$$
If $g(\infty)=-\infty$ and 
$$
b(t_{2})-b(t_{1})\leq N_{0}+N_{1}(t_{2}-t_{1})
$$
for any $0\leq t_{1}<t_{2}\leq T$ with some constants $N_{0}\geq 0$ and $N_{1}\geq 0$, then we have
$$
y(t)\leq \max\{y_{0},\xi_{*}\}+N_{0}<\infty\quad\text{on}~[0,T],
$$
where $\xi_{*}\in\mathbb{R}$ is a constant such that it holds
$$
g(\xi)\leq -N_{1}\quad\forall ~\xi\geq \xi_{*}.
$$
\end{lemma}

The spaces of weak topology on $X$ is denoted by
\begin{equation}\label{weaktimespace}
\begin{split}
C([0,T];X_{weak}):=\big{\{} g:[0,T]\rightarrow X~\big{|}~<g,\phi>_{X,X^{*}}\in C([0,T])~\text{for any $\phi\in X^{*}$}\big{\}},
\end{split}
\end{equation}
and $g^{n}\rightarrow g$ in $C([0,T];X_{weak})$ as $n\rightarrow \infty$ if and only if $<g^{n}(t),\varphi>_{X,X^{*}}\rightarrow <g(t),\varphi>_{X,X^{*}}$ uniformly in $t\in[0,T]$ for any $\phi\in X^{*}$. 
We have
\begin{lemma}[\!\!\cite{lions2}, Lemma C.1]\label{lemma62}
Let $T>0$, $X$ be a separable reflexive Banach space, and $Y$ be a Banach space such that $X$ is continuously embedded in $Y$ and $Y^{*}$ is separable and dense in $X^{*}$. If $u^{n}$ is uniformly bounded in $L^{\infty}(0,T;X)$ and $(u^{n})_{t}$ is uniformly bounded in $L^1(0,T;Y)$, then $u^{n}$ is strongly compact in $C([0,T];X_{weak})$.
\end{lemma}

 To get the strong convergence of the moments $n$ and $nw$, we need the following lemma about the average compactness for kinetic transport equations:
\begin{lemma}[\!\!\cite{karper1}]\label{lemma63}
Let $T>0, p\in (1,\frac{3}{2})$, and $G^{n}$ be uniformly bounded in $L^{p}(0,T;L^{p}(\mathbb{T}\times\mathbb{R}))$. If the sequence $f^{n}$ is uniformly bounded in $L^{\infty}(0,T; L^1_{2}\cap L^{\infty}(\mathbb{T}\times\mathbb{R}))$ and satisfies
$$
(f^{n})_{t}+v(f^{n})_{x}=(G^{n})_{v}\quad\text{in}~\mathcal{D}'( \mathbb{T}\times \mathbb{R}\times (0,T)),
$$
then for any $\chi(v)\in\mathcal{D}(\mathbb{R})$, the sequence $\int_{\mathbb{R}} f^{n}(x,v,t)\chi(v)dv$ is strongly compact in $L^p(0,T;L^{p}(\mathbb{T}))$.
\end{lemma}

By virtue of Dunford-Pettis criterion, we have the following lemma:

\begin{lemma}[\!\!\cite{arkeryd1}]\label{lemma64}
Let $d\geq1$, and $f^{n}$ be a sequence of functions in $L^1(\mathbb{R}^{d})$ satisfying
$$
\sup_{n\geq0}\int_{\mathbb{R}^{d}}\big{(}|f^{n}|+|x|^2|f^{n}|+|f^{n}\log{f^{n}}|\big{)}(x)dx<\infty.
$$
Then $f^{n}$ is weakly compact in $L^1(\mathbb{R}^{d})$.
\end{lemma}

In addition, we need

\begin{lemma}[\!\!\cite{brezis1}, pp. 468]\label{lemma65}
Let $d\geq1$, and $f^{n}$ be a sequence of functions in $L^1(\mathbb{R}^{d})$ satisfying as $n\rightarrow\infty$ that
\begin{equation}\nonumber
\left\{
\begin{split}
&f^{n}\rightharpoonup f \quad \text{in }~L^1(\mathbb{R}^{d}),\\
&f^{n}\rightarrow f\quad \text{a.e. in}~\mathbb{R}^{d},
\end{split}
\right.
\end{equation}
 then we have
$$
f^{n}\rightarrow f\quad\text{ in } ~L^1(\mathbb{R}^{d}),\quad\text{as}~n\rightarrow\infty.
$$
\end{lemma}

We need the following lemma, which implies the strong compactness in all variables of the distribution function for the Vlasov-Fokker-Planck equation $(\ref{m1})_{3}$ provided that this sequence is strongly compact with respect to velocity variable.
\begin{lemma}[\!\!\cite{arsenio1}]\label{lemma66}
Let $d\geq 1, 1<p<\infty,~\alpha_{1}\geq 0,~0\leq \alpha_{2}<1$, and the nonnegative sequence $f^{n}\in L^{p}(\mathbb{R}^{d}\times\mathbb{R}^{d}\times\mathbb{R})$ uniformly in $n\geq0$, be locally strongly compact with respect to $v$ in $L^{p}(\mathbb{R}^{d}\times\mathbb{R}^{d}\times\mathbb{R})$, and satisfy
\begin{equation}\nonumber
\begin{split}
(f^{n})_{t}+v\cdot\nabla_{x}f^{n}=(I-\Delta_{t,x})^{\frac{\alpha_{2}}{2}}(I-\Delta_{v})^{\frac{\alpha_{1}}{2}}g^{n}\quad\text{in}~\mathcal{D}'(\mathbb{R}^{d}\times \mathbb{R}^{d}\times\mathbb{R}),
\end{split}
\end{equation}
for $g^{n}\in L^{p}(\mathbb{R}^{d}\times\mathbb{R}^{d}\times\mathbb{R})$ uniformly in $n\geq 0$. Then, the sequence $f^{n}$ is locally strongly compact in $L^{p}(\mathbb{R}^{d}\times\mathbb{R}^{d}\times\mathbb{R})$ (in all variables).
\end{lemma}

Since Lemma \ref{lemma66} involves fractional derivatives, we need the following two lemmas:

\begin{lemma}[\!\!\cite{grafakos1}, pp.~13] \label{lemma67}
Let $d\geq1, p\in(1,\infty)$, $\Omega$ be any open set in $\mathbb{R}^{d} $, and $L^{p,\infty}(\Omega)$ be the weak $L^{p}(\Omega)$ space defined as the set of all measurable function $f$ satisfying
$$
\|f\|_{L^{p,\infty}(\Omega)}:=\sup_{\lambda>0}\lambda |\{x\in\Omega~\big{|}~|f(x)|>\lambda\}|<\infty.
$$
If $|\Omega|<\infty$, then we have
$$
\int_{\Omega}|f(x)|^{q}dx\leq \frac{p}{p-q}|\Omega|^{1-\frac{q}{p}}\|f\|_{L^{p,\infty}(\Omega)}^{q},\quad  \forall q\in[1,p).
$$
\end{lemma}

\begin{lemma}[\!\!\cite{grafakos2}, pp.~8]\label{lemma68}
Let $d\geq1, \alpha>0$, and $(I-\Delta)^{-\frac{\alpha}{2}}$ be the Bessel potential of order $\alpha$. Then, we have

(1) $(I-\Delta)^{-\frac{\alpha}{2}}$ maps $L^{p}(\mathbb{R}^{d})$ to itself with norm 1 for any $p\in[1,\infty]$.

(2) Let $\alpha\in(0,d)$ and $q\in(1,\infty)$ satisfy $1-\frac{1}{q}=\frac{\alpha}{d}$. Then there exist a constant $C_{q,d,\alpha}>0$ such that for any $f\in L^{p}(\mathbb{R}^{d})$, we have
\begin{equation}\nonumber
\begin{split}
\|(I-\Delta)^{-\frac{\alpha}{2}}f\|_{L^{q,\infty}(\mathbb{R}^{d})}\leq C_{q,d,\alpha}\|f\|_{L^{1}(\mathbb{R}^{d})}.
\end{split}
\end{equation}
\end{lemma}

\vspace{2ex}

The following lemma about the renormalized solutions for continuity equations is due to Diperna-Lions:
\begin{lemma}[\!\!\cite{diperna2,lions1}]\label{lemma69}
For any $T>0$, assume
$$
\rho \in L^{\infty}(0,T;L^{\infty}(\mathbb{T})), \quad u\in L^2(0,T;H^{1}(\mathbb{T})), \quad g\in L^1(0,T;L^1(\mathbb{T})).
 $$
If $\rho$ solves the equation
\begin{equation}\label{ltransport}
\begin{split}
\rho_{t}+(\rho u)_{x}=g\quad \text{in}~\mathcal{D}'(\mathbb{T}\times(0,T)),
\end{split}
\end{equation}
 then $\rho$ is a renormalized solution to $(\ref{ltransport})$; i.e., for any $b(s)\in C^{1}(\mathbb{R})$, it holds
\begin{equation}\label{610}
\begin{split}
(b(\rho))_{t}+(b(\rho)u)_{x}+(b'(\rho)\rho-b(\rho))u_{x}=g b'(\rho)\quad\text{in}~\mathcal{D}'(\mathbb{T}\times(0,T)),
\end{split}
\end{equation}
and we have $\rho\in C([0,T];L^{p}(\mathbb{T}))$ for any $p\in[1,\infty)$.
\end{lemma}

Finally, as in \cite{bris1}, we show the following result about the renormalized solutions for Fokker-Planck equations with variable coefficients:
\begin{lemma}\label{lemma610}
For any $T>0$, assume
\begin{equation}\nonumber
\left\{
\begin{split}
&f\in L^{\infty}(0,T;L^1_{2}\cap L^{\infty}(\mathbb{T}\times\mathbb{R})),\quad  f_{v}\in L^2(0,T;L^2(\mathbb{T}\times\mathbb{R})),\\
 & \rho\in L^{\infty}(0,T;L^{\infty}(\mathbb{T})), \quad u\in L^2(0,T;L^{\infty}(\mathbb{T})), \quad G\in L^{1}(0,T;L^1(\mathbb{T}\times\mathbb{R})).
 \end{split}
 \right.
 \end{equation}
 If $f$ solves the equation 
\begin{equation}\label{lfp}
\begin{split}
f_{t}+vf_{x}+(\rho(u-v)f)_{v}-\rho f_{vv}=G\quad\text{in}~\mathcal{D}'(\mathbb{T}\times\mathbb{R}\times(0,T)),
 \end{split}
 \end{equation}
then $f$ is a renormalized solution to $(\ref{lfp})$; i.e., for any $b(s)\in C^2(\mathbb{R})$, it holds
\begin{equation}
\begin{split}
&(b(f))_{t}+v(b(f))_{x}+(\rho(u-v)b(f))_{v}-(\rho b(f))_{vv}\\
&\quad\quad+\rho (b(f)-b'(f)f) +b''(f)|f_{v}|^2= b'(f)G\quad \text{in}~\mathcal{D}'(\mathbb{T}\times\mathbb{R}\times(0,T)).\label{lfpd}
\end{split}
\end{equation}
\end{lemma}
\begin{proof}
Denote
$$
f^{\varepsilon}:=J^{\varepsilon}\ast f,\quad J^{\varepsilon}(x,v):=J_{1}^{\varepsilon}(x)J_{2}^{\varepsilon}(v),
$$
where $J_{1}^{\varepsilon}(x)$ and $J_{2}^{\varepsilon}(v)$ are the Friedrichs mollifier for $\varepsilon>0$ with respect to the variables $x$ and $v$ respectively. Applying $J^{\varepsilon}\ast$ to the equation (\ref{lfp}), we obtain
\begin{equation}\label{fapp}
\begin{split}
&(f^{\varepsilon})_{t}+v(f^{\varepsilon})_{x}+(\rho(u-v)f^{\varepsilon})_{v}-\rho (f^{\varepsilon})_{vv}-G\ast J^{\varepsilon}\\
&\quad= v(f^{\varepsilon})_{x}-(vf)\ast J^{\varepsilon}+[(\rho(u-v)f^{\varepsilon}-(\rho(u-v)f)\ast  J^{\varepsilon}]_{v}+[(\rho f_{v})\ast J^{\varepsilon}-\rho (f^{\varepsilon})_{v}]_{v}.
\end{split}
\end{equation}
Multiplying (\ref{fapp}) by $b'(f^{\varepsilon})$ and using the fact $
-b'(f^{\varepsilon})(f^{\varepsilon})_{vv}=-b(f^{\varepsilon})_{vv}+b''(f^{\varepsilon})|(f^{\varepsilon})_{v}|^2$, we have
\begin{equation}\label{lfpp}
\begin{split}
&b(f^{\varepsilon})_{t}+vb(f^{\varepsilon})_{x}+(\rho(u-v)b(f^{\varepsilon}))_{v}-\rho b(f^{\varepsilon})_{vv}-b'(f^{\varepsilon})G\ast J^{\varepsilon}\\
&\quad\quad-\rho (b'(f^{\varepsilon})f^{\varepsilon}-b(f^{\varepsilon}))+\rho b''(f^{\varepsilon})|(f^{\varepsilon})_{v}|^2=\mathcal{R}_{1}^{\varepsilon}+(\mathcal{R}_{2}^{\varepsilon})_{v}+\mathcal{R}_{3}^{\varepsilon}+(\mathcal{R}^{\varepsilon}_{4})_{v}+\mathcal{R}^{\varepsilon}_{5},
\end{split}
\end{equation}
where $\mathcal{R}_{i}^{\varepsilon}$, $i=1,...,5$, are given as
\begin{equation}\nonumber
\begin{split}
&\mathcal{R}_{1}^{\varepsilon}:=b'(f^{\varepsilon})[v(f^{\varepsilon})_{x}-(vf_{x})\ast J^{\varepsilon}],\\
&\mathcal{R}_{2}^{\varepsilon}:=b'(f^{\varepsilon})[\rho(u-v)f^{\varepsilon}-(\rho(u-v)f)\ast  J^{\varepsilon}],\\
&\mathcal{R}_{3}^{\varepsilon}:=-b''(f^{\varepsilon})(f^{\varepsilon})_{v}[(\rho(u-v)f)\ast  J^{\varepsilon}-\rho(u-v)f^{\varepsilon}],\\
&\mathcal{R}_{4}^{\varepsilon}:=b'(f^{\varepsilon})[(\rho f_{v})\ast J^{\varepsilon}-\rho (f^{\varepsilon})_{v}],\\
&\mathcal{R}_{5}^{\varepsilon}:=-b''(f^{\varepsilon})(f^{\varepsilon})_{v}[(\rho f_{v})\ast J^{\varepsilon}-\rho (f^{\varepsilon})_{v}].
\end{split}
\end{equation}
Since it holds as $\varepsilon\rightarrow 0$ that
\begin{equation}\nonumber
\left\{
\begin{split}
&f^{\varepsilon}\rightarrow f \quad\text{in $L^{\infty}(0,T;L^1_{2}\cap L^2_{1}(\mathbb{T}\times\mathbb{R}))$},\\
&(f^{\varepsilon})_{v}\rightarrow f_{v} \quad\text{in $L^{2}(0,T;L^{2}(\mathbb{T}\times\mathbb{R}))$},\\
&G\ast J^{\varepsilon}\rightarrow G \quad\text{in $L^{1}(0,T;L^1(\mathbb{T}\times\mathbb{R}))$},
\end{split}
\right.
\end{equation}
one can conclude (\ref{lfpd}) provided that the right-hand side of (\ref{lfpp}) tends to $0$ as $\varepsilon\rightarrow 0$ in the sense of distributions. To show it, we first have
\begin{equation}\nonumber
\begin{split}
&|R_{1}^{\varepsilon}|=\Big{|}b'(f^{\varepsilon})\int_{\mathbb{T}\times\mathbb{R}} ( J^{\varepsilon})_{x}(x-x_{1},x-v_{1})(v-v_{1})f(x_{1},v_{1},t)dv_{1}dx_{1}\Big{|}\\
&\quad~~=\Big{|}b'(f^{\varepsilon})\int_{\{|x-x_{1}|\leq \varepsilon,|v-v_{1}|\leq \varepsilon\}} \frac{v-v_{1}}{\varepsilon} J_{x_{1}}^{\varepsilon}(x-x_{1},v-v_{1})(f(x_{1},v_{1},t)-f(x,v,t))dv_{1}dx_{1}\Big{|}\\
&\quad~~\leq C\sup_{h\in (0,\varepsilon]}\int_{\mathbb{T}\times\mathbb{R}}|J_{x_{1}}^{\varepsilon}(x-x_{1},v-v_{1})|| f(x_{1},v_{1},t)-f(x_{1}+h,v_{1}+h,t)|dv_{1}dx_{1},
\end{split}
\end{equation}
which together with the Young inequality of convolution type gives rise to
\begin{equation}\nonumber
\begin{split}
&\|R_{1}^{\varepsilon}\|_{L^{1}(0,T;L^1(\mathbb{T}\times\mathbb{R}))}\leq C\sup _{h\in (0,\varepsilon]}\|f(\cdot+h,\cdot+h,\cdot)-f(\cdot,\cdot,\cdot)\|_{L^{1}(0,T;L^1(\mathbb{T}\times\mathbb{R}))}\rightarrow  0,\quad \text{as}~\varepsilon\rightarrow0.
\end{split}
\end{equation}
Due to $f\in L^{\infty}(0,T;L^2_{1}(\mathbb{T}\times\mathbb{R}))$ and $\rho(u-v)f\in L^2(0,T;L^2(\mathbb{T}\times\mathbb{R}))$, it holds as $\varepsilon\rightarrow 0$ that
\begin{equation}\nonumber
\left\{
\begin{split}
&\|R_{2}^{\varepsilon}\|_{L^{2}(0,T;L^2(\mathbb{T}\times\mathbb{R}))}\\
&\quad\leq \|\rho(u-v)f-(\rho(u-v)f)\ast J^{\varepsilon}\|_{L^2(0,T;L^2(\mathbb{T}\times\mathbb{R}))}\\
&\quad\quad+C\big{(}1+\|u\|_{L^2(0,T;L^{\infty}(\mathbb{T}))}\big{)}\|f^{\varepsilon}-f\|_{L^2(0,T;L^2_{1}(\mathbb{T}\times\mathbb{R}))}\rightarrow  0,\\
&\|R_{3}^{\varepsilon}\|_{L^{1}(0,T;L^1(\mathbb{T}\times\mathbb{R}))}\\
&\quad\leq C\|(f^{\varepsilon})_{v}\|_{L^2(0,T;L^2(\mathbb{T}\times\mathbb{R}))}\|\rho(u-v)f-(\rho(u-v)f)\ast J^{\varepsilon}\|_{L^2(0,T;L^2(\mathbb{T}\times\mathbb{R}))}\\
&\quad\quad+C\big{(}1+\|u\|_{L^2(0,T;L^{\infty}(\mathbb{T}))}\big{)}\|(f^{\varepsilon})_{v}\|_{L^2(0,T;L^2(\mathbb{T}\times\mathbb{R}))}\|f^{\varepsilon}-f\|_{L^{\infty}(0,T;L^2_{1}(\mathbb{T}\times\mathbb{R}))}\rightarrow  0.
\end{split}
\right.
\end{equation}
Similarly, it is easy to prove
\begin{equation}\nonumber
\begin{split}
&\|R_{4}^{\varepsilon}\|_{L^2(0,T;L^2(\mathbb{T}\times\mathbb{R}))}+\|R_{5}^{\varepsilon}\|_{L^1(0,T;L^1(\mathbb{T}\times\mathbb{R}))}\rightarrow  0,\quad\text{as}~\varepsilon\rightarrow 0.
\end{split}
\end{equation}
The proof of Lemma \ref{lemma610} is completed.
\end{proof}

\textbf{Acknowledgments.} The authors are grateful to the referees for the
valuable comments and the helpful suggestions on the manuscript.

The research of the paper is supported by National Natural Science Foundation of China (No.11931010, 11671384, and 11871047) and by the key research project of Academy for Multidisciplinary Studies, Capital Normal University, and by the Capacity Building for Sci-Tech Innovation-Fundamental Scientific Research Funds (No.007/20530290068).



\end{document}